\documentclass[11pt]{amsart}
\usepackage{amssymb}
\usepackage{amsfonts}
\usepackage{amsmath}
\usepackage{graphicx}
\usepackage{xcolor}
\usepackage{mathrsfs}
\usepackage{stmaryrd}
\usepackage{epsfig,color}
\usepackage{blindtext}
\usepackage{enumerate}
\usepackage{hyperref}
\usepackage{url}
\usepackage{bbm}
\usepackage{filecontents}
\usepackage{nicefrac,mathtools}
\usepackage{bm}  
\usepackage[nocompress]{cite}
\DeclareGraphicsExtensions{.pdf,.jpeg,.png}
\usepackage{epstopdf}
\usepackage{cancel} 
\usepackage[normalem]{ulem} 
\usepackage{verbatim} 
\usepackage{enumitem} 
\usepackage{tikz-cd}
\usetikzlibrary{cd}
\usepackage{color}
\usepackage[msc-links, lite]{amsrefs}
\usepackage{geometry}
\usepackage{tikz}
\usetikzlibrary{decorations.markings}
\usetikzlibrary{arrows.meta}

\usepackage{extarrows}
\newtheorem{theorem}{Theorem}[section]

\newtheorem{proposition}[theorem]{Proposition}
\newtheorem{lemma}[theorem]{Lemma}
\newtheorem{corollary}[theorem]{Corollary}
\newtheorem{question}[theorem]{Question}

\theoremstyle{definition}
\newtheorem{definition}[theorem]{Definition}
\newtheorem{example}[theorem]{Example}

\newtheorem{remark}[theorem]{Remark}

\numberwithin{equation}{section}

\numberwithin{equation}{section}

\makeatletter
\let\save@mathaccent\mathaccent
\newcommand*\if@single[3]{%
\setbox0\hbox{${\mathaccent"0362{#1}}^H$}%
\setbox2\hbox{${\mathaccent"0362{\kern0pt#1}}^H$}%
\ifdim\ht0=\ht2 #3\else #2\fi
}

\makeatother

\makeatletter
\newcommand*{\transpose}{%
{\mathpalette\@transpose{}}%
}
\newcommand*{\@transpose}[2]{%
\raisebox{\depth}{$\m@th#1\intercal$}%
}
\makeatother
\usetikzlibrary{hobby}

 \usetikzlibrary{decorations}
\makeatletter
\def\pgfutil@Repeat#1#2{#2\ifnum#1>0
  \expandafter\pgfutil@firstofone\else\expandafter\pgfutil@gobble\fi
  {\expandafter\pgfutil@Repeat\expandafter{\the\numexpr#1-1\relax}{#2}}}
\tikzset{
  dash between/.code args={#1 and #2}{%
    \tikz@addoption{%
      \pgfgetpath\currentpath
      \pgfprocessround{\currentpath}{\currentpath}%
      \pgf@decorate@parsesoftpath{\currentpath}{\currentpath}%
      \pgfmathsetlengthmacro\firstpart{(#1)*\pgf@decorate@totalpathlength}%
      \pgfmathsetlengthmacro\secondpart{(#2-(#1))*\pgf@decorate@totalpathlength}%
      \pgfmathsetlengthmacro\thirdpart{(1-(#2))*\pgf@decorate@totalpathlength}%
      \edef\thirdpart{{\thirdpart}{0pt}}%
      \edef\firstpart{{\firstpart}{0pt}}%
      \pgfmathsetlengthmacro\secondpartlength{\pgfkeysvalueof{/tikz/dash between on}
                                            +(\pgfkeysvalueof{/tikz/dash between off})}%
      \pgfmathtruncatemacro\repetitions{\secondpart/\secondpartlength}%
      \pgfmathsetlengthmacro\secondexpand{\secondpart/\repetitions-\secondpartlength}%
      \edef\secondexpand{\the\dimexpr\pgfkeysvalueof{/tikz/dash between off}+\secondexpand\relax}%
      \edef\secondpart{%
        \pgfutil@Repeat{\the\numexpr\repetitions-1\relax}%
          {{\pgfkeysvalueof{/tikz/dash between on}}{\secondexpand}}%
      }%
      \edef\tikz@temp{\firstpart\secondpart\thirdpart}%
      \expandafter\pgfsetdash\expandafter{\tikz@temp}{+0pt}%
    }
  }
}
\makeatother
\tikzset{
  dash between style/.is choice,
  dash between style/dotted/.style        ={dash between on=\pgflinewidth,dash between off=2pt},
  dash between style/densely dotted/.style={dash between on=\pgflinewidth,dash between off=1pt},
  dash between style/loosely dotted/.style={dash between on=\pgflinewidth,dash between off=4pt},
  dash between style/dashed/.style        ={dash between on=3pt,dash between off=2pt},
  dash between style/loosely dashed/.style={dash between on=3pt,dash between off=6pt},
  dash between style/densely dashed/.style={dash between on=3pt,dash between off=2pt},
  dash between style/no/.style={dash between on=0pt, dash between off=1pt},
  dash between on/.initial=\pgflinewidth,
  dash between off/.initial=2pt,
  middle dotted line/.style={
    thick,
    dash between=.35 and .65}}

\newcommand\QQ{\mathbb{Q}}
\newcommand\CC{\mathbb{C}}

\newcommand\ZZ{\mathbb{Z}}

\newcommand\FF{\mathbb{F}}
\newcommand\OO{\mathcal{O}}

\DeclareMathOperator{\Aut}{Aut}
\DeclareMathOperator{\homo}{Hom}
\DeclareMathOperator{\Tor}{Tor}
\let\Im\relax
\DeclareMathOperator{\Im}{Im}
\DeclareMathOperator{\Ker}{Ker}
\DeclareMathOperator{\Coker}{Coker}
\DeclareMathOperator{\et}{\acute{e}t}

\DeclareMathOperator{\dR}{dR}
\DeclareMathOperator{\crys}{crys}

\DeclareMathOperator{\Sym}{Sym}

\DeclareMathOperator{\Gal}{Gal}

\let\Im\relax
\DeclareMathOperator{\Im}{Im}

\newcommand{\catname}[1]{\textbf{#1}}

\newcommand{\Gr}{\catname{Gr}}

\DeclareMathOperator{\Spec}{Spec}
\DeclareMathOperator{\Lie}{Lie}

\DeclareMathOperator{\End}{End}
\DeclareMathOperator{\cont}{cont}

\def\hocolim{\qopname\relax m{hocolim}}

\usepackage[all]{xy}

\hypersetup{pdfborder=0 0 0}

\title{$\QQ_p$-homotopy types and applications to topology and algebraic geometry}
\author{Runjie Hu, Guozhen Wang}
\newcommand{\Addresses}{{
  \bigskip
  \footnotesize

  Runjie Hu, \textsc{Department of Mathematics, Texas A\&M University,
    College Sta, TX 77843, US}\par\nopagebreak
  \textit{E-mail address}, \texttt{ runjie.hu@tamu.edu}

  Guozhen Wang, \textsc{Fudan University Jiangwan Campus,
2005 Songhu Road, Shanghai, China}\par\nopagebreak
  \textit{E-mail address}, \texttt{ wangguozhen@fudan.edu.cn}
}}
\date{}
\begin{document}

\maketitle

\begin{abstract}
We develop a $\mathbb Q_p$-homotopy theory for $p$-complete spaces. To a $p$-complete space $X$, we associate a commutative
differential graded algebra over $\QQ_p$ by rectifying the $E_\infty$-algebra $S^*(X;\widehat{\ZZ}_p)\otimes_{\widehat{\ZZ}_p} \QQ_p$ of singular cochains. For nilpotent $p$-complete finite type spaces, we prove that the minimal model of this algebra recovers the $\mathbb Q_p$-homotopy groups and Whitehead products, in direct analogy with Sullivan's rational homotopy theory. We also prove that, for a non-simply-connected $p$-complete space, the Lie algebra dual to its $1$-minimal model is the Lie algebra of the continuous Mal'cev $\QQ_p$-completion of the fundamental group. 

We apply the $\mathbb Q_p$-homotopy theory to several questions in topology and algebraic geometry, including finite realization problems for $p$-complete spaces, finiteness properties of étale homotopy types, formality of smooth proper varieties, Galois representations on étale homotopy groups, and constraints on étale fundamental groups.
\end{abstract}

\section*{Introduction}

A classical problem at the interface of topology and algebraic geometry asks which homotopy types can be realized by complex algebraic varieties.  Sullivan’s rational homotopy theory (\cite{Sullivan-rational-homotopy}) provides a powerful algebraic approach to this question: many topological restrictions on smooth varieties can be detected from their de Rham algebras. This perspective underlies, for example, formality results for smooth proper complex varieties (\cite{Deligne-Griffiths-Morgan-Sullivan-formality}) and restrictions on their fundamental groups (\cite{Morgan-hodge-theory}).

In arithmetic geometry, the \'etale homotopy type plays a role analogous to that of the underlying topological space of a complex variety. It is therefore natural to seek a $p$-adic analogue of Sullivan’s theory and to use it to determine which $p$-complete spaces can arise from algebraic varieties. More precisely, we ask:

\begin{question}\label{Question: main question}
Which $p$-complete CW complexes are, up to homotopy equivalence, the $p$-completion of the \'etale homotopy type of a smooth algebraic variety, particularly one in positive characteristic?    
\end{question}

In this manuscript, we develop a $\QQ_p$-homotopy theory for $p$-complete spaces, and apply it to this question and several related problems. 

\subsubsection*{\underline{$\QQ_p$-homotopy theory}}\,

The central object in Sullivan's rational homotopy theory (\cite{Sullivan-rational-homotopy}) is the \textit{minimal model}. The theory
associates with every space a rational ``de Rham'' algebra whose cohomology is its rational singular cohomology ring. Algebraic invariants of the minimal model of the rational ``de Rham'' algebra are closely related to
topological invariants of the rational homotopy type, especially nilpotent rational spaces.

However, one cannot obtain a useful $\QQ_p$-homotopy theory for $p$-complete spaces merely by replacing $\QQ$ by $\QQ_p$ in Sullivan's construction, for two principal reasons. First, for a nilpotent $p$-complete space $X$, the cohomology naturally associated with $X$ is $H^*(X;\widehat{\ZZ}_{p})\otimes_{\widehat{\ZZ}_p}\QQ_p$, rather than the ordinary singular cohomology $H^*(X;\QQ_p)$. For instance, if $S^n_{p}$ denotes the $p$-completion of the $n$-sphere, then $H^{qn}(S^n_{p};\QQ_p)$ has uncountable dimension for $q\geq 2$ (\cite{more-concise}*{Remark 11.1.5}). Second, $p$-complete CW complexes are typically very large and often have uncountably many cells. Consequently, the usual rational ``de Rham'' construction does not directly yield a useful algebra over $\QQ_p$.

We overcome these difficulties by using rectification of $E_{\infty}$-algebras. For a $p$-complete space $X$, the singular cochain complex $S^*(X;\widehat{\ZZ}_p)\otimes_{\widehat{\ZZ}_p}\QQ_p$ naturally carries the structure of an $E_{\infty}$-algebra. By the
rectification theorem for $E_\infty$-algebras over fields of characteristic zero (\cite{Hinich-Model-Structure-on-homotopy-algbras}\cite{White-E-infinity-algebra-characteristic-zero}),
this $E_\infty$-algebra can be modeled, up to quasi-isomorphism, by a
commutative differential graded algebra (CDGA). We denote the resulting rectified CDGA by $C^*(X;\widehat{\ZZ}_p)\otimes_{\widehat{\ZZ}_p}\QQ_p$. 

Our first main result states that, for a nilpotent $p$-complete space $X$, the minimal model of the CDGA $C^*(X;\widehat{\ZZ}_p)\otimes_{\widehat{\ZZ}_p}\QQ_p$ recovers the homotopy groups of $X$ after tensoring them with $\mathbb Q_p$, and that its differential encodes Whitehead products, just as in Sullivan's rational homotopy theory.

\begin{theorem}[Theorem \ref{Thm: minimal model computes rational homotopy groups}]\label{Intro: Q_p minimal model and homotopy groups}
Let $X$ be a nilpotent $p$-complete finite type CW complex (Definition \ref{Def: finite type spaces}). Let $\phi:M\rightarrow C^*(X;\widehat{\ZZ}_p)\otimes_{\widehat{\ZZ}_p} \QQ_p$ be a minimal model, and let $L(n)=I^n(M)^{\vee}$ be the dual of the degree-$n$ indecomposable space of $M$. Then
\begin{enumerate}[leftmargin=0.25in]
    \item for every $i\geq 2$, there is a natural nondegenerate pairing  of $\QQ_p$-vector spaces  $\langle -,- \rangle: I^i(M)\otimes_{\QQ_p} (\pi_i(X)\otimes_{\widehat{\ZZ}_p} \QQ_p) \rightarrow \QQ_p$;
    \item the vector-space isomorphism $\bigoplus_{*\geq 2} L(*)\rightarrow \bigoplus_{*\geq 2}(\pi_{*}(X)\otimes_{\widehat{\ZZ}_p} \QQ_p)$ induced by item (1) is an isomorphism of graded Lie algebras, where $\bigoplus_{n\geq 2}L(n)$ carries the Lie bracket dual to the quadratic part of the differential of $M$ and $\bigoplus_{*\geq 2}\pi_{*}(X)$ carries the Whitehead products.
\end{enumerate}
\end{theorem}

Thus, the minimal model of the CDGA $C^*(X;\widehat{\ZZ}_p)\otimes_{\widehat{\ZZ}_p}\QQ_p$ contains more information than cohomological algebra. It is tantamount to all the $\QQ_p$-homotopy information of $X$, including the $\QQ_p$-homotopy groups and the Lie algebra structure induced by Whitehead products.

For a non-nilpotent space, the rational homotopy theory relates its $1$-minimal model to the Mal'cev completion (\cite{Malcev-nilpotent-torsion-free-groups}\cite{Quillen-rational-homotopy}\cite{Betts-thesis}\cite{Clement-Majewicz-Zyman-theory-nilpotent-groups}) of its fundamental group. We also establish a $p$-adic analogue of this correspondence. 

The \textit{$1$-minimal model} $M(1)$ of a CDGA is the subalgebra of its minimal model generated by degree-$1$ elements. The dual of the differential $d:M(1)^1\rightarrow M(1)^2$ defines a Lie bracket. The relationship between Mal'cev completion and minimal models is expressed by the following theorem.

\begin{theorem}[Theorem \ref{Thm: 1-minimal model is Malcev completion}]
Let $X$ be a connected CW complex whose fundamental group $\pi_1(X)$ is a topologically finitely generated pro-$p$ group. Let $\phi_X:M_X(1)\rightarrow C^*(X;\widehat{\ZZ}_p)\otimes_{\widehat{\ZZ}_p}\QQ_p$ be a $1$-minimal model, and let $L_X(1)$ denote the pro-nilpotent Lie algebra dual to $M_X(1)^1$. Then there is a canonical isomorphism of pro-nilpotent Lie algebras $L_X(1)\cong \Lie(\pi_1(X)\widehat{\otimes} \QQ_p)$.
\end{theorem}

\subsubsection*{\underline{Finite realization of $p$-complete spaces}}\,

We use the $\QQ_p$-homotopy theory to determine when a $p$-complete space can be realized as the $p$-completion of a finite CW complex, which is the following question. 

\begin{question}\label{Question: finite CW complex realization for p-complete space}
Let $X$ be a $p$-complete space such that $H^q(X;\widehat{\ZZ}_p)$ is a finitely generated $\widehat{\ZZ}_p$-module for each $q$ and vanishes for all sufficiently large $q$. Must $X$ be the $p$-completion of a finite CW complex?
\end{question}

The notion of ``$p$-completion'' here is Sullivan's $p$-completion (\cite{Sullivan-MIT-notes}*{p.~73, the last paragraph}), which is defined even for non-nilpotent spaces. This is \textit{not} Bousfield-Kan's $\ZZ/p$-localization, Bousfield's $\mathbb{H}(\ZZ/p)$-localization or Bousfield's $\mathbb{S}(\ZZ/p)$-localization.

Belfi and Wilkerson \cite{Victor-Wilkerson-Counterexample-p-completion}*{the paragraph above Lemma 3.4} constructed a nilpotent $p$-complete finite type CW complex that is not the $p$-completion of any nilpotent finite type CW complex. They also showed a sufficient condition for a nilpotent $p$-complete finite type CW complex to be the $p$-completion of a nilpotent finite type CW complex in \cite{Victor-Wilkerson-Counterexample-p-completion}*{Theorem 3.1}.

We show a sufficient and necessary condition for Question \ref{Question: finite CW complex realization for p-complete space} in the simply-connected case. 

\begin{theorem}[Corollary \ref{Cor: integral lifting of simply-connected p-complete spaces}]\label{Intro Thm: p-completion of a finite CW complex}
Let $X$ be a simply-connected $p$-complete finite type CW complex. The following statements are equivalent.
\begin{enumerate}[leftmargin=0.25in]
    \item $X$ is the $p$-completion of a finite CW complex.
    \item The singular cochain complex $S^*(X;\widehat{\ZZ}_p)\otimes_{\widehat{\ZZ}_p} \QQ_p$ is quasi-isomorphic to the scalar extension of an $E_{\infty}$-algebra over $\QQ$, and $H^{q}(X;\ZZ/p)=0$ for sufficiently large $q$.
    \item The $\QQ_p$-minimal CDGA model of $X$ is the scalar extension of a minimal CDGA over $\QQ$, and $H^{q}(X;\ZZ/p)=0$ for sufficiently large $q$.
\end{enumerate}
\end{theorem}

We briefly introduce our idea of Theorem \ref{Intro Thm: p-completion of a finite CW complex} here. From the following arithmetic square of a simply-connected finite type space $Z$ as in \cite{Sullivan-MIT-notes}*{p.~87}, the sufficient and necessary condition for a simply-connected $p$-complete finite type CW complex $X$ to be the $p$-completion of a simply-connected $p$-local space is whether its rationalization is the formal $p$-completion of a simply-connected rational space.
\[
\begin{tikzcd}
    Z_{(p)} \arrow[rrr,"\text{$p$-completion}"] \arrow[d,"\text{rationalization}"'] & & & Z_p \arrow[d,"\text{rationalization}"] \\
    Z_{\QQ} \arrow[rrr,"\text{formal $p$-completion}"'] & & &  Z_{\QQ_p}
\end{tikzcd}
\]
This condition is equivalent to whether the $\QQ_p$-minimal model of $X$ is a $\QQ$-minimal model tensoring with $\QQ_p$. Along with Lemma \ref{Lem: local spaces have an integral lifting}, which shows that every simply-connected $p$-local finite space is the $p$-localization of a simply-connected finite type space, we obtain the following theorem.

Moreover, we also construct a counterexample to Question \ref{Question: finite CW complex realization for p-complete space}. Note that our example is different from Belfi-Wilkerson's in two aspects: (1) we add the cohomological condition on the $p$-complete space $X$ to make $X$ cohomologically look like the $p$-completion of a finite CW complex; (2) we consider the $p$-completion of arbitrary finite CW complexes, including non-nilpotent cases.

\begin{theorem}[Example \ref{Example: p-complete space which is not a finite type space}]
There exists a simply-connected, $p$-complete finite type CW complex $X$ such that $H^q(X;\widehat{\ZZ}_p)\otimes_{\widehat{\ZZ}_p} \QQ_p=0$ for all sufficiently large $q$, but $X$ is not, up to homotopy equivalence, the $p$-completion of any finite CW complex.
\end{theorem}

\subsubsection*{\underline{\'Etale homotopy types and finite CW complexes}}\,

We also apply the $\QQ_p$-homotopy theory to several problems in algebraic geometry. 

In \cite{grothendieck2021pursuing}*{p.~45, Appendix \S17}, Grothendieck suggested that the \'etale
homotopy type of a well-behaved scheme should resemble a finite polyhedron. This motivates the following realization problem.

\begin{question}\label{Question: finite CW complex for varieties}
When is the $\ell$-completion of the \'etale homotopy type of an algebraic variety over an algebraically closed field, up to homotopy equivalence, the $\ell$-completion of a finite CW complex?
\end{question}

There are two reasons why this question is useful. First, for a smooth proper variety in positive characteristic, the finite CW realizability of its
$\ell$-completed \'etale homotopy type can provide an obstruction to lifting the variety to characteristic zero. Second, many tools in topology, including Spanier-Whitehead duality and surgery theory, are most naturally formulated for finite CW complexes.

A first indication is that, for a proper variety, each $\QQ_{\ell}$-\'etale cohomology group is finite-dimensional and the cohomology vanishes in sufficiently high degree. Further evidence is that the $\ell$-completed \'etale homotopy type of an $\ell$-adically simply-connected smooth proper variety is $\ell$-complete finite type (\cite{Hu-Zhang-formal-manifold-structures}). Nevertheless, finite-dimensional cohomology alone is not sufficient. Our $\QQ_\ell$-homotopy theory shows that the cohomology algebra must  also satisfy a stronger rational descent condition.

\begin{theorem}[Corollary \ref{Cor: finite CW complex for positive haract3eristic variety}]
\label{Cor: finite CW complex for positive haract3eristic variety}\label{Thm: finitenes result for varieties in Introduction}
Let $X$ be a connected smooth proper variety over a separably closed field $k$ of characteristic $p\geq 0$, and let $\ell$ be a prime number. Assume that $(\pi_1^{\et}X)^{\wedge}_{\ell}=0$. Then the \'etale homotopy type of $X$ is $\ell$-adic weak equivalent to a finite CW complex if and only if $H^*_{\et}(X;\QQ_{\ell})$ is isomorphic to the scalar extension of a graded $\QQ$-algebra.
\end{theorem}

In \cite{Hu-Zhang-nonliftable-varieties}, Hu-Zhang constructed a simply-connected smooth projective variety $X$ over $\overline{\FF}_p$ whose $H^*_{\et}(X;\QQ_{\ell})$ is not the scalar extension of any graded $\QQ$-algebra, for every $\ell\neq p$. From Theorem \ref{Thm: finitenes result for varieties in Introduction}, the $\ell$-completion of the \'etale homotopy type of Hu-Zhang's example $X$ is not the $\ell$-completion of any finite CW complex. In particular, $X$ does not admit a characteristic zero lifting.

\subsubsection*{\underline{Formality of smooth proper varieties}}\,

A CDGA is \textit{formal} if its minimal model is quasi-isomorphic to its cohomology ring, equipped with the zero differential. Formality is a strong constraint: for example, it implies that every Massey product vanishes.

In \cite{Deligne-Griffiths-Morgan-Sullivan-formality}, Deligne, Griffiths, Morgan and Sullivan proved that smooth proper complex varieties are formal over $\QQ$. Later, Deligne (\cite{Deligne-Weil-II}*{Section 5}) sketched a proof for the $\QQ_{\ell}$-formality for the \'etale homotopy type of a smooth proper variety in characteristic $p\geq 0$ with $\ell\neq p$, based on the idea that ``purity implies formality'' (\cite{Cirici-Horel-formality-torsion-coefficient}\cite{Emprin-Horel-weight-and-formality}). We prove that the formality result also holds for the $\ell=p$ case. A basic idea is to use purity of the Frobenius weight on $\QQ_p$-\'etale cohomology (\cite{Chiarellotto-le-Stum-on-purity-of-crystalline-cohomology}\cite{Milne-values-of-zeta-functions-of-varieties-over-finite-fields}). 

\begin{theorem}[Theorem \ref{Thm: formality of varieties}, Deligne-Griffiths-Morgan-Sullivan for $p=0$, Deligne for $\ell\neq p$]\label{Thm: formality in Introduction}\label{Intro Thm: formality}
Let $k$ be a separably closed field of characteristic $p\geq 0$, and let $\ell$ be a prime. Let $X$ be a connected smooth proper variety over $k$. Then the $\ell$-completion of the \'etale homotopy type of $X$ is $\QQ_{\ell}$-formal.    
\end{theorem}

Theorem \ref{Intro Thm: formality} gives a constraint for Question \ref{Question: main question} on the $\ell$-complete spaces that can arise from \'etale homotopy types of smooth proper varieties.

With the $\QQ_{p}$-homotopy theory, we also obtain lower and upper bounds for the Frobenius weights of higher \'etale homotopy groups (Theorem \ref{Thm: weights on etale homotopy groups}).

\subsubsection*{\underline{Continuity of homotopy Galois actions on $\QQ_{\ell}$ minimal models}}\,

The duality between the $\QQ_p$-minimal model of a $p$-complete space and its homotopy groups, established in Theorem \ref{Intro: Q_p minimal model and homotopy groups}, allows us to apply the $\QQ_p$-homotopy theory to study the Galois representations on higher \'etale homotopy groups. From a topological perspective, the Galois action on \'etale homotopy groups is more fundamental than the induced action on cohomology groups. 

Before doing so, however, we must address the following question posed by Deligne in \cite{Deligne-Weil-II}*{(5.3.9.1)}.

\begin{question}[Deligne]\label{Question: Deligne}
Let $k$ be a field of characteristic $p\geq 0$, let $\ell$ be a prime, and let $X$ be a connected smooth proper variety over $k$. Let $M$ be a $\QQ_{\ell}$-minimal model for the $\ell$-completion of the \'etale homotopy type of $X_{\overline{k}}$. 
Is the natural homomorphism $\Gal(\overline{k}/k)\rightarrow \Aut_h(M)$ continuous, where $\Aut_h(M)$ is the group of homotopy classes of automorphisms of $M$, equipped with its natural $\ell$-adic topology?
\end{question}

We provide an affirmative answer to Question \ref{Question: Deligne}.

\begin{proposition}[Proposition \ref{Prop: continuity of Galois action on automorphism group of minimal model}]\label{Intro Prop: continuity from Galois to automorphisms of minimal models}
Let $k$ be a field of characteristic $p\geq 0$, and let $\ell$ be a prime number. Let $X$ be a connected smooth proper variety over $k$, equipped with a $k$-rational point. Let $M$ be a $\QQ_{\ell}$-minimal model of $(X_{\overline{k}})^{\wedge}_{\et,\ell}$. Then the natural homomorphism $\Gal(\overline{k}/k)\rightarrow \Aut_h(M)$ induced by the continuous $\Gal(\overline{k}/k)$-action on $(X_{\overline{k}})^{\wedge}_{\et,\ell}$ is continuous, where $\Aut_h(M)$ is equipped with the $\ell$-adic topology arising from its pro-algebraic group structure over $\QQ_{\ell}$.  
\end{proposition}

The idea to prove Proposition \ref{Intro Prop: continuity from Galois to automorphisms of minimal models} is the following. Sullivan in \cite{Sullivan-rational-homotopy}*{Theorem 6.1} showed that, if $M$ is finitely generated, $\Aut_h(M)$ is an algebraic group over $\QQ_{\ell}$. In \cite{Pridham-Galois-actions-on-homotopy-groups-algebraic-varieties}, Pridham constructed a simplicial pro algebraic group over $\QQ_p$ for any space. This construction provides an algebraic model for the ``$\QQ_p$-completion'' of a space. For a general minimal model $M$, we use Pridham's results to show that $\Aut_h(M)$ is a pro algebraic group over $\QQ_{\ell}$ and the homomorphism $\Gal(\overline{k}/k)\rightarrow \Aut_h(M)$ factors through the algebraization of $\Gal(\overline{k}/k)$.

\subsubsection*{\underline{Galois actions on \'etale homotopy groups}}\,

\'Etale cohomology carries natural Galois actions satisfying strong arithmetic constraints. For example, $p$-adic Hodge theory studies the property of Galois representations on $\QQ_p$-\'etale cohomology groups of $p$-adic varieties. We use the $\QQ_p$-homotopy theory to study the $p$-adic Hodge theoretic properties of Galois representations on higher \'etale homotopy groups.

We first briefly recall some results of $p$-adic Hodge theory. Let $K/\QQ_{p}$ be a finite field extension, and let $\OO_K$ be its ring of integers. Fontaine introduced
the $p$-adic period rings $B_{\dR}$ and $B_{\crys}$, which are $\QQ_{p}[\Gal(\overline{K}/K)]$-algebras  (\cite{Fontaine-p-adic-hodge-I}\cite{Fontaine-p-adic-Hodge-II}). One has that $B_{\dR}^{\Gal(\overline{K}/K)}=K$ and $B_{\crys}^{\Gal(\overline{K}/K)}=K_0$, where $K_0$ is the maximal unramified subfield of $K$. A continuous finite-dimensional $\QQ_p$-representation $V$ of $\Gal(\overline{K}/K)$ is \textit{de Rham} or \textit{crystalline} if $\dim_K((B_{\dR}\otimes_{\QQ_p}V)^{\Gal(\overline{K}/K)})=\dim_{\QQ_p}(V)$ or $\dim_{K_0}((B_{\crys}\otimes_{\QQ_p}V)^{\Gal(\overline{K}/K)})=\dim_{\QQ_p}(V)$, respectively. For every connected smooth proper algebraic variety $X$ over $K$, the representation $H^n(X_{\overline{K}};\QQ_p)$ is de Rham for each $n$. If $X$ has \textit{good reduction}, namely, if there exists a smooth proper scheme $\chi$ over $\OO_K$ with generic fiber $\chi_{K}=X$, then $H^n(X_{\overline{K}};\QQ_p)$ is crystalline for every $n$. 

We prove that the higher \'etale homotopy groups satisfy weaker $p$-adic Hodge-theoretic properties. Define a continuous finite-dimensional $\Gal(\overline{K}/K)$-representation to be \textit{successively de Rham} if it admits a finite $\Gal(\overline{K}/K)$-stable filtration whose graded pieces are de Rham (Definition \ref{Def: successively de Rham}). Successively crystalline representations are defined analogously. Using the $\QQ_p$-homotopy theory, we then obtain the following theorem for Galois representations on the \'etale homotopy groups of a $p$-adic variety.

\begin{theorem}[Theorem \ref{Thm: de Rham represenations for etale homotopy groups}, Theorem \ref{Thm: crystalline represenations for etale homotopy groups}]\label{Intro Thm: p-adic hodge}
Let $K/\QQ_p$ be a finite field extension, and let $\mathcal{O}_{K}$ denote its ring of integers.
\begin{enumerate}[leftmargin=0.25in]
    \item Let $X$ be a connected smooth proper adic space over $\mathrm{Spa}(K,\OO_K)$, equipped with a $\mathrm{Spa}(K,\OO_K)$-point. 
    \begin{enumerate}[leftmargin=0.27in]
        \item The Lie algebra of the continuous Mal'cev $\QQ_p$-completion of $\pi^{\et}_1(X_{\overline{K}})^{\wedge}_{p}/\Gamma^n$ is a successively de Rham representation of $\Gal(\overline{K}/K)$ for every $n$, and the Lie bracket respects the $\Gal(\overline{K}/K)$-action, where $\Gamma^*$ denotes the lower central series.
    \item If $\pi^{\et}_1(X_{\overline{K}})^{\wedge}_{p}=0$, then $\pi_n((X_{\overline{K}})^{\wedge}_{\et,p})\otimes_{\widehat{\ZZ}_{p}}\QQ_{p}$ is a successively de Rham representation of $\Gal(\overline{K}/K)$, for every $n\geq 2$.
    \end{enumerate}
    \item Let $X$ be a connected smooth proper algebraic variety over $K$, equipped with a $K$-rational point. Assume that $X$ has good reduction.
    \begin{enumerate}[leftmargin=0.27in]
        \item The Lie algebra of the continuous Mal'cev $\QQ_p$-completion of $\pi^{\et}_1(X_{\overline{K}})^{\wedge}_{p}/\Gamma^n$ is a successively crystalline representation of $\Gal(\overline{K}/K)$ for every $n$, and the Lie bracket respects the $\Gal(\overline{K}/K)$-action.
    \item If $\pi^{\et}_1(X_{\overline{K}})^{\wedge}_{p}=0$, then $\pi_n((X_{\overline{K}})^{\wedge}_{\et,p})\otimes_{\widehat{\ZZ}_{p}}\QQ_{p}$ is a successively crystalline representation of $\Gal(\overline{K}/K)$, for every $n\geq 2$.
    \end{enumerate}
\end{enumerate}
\end{theorem}

For item (2)(a), Pridham (unpublished \cite{Pridham-arxiv-galois-actions-pro-l-unipotent-fundamental}*{Corollary 3.4}) and Olsson (\cite{Olsson-p-adic-hodge}*{Theorem 1.8}) independently proved that the Lie algebra of the continuous Mal'cev $\QQ_p$-completion of $\pi^{\et}_1(X_{\overline{K}})^{\wedge}_{p}$ is a pro-object of crystalline representations.

For item (2)(b), Pridham defined a finite-dimensional $\QQ_p$-representation of $\Gal(\overline{K}/K)$ to be \textit{potentially crystalline} if there exists a finite extension $K'$ of $K$ such that the action of $\Gal(\overline{K}/K')$ is crystalline (\cite{Pridham-Galois-actions-on-homotopy-groups-algebraic-varieties}*{Definition 5.17}). He then showed that $\pi_n((X_{\overline{K}})^{\wedge}_{\et,p})\otimes_{\widehat{\ZZ}_{p}}\QQ_{p}$ is potentially crystalline if $\pi^{\et}_1(X_{\overline{K}})^{\wedge}_{p}=0$ (\cite{Pridham-Galois-actions-on-homotopy-groups-algebraic-varieties}*{Theorem 2.25, Corollary 7.29}). In \cite{Olsson-p-adic-hodge}*{Theorem 1.9}, Olsson proved that certain algebraic homotopy groups of $X$ are pro-objects of crystalline representations.

For item (1)(a), Betts in \cite{Betts-de-Rham}*{Theorem 1.4} proved that, if $X$ is the analytification of a connected smooth algebraic variety over $K$, then the Lie algebra of the continuous Mal'cev $\QQ_p$-completion of $\pi^{\et}_1(X_{\overline{K}})^{\wedge}_{p}$ is a pro-object of de Rham representations of $\Gal(\overline{K}/K)$.

To the best of our knowledge, item (1)(b) is not known in references. Related results appeared in \cite{Deglise-niziol-p-adic-hodge}*{Lemma 2.22} and \cite{Petrov-Zavyalov-formality-rigid-analytic}*{Proposition 4.17}. 

\subsubsection*{\underline{\'Etale fundamental groups}}\,

Finally, we apply the relationship between $1$-minimal models and continuous Mal'cev completions to the study of \'etale fundamental groups. Morgan proved that the Lie algebra of the Mal'cev completion of the fundamental group of a smooth complex variety admits relations of bounded bracket lengths (\cite{Morgan-hodge-theory}): in the proper case, the Mal'cev Lie algebra is quadratically presented, and in the non-proper case the defining ideal is generated by brackets of lengths $2$, $3$ and $4$. We establish an analogous result for \'etale fundamental groups, thereby obtaining another constraint on the realization problem in Question \ref{Question: main question}.

\begin{theorem}[Theorem \ref{Thm: fundamental group of algebraic varieties}]
Let $k$ be a separably closed field of characteristic $p\geq 0$, let $\ell$ be a prime number, and let $X$ be a connected smooth variety over $k$. 
\begin{enumerate}[leftmargin=0.25in]
    \item If $X$ is proper, then the Lie algebra of the continuous Mal'cev $\QQ_{\ell}$-completion of  $\pi_1^{\et}(X)^{\wedge}_{\ell}$ is the pro-nilpotent completion of a free Lie algebra on a finite-dimensional $\QQ_{\ell}$-vector space modulo an ideal generated by elements of bracket length $2$. In particular, $\pi_1^{\et}(X)^{\wedge}_{\ell}\widehat{\otimes }\QQ_{\ell}$ is determined by $(\pi_1^{\et}(X)^{\wedge}_{\ell}/\Gamma^3)\widehat{\otimes}\QQ_{\ell}$.
    \item If $\ell\neq p$ and $k$ is the algebraic closure of a finite field, then the Lie algebra of the continuous Mal'cev $\QQ_{\ell}$-completion of $\pi_1^{\et}(X)^{\wedge}_{\ell}$ is the pro-nilpotent completion of a free Lie algebra on a finite-dimensional $\QQ_{\ell}$-vector space modulo an ideal generated by elements of bracket lengths $2$, $3$ and $4$. In particular, $\pi_1^{\et}(X)^{\wedge}_{\ell}\widehat{\otimes }\QQ_{\ell}$ is determined by $(\pi_1^{\et}(X)^{\wedge}_{\ell}/\Gamma^5)\widehat{\otimes}\QQ_{\ell}$.
\end{enumerate}
\end{theorem}

Pridham (unpublished \cite{Pridham-etale-fundamental-group-of-smooth-variety}*{Corollaries 3.4, 3.6}) used Lie models to obtain an analogous result when $\ell\neq p$ and $k$ is a finite field, while our proof is based on $1$-minimal CDGA's and treats more cases.

\subsubsection*{\underline{Organization of the paper}}\,

Section \ref{Section: Q_p homotopy theory} constructs the $\mathbb Q_p$-homotopy theory of $p$-complete spaces.
For every space $X$, we define the CDGA $C^*(X;\widehat{\ZZ}_p)
\otimes_{\widehat{\ZZ}_p}\QQ_p$ by rectifying the $E_{\infty}$-algebra $S^*(X;\widehat{\ZZ}_p)
\otimes_{\widehat{\ZZ}_p}\QQ_p$ of singular cochains. We then show that its minimal model recovers the $\QQ_p$-homotopy groups and
Whitehead products, while its $1$-minimal model recovers the Lie algebra of the continuous Mal'cev $\QQ_p$-completion of the fundamental group. Section \ref{Section: Application} applies the $\QQ_p$-homotopy theory to problems in topology and algebraic geometry. We study the finite CW complex realization problem for $p$-complete spaces and \'etale homotopy types, prove $\QQ_{\ell}$-formality for smooth proper varieties, analyze Galois representations on higher \'etale homotopy groups, and obtain bounded presentations for \'etale fundamental groups.

\medskip

\subsubsection*{Acknowledgements.}
We thank Fedor Manin, John Morgan, John Pardon, Shmuel Weinberger, Ruijie Yang and Siqing Zhang for valuable comments and discussions. We are grateful to Dennis Sullivan for discussions on rational homotopy theory and for suggesting several of the questions studied in this manuscript. R.H. is supported by 
NSF Grant 2247322. G.W. is partially supported by grants NSFC12325102, NSFC-12226002, the New Cornerstone Science Foundation, and Shanghai
Pilot Program for Basic Research–Fudan University 21TQ1400100 (21TQ002).

\tableofcontents

\subsection*{Terminology and Notation}

\begin{itemize}[leftmargin=0.25in]
    \item CDGA abbreviates ``commutative differential graded algebra''.
    \item For a CDGA $C$, $\Aut(C)$ is the group of automorphisms of $C$ and $\Aut_h(C)$ is the group of homotopy classes of automorphisms of $C$.
    \item For simplicity, we use the definition of minimal CDGA's in \cite{Deligne-Griffiths-Morgan-Sullivan-formality}\cite{Morgan-hodge-theory}; these are called nilpotent minimal CDGA's in \cite{Sullivan-rational-homotopy}.
    \item For a pointed CW complex $Z$, $Z^{\wedge,S}_{p}$ denotes Sullivan's $p$-completion (\cite{Sullivan-MIT-notes}*{Chapter III}). \cite[Section 2.3]{Barthel-Bousfield-Comparison-of-p-completion} shows that, for nilpotent finite type  CW complexes, Bousfield's $\mathbb{H}(\ZZ/p)$-localization, Bousfield-Kan's $\ZZ/p$-localization (\cite{Bousfield-Kan-homotopy-limits-completions-localizations}) and Sullivan's $p$-completion agree up to homotopy.
    \item For a CW complex $X$ and a commutative ring $k$, $S^*(X;k)$ denotes the $E_{\infty}$-algebra of singular cochains on $X$. If $k$ is a field containing $\QQ_p$, $C^*(X;\widehat{\ZZ}_p)\otimes_{\widehat{\ZZ}_p}k$ is a CDGA rectification of $S^*(X;\widehat{\ZZ}_p)\otimes_{\widehat{\ZZ}_p}k$ (paragraph below Fact \ref{Fact: rectification}).
    \item For a group $G$, we write $\Gamma^*G$ and $Z_*(G)$ for its lower and upper central series, respectively.
    \item For a topological group $G$ and a topological field $k$ of characteristic zero, we write $G\widehat{\otimes } k$ for the continuous Mal'cev $k$-completion of $G$ (Definition \ref{Fact: definition of continuous Malcev completion}).
    \item We say that $\QQ_p$ is a weakly cartesian subfield of a topological field $k$ if every finite-dimensional $\QQ_p$-linear subspace of $k$ is closed (Definition \ref{Def: weakly cartesian subfield}). A typical example of such a topological field $k$ is a closed subfield of $\CC_p$ (Example \ref{Example: weakly cartesian pair}).
    \item Several notions of finite type for pointed CW complexes are introduced in Definition \ref{Def: finite type spaces}.
\end{itemize}

\section{$\QQ_p$-homotopy Theory}\label{Section: Q_p homotopy theory}

In this section, we develop a $\QQ_p$-homotopy theory for $p$-complete spaces. For a $p$-complete space $X$, we consider the $E_\infty$-algebra
$S^*(X;\widehat{\ZZ}_p)\otimes_{\widehat{\ZZ}_p}\QQ_p$
of singular cochains and use rectification to model it by a commutative differential graded algebra $C^*(X;\widehat{\ZZ}_p)\otimes_{\widehat{\ZZ}_p}\QQ_p$ over $\mathbb Q_p$.
We then study the minimal model of this CDGA. For a nilpotent $p$-complete finite type space $X$, the indecomposable spaces of its $\QQ_{p}$-minimal model recover the $\QQ_p$-homotopy groups of $X$, and the quadratic part of its differential encodes Whitehead products (Theorem \ref{Thm: minimal model computes rational homotopy groups}). We also study the $1$-minimal model for the non-simply-connected case. For spaces whose fundamental groups are topologically finitely generated pro-$p$ groups, we identify the Lie algebra dual to the $1$-minimal model with the Lie algebra of the continuous Mal’cev $\QQ_p$-completion of the fundamental group (Theorem \ref{Thm: 1-minimal model is Malcev completion}). 

We extend these results to pro-systems of $p$-finite spaces (Corollary \ref{Cor: duality between minimal model and homotopy groups for pro spaces}), which will be used to study étale homotopy types of algebraic varieties in Section \ref{Section: Application}.

\subsection{Nilpotent $p$-Complete Finite Type Spaces}\,

In Sullivan’s rational homotopy theory, the minimal model of a nilpotent rational finite-type space $X$ recovers its homotopy type: the degree-$n$ indecomposable space is dual to $\pi_n(X)\otimes \mathbb{Q}$, while the quadratic part of the differential is dual to the Whitehead product. This construction is closely related to the Postnikov tower, whose successive $k$-invariants correspond to elementary extensions of minimal CDGA's.

In this subsection, we establish $p$-adic analogues of these results for a nilpotent $p$-complete finite-type space $X$. Using a refined principal Postnikov tower and the correspondence between the $k$-invariants of $X$ and elementary extensions of minimal CDGA's, we show that the degree-$n$ indecomposable space of a $\QQ_p$-minimal model is naturally dual to $\pi_n(X)\otimes_{\widehat{\ZZ}_{p}} \QQ_p$. We also prove that the induced Lie bracket agrees with the Whitehead product (Theorem \ref{Thm: minimal model computes rational homotopy groups}).

We begin by reviewing various notions of finite type for spaces.

\begin{definition}\label{Def: finite type spaces}
Let $X$ be a pointed connected CW complex, and let $T$ be a set of primes.
\begin{enumerate}[leftmargin=0.25in]
    \item $X$ is \textbf{nilpotent finite type} if $\pi_1(X)$ is a finitely generated nilpotent group, and for every $i\geq 2$, $\pi_i(X)$ is a finitely generated abelian group that is nilpotent as a $\pi_1(X)$-module.
    \item $X$ is \textbf{nilpotent $T$-local finite type} if $\pi_1(X)$ is a finitely generated nilpotent $T$-local group, and, for every $i\geq 2$, $\pi_i(X)$ is a finitely generated $\ZZ_{(T)}$-module that is nilpotent as a $\pi_1(X)$-module.
    \item $X$ is \textbf{nilpotent $p$-complete finite type} if $\pi_1(X)$ is a topologically finitely generated nilpotent pro-$p$ group, and for every $i\geq 2$, $\pi_i(X)$ is a finitely generated $\widehat{\ZZ}_p$-module that is nilpotent as a continuous $\pi_1(X)$-module.
    \item $X$ is \textbf{cohomologically $p$-complete finite type} if, for every $i$, $H^i(X;\ZZ/p)$ is a finitely generated $\widehat{\ZZ}_p$-module.
\end{enumerate}
\end{definition}

Let $k$ be a field containing $\QQ_p$. 

\begin{definition}
Assume that $k$ is a topological field whose topology restricts to the $p$-adic topology on $\QQ_p$. The pair $\QQ_p\subset k$ is called a \textbf{weakly cartesian pair} if every finite-dimensional $\QQ_p$-linear subspace of $k$ is closed.
\end{definition}

The notion of weakly cartesian pairs for more general topological fields is given in Definition \ref{Def: weakly cartesian subfield}. A typical example for a weakly cartesian pair $\QQ_p\subset k$ is obtained when $k$ is closed subfield of $\CC_p$ (Example \ref{Example: weakly cartesian pair}).

We fix a cofibrant $E_{\infty}$-operad $\mathcal{E}$ over $k$ as in Theorem \ref{Thm: simplicial cochain complex is a functorial E-infinity algebra}.

\begin{definition}\label{Def: rectified CDGA}
Let $X$ be a CW complex. Then the singular cochain $S^*(X;\widehat{\ZZ}_p)\otimes_{\widehat{\ZZ}_p} k$ is naturally an $\mathcal{E}$-algebra. Define $
C^{*}(X;\widehat{\ZZ}_{p})\otimes_{\widehat{\ZZ}_{p}} k
$ to be the rectified CDGA of the $E_{\infty}$-algebra $S^*(X;\widehat{\ZZ}_p)\otimes_{\widehat{\ZZ}_p} k$ of $S^*(X;\widehat{\ZZ}_p)\otimes_{\widehat{\ZZ}_p} k$ as in Definition \ref{Def: rectification of an E-infinity algebra}.
\end{definition}

Theorem \ref{Fact: rectification} shows that $S^*(X;\widehat{\ZZ}_p)\otimes_{\widehat{\ZZ}_p} k$ and $
C^{*}(X;\widehat{\ZZ}_{p})\otimes_{\widehat{\ZZ}_{p}} k
$ are connected by a functorial zigzag of quasi-isomorphisms, as $\mathcal{E}$-algebras. 

\begin{remark}
In Definition \ref{Def: rectified CDGA}, we do not use the singular cochain $S^*(X;k)$. In general, $S^*(X;k)$ is not quasi-isomorphic to $S^*(X;\widehat{\ZZ}_p)\otimes_{\widehat{\ZZ}_p} k$ for $p$-complete spaces. For example, if $X$ is the $p$-complete $n$-sphere, then $H^{qn}(X;\QQ_p)$ has uncountable dimension for $q\geq 2$ (\cite{more-concise}*{Remark 11.1.5}), while $H^{qn}(X;\widehat{\ZZ}_p)\otimes_{\widehat{\ZZ}_p}\QQ_p=0$ for $q\geq 2$.
\end{remark}

The following theorem is one of the main results of this section. Its proof is given at the end of this subsection.

\begin{theorem}\label{Thm: minimal model computes rational homotopy groups}
Let $X$ be a connected nilpotent $p$-complete finite type CW complex. Let $\phi:M\rightarrow C^*(X;\widehat{\ZZ}_p)\otimes_{\widehat{\ZZ}_p} k$ be a minimal model, and let $L(n)$ be the dual of the degree-$n$ indecomposable space $I^n(M)$. Let $\Gamma^*$ be the lower central series of $\pi_1(X)$. Then
\begin{enumerate}[leftmargin=0.25in]
    \item there are natural nondegenerate $k$-bilinear pairings as follows;
    \[
\langle -,- \rangle: I^i(M)\otimes_{k} (\pi_i(X)\otimes_{\widehat{\ZZ}_p} k) \rightarrow k \text{\,\,\, (for $i\geq 2$)}
\]
\[
\langle -,- \rangle: I^1(M)\otimes_{k} ((\bigoplus_s \Gamma^s\pi_1(X)/\Gamma^{s+1}\pi_1(X))\otimes_{\widehat{\ZZ}_p} k) \rightarrow k
\]
    \item the vector-space isomorphism $\bigoplus_{*\geq 2} L(*)\rightarrow \bigoplus_{*\geq 2}(\pi_{*}(X)\otimes_{\widehat{\ZZ}_p} k)$ induced by item (1) is an isomorphism of graded Lie algebras after shifting degrees by $1$, where the Lie bracket on $\bigoplus_{*\geq 2}L(*)$ is dual to the quadratic part of the differential $d$ of $M$, as in Theorem \ref{Fact: nilpotent Lie algebra induced from minimal model}, and on $\bigoplus_{*\geq 2}\pi_{*}(X)$ it is given by the Whitehead product;
    \item if, in addition, $\QQ_p\subset k$ is a closed weakly cartesian pair of Hausdorff topological fields, then the nilpotent Lie algebra $L(1)$ associated with the $1$-minimal model, as in Theorem \ref{Fact: nilpotent Lie algebra induced from minimal model}, is naturally isomorphic to the Lie algebra $\Lie(\pi_1(X)\widehat{\otimes} k)$ of the continuous Mal'cev $k$-completion of $\pi_1(X)$.
\end{enumerate}
\end{theorem}

\begin{lemma}\label{Lemma: Computation of Eilenberg-Maclane spaces}
Let $A$ be a finitely generated $\widehat{\ZZ}_p$-module, and let $V_n=\homo_{\widehat{\ZZ}_p}(A,k)$, regarded as a graded $k$-vector space concentrated in degree $n$. Then there are natural isomorphisms $H^*(C^*(K(A,n);\widehat{\ZZ}_p)\otimes_{\widehat{\ZZ}_p} k)\cong H^*(K(A,n);\widehat{\ZZ}_p)\otimes_{\widehat{\ZZ}_p} k\cong \bigwedge^* V_n$.
\end{lemma}

\begin{proof}
It suffices to consider the cases $A=\widehat{\ZZ}_p$ and $A=\ZZ/p^{\alpha}$. 
The second case is clear. It remains to consider the case $A=\widehat{\ZZ}_p$.

Since $k$ is flat over $\widehat{\ZZ}_p$, the first isomorphism follows from the quasi-isomorphism defining the rectification. By \cite{Sullivan-MIT-notes}*{Theorem 3.9}, $
H^*(K(\widehat{\ZZ}_p,n);\widehat{\ZZ}_p)\cong  H^*(K(\ZZ,n);\ZZ)\otimes_{\ZZ} \widehat{\ZZ}_p
$. Since $H^*(K(\ZZ,n);\ZZ)\otimes_{\ZZ} \QQ\cong  H^*(K(\QQ,n);\QQ)
$, $H^*(K(\widehat{\ZZ}_p,n);\widehat{\ZZ}_p)\otimes_{\widehat{\ZZ}_p} k \cong (H^*(K(\ZZ,n);\ZZ)\otimes_{\ZZ} \widehat{\ZZ}_p)\otimes_{\widehat{\ZZ}_p} k \cong (H^*(K(\ZZ,n);\ZZ)\otimes_{\ZZ} \QQ) \otimes_{\QQ} k \cong H^*(K(\QQ,n);\QQ) \otimes_{\QQ}k$. Since $H^*(K(\QQ,n);\QQ)\cong \bigwedge^* W_n$, where $W_n=\homo_{\ZZ}(\ZZ,\QQ)$ is a graded $\QQ$-vector space concentrated in degree $n$, $H^*(K(\widehat{\ZZ}_p,n);\widehat{\ZZ}_p)\otimes_{\widehat{\ZZ}_p} k\cong \bigwedge^* V_n$, where $V_n=W_n\otimes_{\QQ}k\cong \homo_{\widehat{\ZZ}_p}(\widehat{\ZZ}_p,k)$.
\end{proof}

\begin{lemma}\label{Lemma: Serre spectral sequence}
Let $F\xrightarrow{h} E\xrightarrow{f} B$ be a fibration, and let $(E^{i,j}_r,d^r)$ denote its Leray-Serre spectral sequence, with $E^{i,j}_2=H^i(B;H^j(F;\widehat{\ZZ}_p))$. Assume that  the action of $\pi_1(B)$ on $H^*(F;\widehat{\ZZ}_p)$ is trivial and that $F$ is cohomologically $p$-complete finite type. Then $(E^{i,j}_r\otimes_{\widehat{\ZZ}_p} k,d^r)$ is a spectral sequence satisfying the following properties: 
\begin{enumerate}[leftmargin=0.25in]
    \item $E^{i,j}_2\otimes_{\widehat{\ZZ}_p} k\cong (H^i(B;\widehat{\ZZ}_p)\otimes_{\widehat{\ZZ}_p} k)\otimes_{k} (H^j(F;\widehat{\ZZ}_p)\otimes_{\widehat{\ZZ}_p} k)$;
    \item $(E^{i,j}_r\otimes_{\widehat{\ZZ}_p} k,d^r)$ converges to $H^{i+j}(E;\widehat{\ZZ}_p)\otimes_{\widehat{\ZZ}_p} k$.
\end{enumerate}
\end{lemma}

\begin{proof}
Since $k$ is a flat $\widehat{\ZZ}_p$-module, $(\{E^{p,q}_r,d^r\})\otimes_{\widehat{\ZZ}_p} k$ is a spectral sequence. 

\textbf{Item (1).} It suffices to prove that the canonical homomorphism $\phi:(H^i(B;\widehat{\ZZ}_p)\otimes_{\widehat{\ZZ}_p} k)\otimes_{k} (H^j(F;\widehat{\ZZ}_p)\otimes_{\widehat{\ZZ}_p} k)\rightarrow H^i(B;H^j(F;\widehat{\ZZ}_p))\otimes_{\widehat{\ZZ}_p} k$ is an isomorphism.
Since $H^j(F;\widehat{\ZZ}_p)$ is a finitely generated $\widehat{\ZZ}_p$-module, the proof reduces to the cases $H^j(F;\widehat{\ZZ}_p)=\widehat{\ZZ}_p$ and $H^j(F;\widehat{\ZZ}_p)=\ZZ/p^s$. The assertion is immediate for the free case. In the torsion case, both sides vanish after tensoring with $k$.

\textbf{Item (2).} The Leray-Serre filtration $F^*$ on the cohomology $H^*(E;\widehat{\ZZ}_p)$ gives a short exact sequence $0\rightarrow F^{i+1}H^{i+j}(E;\widehat{\ZZ}_p)\rightarrow F^{i}H^{i+j}(E;\widehat{\ZZ}_p)\rightarrow E^{i,j}_{\infty}\rightarrow 0$.
Tensoring this short exact sequence with the flat $\widehat{\ZZ}_p$-module $k$ preserves exactness. This proves item (2).
\end{proof}

\begin{definition}[\cite{McCleary-spectral-sequence}*{Section 6.2}]\label{Def: Transgression}
Retain the assumptions of Lemma \ref{Lemma: Serre spectral sequence}. Let $x\in C^i(F;\widehat{\ZZ}_p)\otimes_{\widehat{\ZZ}_p} k$ be a cocycle. The cohomology class $[x]$ is \textbf{transgressive} if there exist  $\widetilde{x}\in C^i(E;\widehat{\ZZ}_p)\otimes_{\widehat{\ZZ}_p} k$ and a cocycle $y\in C^{i+1}(B;\widehat{\ZZ}_p)\otimes_{\widehat{\ZZ}_p} k$ such that the restriction of $[\widetilde{x}]$ is $[x]$ and $f^*y=d\widetilde{x}$. The image of $[y]$ in $H^{i+1}(B;\widehat{\ZZ}_p)\otimes_{\widehat{\ZZ}_p} k/\Ker(f^*)$ is called the \textbf{transgression} of $[x]$.
\end{definition}

Consider the following commutative diagram of long exact cohomology sequences.
\[
\begin{tikzcd}
    ... \arrow[r] & H^n(pt;\widehat{\ZZ}_p)\otimes_{\widehat{\ZZ}_p} k \arrow[r] \arrow[d] & H^{n+1}(B,pt;\widehat{\ZZ}_p)\otimes_{\widehat{\ZZ}_p} k \arrow[r] \arrow[d,"f^*"] & H^{n+1}(B;\widehat{\ZZ}_p)\otimes_{\widehat{\ZZ}_p} k \arrow[r] \arrow[d,"f^*"] & ... \\
    ... \arrow[r] & H^n(F;\widehat{\ZZ}_p)\otimes_{\widehat{\ZZ}_p} k \arrow[r,"\partial"] \arrow[rru,dashed] & H^{n+1}(E,F;\widehat{\ZZ}_p)\otimes_{\widehat{\ZZ}_p} k \arrow[r] & H^{n+1}(E;\widehat{\ZZ}_p)\otimes_{\widehat{\ZZ}_p} k \arrow[r] & ...
\end{tikzcd}
\]
The class $[x]$ is transgressive precisely when it is in $\partial^{-1}(\Im(f^*))$. In that case, the dashed arrow gives its transgression.

\begin{lemma}[\cite{McCleary-spectral-sequence}*{Theorem 6.8}]
Under the assumptions of Lemma \ref{Lemma: Serre spectral sequence},
\begin{enumerate}[leftmargin=0.25in]
    \item $E^{n+1,0}_{n+1}\otimes_{\widehat{\ZZ}_p} k\cong H^{n+1}(B;\widehat{\ZZ}_p)\otimes_{\widehat{\ZZ}_p} k/\Ker(f^*)$;
    \item $E^{0,n}_{n+1}\otimes_{\widehat{\ZZ}_p} k\cong \partial^{-1}(\Im(f^*))$;
    \item under these identifications, the differential $d^{n+1}:E^{0,n}_{n+1}\otimes_{\widehat{\ZZ}_p} k\rightarrow E^{n+1,0}_{n+1}\otimes_{\widehat{\ZZ}_p} k$ agrees with the transgression map.
\end{enumerate}
\end{lemma}

Let $B_s$ denote the $s$-th skeleton of $B$, and let $E_s:=f^{-1}(B_s)$. The spectral sequence $(E^{*,*}_*,d)$ of Lemma \ref{Lemma: Serre spectral sequence} arises from the decreasing filtration $F^*$ on the singular cochain complex given by $F^sS^*(E;\widehat{\ZZ}_p):=\Ker(S^*(E;\widehat{\ZZ}_p)\rightarrow S^*(E_{s-1};\widehat{\ZZ}_p))$.

\begin{lemma}\label{Lemma: identification of spectral sequences tensoring with a field}
Retain the assumptions of Lemma \ref{Lemma: Serre spectral sequence}, and let $A=S^*(E;\widehat{\ZZ}_p)$, equipped with the filtration $F^*$ induced from the skeletons of $B$, as above. Then the spectral sequence associated with the filtration $F^s(A\otimes_{\widehat{\ZZ}_p} k)=\Ker(S^*(E;\widehat{\ZZ}_p)\otimes_{\widehat{\ZZ}_p} k\rightarrow S^*(E_{s-1};\widehat{\ZZ}_p)\otimes_{\widehat{\ZZ}_p} k)$ on $A\otimes_{\widehat{\ZZ}_p} k=S^*(E;\widehat{\ZZ}_p)\otimes_{\widehat{\ZZ}_p} k$ is canonically isomorphic to the Leray-Serre spectral sequence $(E^{*,*}_*\otimes_{\widehat{\ZZ}_p} k,d)$ as in Lemma \ref{Lemma: Serre spectral sequence}.
\end{lemma} 

\begin{proof}
Since $k$ is a flat $\widehat{\ZZ}_p$-module, $\Ker(S^*(E;\widehat{\ZZ}_p)\rightarrow S^*(E_{s-1};\widehat{\ZZ}_p))\otimes_{\widehat{\ZZ}_p}k$ is canonically isomorphic to $\Ker(S^*(E;\widehat{\ZZ}_p)\otimes_{\widehat{\ZZ}_p}k\rightarrow S^*(E_{s-1};\widehat{\ZZ}_p)\otimes_{\widehat{\ZZ}_p}k)$. Then the filtration $F^s(A\otimes_{\widehat{\ZZ}_p} k)$ on $A\otimes_{\widehat{\ZZ}_p} k$ is isomorphic to the filtration $F^sA \otimes_{\widehat{\ZZ}_p} k$. It suffices to show that the spectral sequence $(E^{*,*}_*\otimes_{\widehat{\ZZ}_p} k,d)$ is isomorphic to the spectral sequence associated with the filtration $F^*A \otimes_{\widehat{\ZZ}_p} k$ on $A \otimes_{\widehat{\ZZ}_p} k$.

Recall that $E^{i,j}_r=\frac{\{x\in F^i A^{i+j}\vert dx\in F^{i+r}A^{i+j+1}\}}{\{x\in F^{i+1}A^{i+j}\vert dx\in F^{i+r}A^{i+j+1}\}+dF^{i-r+1}A^{i+j}\bigcap F^i A^{i+j+1}}$ (\cite{Morgan-hodge-theory}*{(1.4)}). The numerator is the pullback of the diagram $F^i A^{i+j}\xrightarrow{d} A^{i+j+1}\hookleftarrow F^{i+r}A^{i+j+1}$.  Since $k$ is a flat $\widehat{\ZZ}_p$-module, $E^{i,j}_r\otimes_{\widehat{\ZZ}_p}k =\frac{\text{``numerator''}\otimes_{\widehat{\ZZ}_p}k}{\text{``denominator''}\otimes_{\widehat{\ZZ}_p}k}$, and $\text{``pullback of a diagram''}\otimes_{\widehat{\ZZ}_p}k$ is the pullback of $\text{``diagram''}\otimes_{\widehat{\ZZ}_p}k$. Then  $\text{``numerator''}\otimes_{\widehat{\ZZ}_p}k=\{y\in F^i A^{i+j}\otimes_{\widehat{\ZZ}_p}k\vert dy\in F^{i+r}A^{i+j+1}\otimes_{\widehat{\ZZ}_p}k\}$. By an analogous argument, $\text{``denominator''}\otimes_{\widehat{\ZZ}_p}k$ is isomorphic to the replacement of every $F^*A^*$ with $F^*A^*\otimes_{\widehat{\ZZ}_p}k$. This finishes the proof.
\end{proof}

\begin{lemma}\label{Lemma: Comparison of spectral sequences for rectification}
Retain the assumptions of Lemma \ref{Lemma: Serre spectral sequence}.  Define $F^s(C^*(E;\widehat{\ZZ}_p)\otimes_{\widehat{\ZZ}_p} k)$ to be the homotopy fiber of $C^*(E;\widehat{\ZZ}_p)\otimes_{\widehat{\ZZ}_p} k\rightarrow C^*(E_{s-1};\widehat{\ZZ}_p)\otimes_{\widehat{\ZZ}_p} k$ in the category of cochain complexes, and let $(\widetilde{E}^{*,*}_*,d)$ be the spectral sequence associated with $(C^*(E;\widehat{\ZZ}_p)\otimes_{\widehat{\ZZ}_p} k,F^\bullet)$, following \cite{Lurie-higher-algebra}*{Section 1.2.2}. Then the rectification zig-zag of quasi-isomorphisms between $S^*(E;\widehat{\ZZ}_p)\otimes_{\widehat{\ZZ}_p} k$ and $C^*(E;\widehat{\ZZ}_p)\otimes_{\widehat{\ZZ}_p}k$ induces a zig-zag of morphisms between the spectral sequences $(E^{*,*}_r\otimes_{\widehat{\ZZ}_p} k)_{r\geq 1}$ and $(\widetilde{E}^{*,*}_r)_{r\geq 1}$, and these spectral-sequence morphisms are isomorphisms.
\end{lemma}

\begin{proof}
Let $S:=S^*(E;\widehat{\ZZ}_p)\otimes_{\widehat{\ZZ}_p} k$, $C:=C^*(E;\widehat{\ZZ}_p)\otimes_{\widehat{\ZZ}_p} k$, $S(s):=S^*(E_s;\widehat{\ZZ}_p)\otimes_{\widehat{\ZZ}_p} k$, and $C(s):=C^*(E_s;\widehat{\ZZ}_p)\otimes_{\widehat{\ZZ}_p} k$. The naturality of rectification (Theorem \ref{Fact: rectification}) gives a zig-zag of quasi-isomorphisms between cochain complexes $F^s S$ and $F^s C$ for every $s$.
By the construction of the spectral sequence in \cite{Lurie-higher-algebra}*{Definition 1.2.2.2 - Construction 1.2.2.6}, the zig-zag of morphisms between $F^s S$ and $F^s C$ induces a zig-zag of morphisms between  spectral sequences $(E^{*,*}_r\otimes_{\widehat{\ZZ}_p} k)_{r\geq 1}$ and $(\widetilde{E}^{*,*}_r)_{r\geq 1}$. Since each morphism in the zig-zag of morphisms between $F^s S$ and $F^s C$ is a quasi-isomorphism, the induced spectral-sequence morphisms are isomorphisms.
\end{proof}

\begin{lemma}\label{Lemma: fibration and elementary extension}
Retain the assumptions of Lemma \ref{Lemma: Serre spectral sequence}. Assume further that every class in $H^{n}(F;\widehat{\ZZ}_p)\otimes_{\widehat{\ZZ}_p} k$ is transgressive and that $H^*(F;\widehat{\ZZ}_p)\otimes_{\widehat{\ZZ}_p} k\cong \bigwedge^*V_n$, where $V_n$  is a finite-dimensional graded $k$-vector space concentrated in degree $n$. Then there is an elementary extension $D$ of $C^*(B;\widehat{\ZZ}_p)\otimes_{\widehat{\ZZ}_p} k$ by $V_n$ together with a CDGA morphism $\phi:D\rightarrow C^*(E;\widehat{\ZZ}_p)\otimes_{\widehat{\ZZ}_p} k$ satisfying the following properties:
\begin{enumerate}[leftmargin=0.25in]
    \item $\phi$ extends $f^*:C^*(B;\widehat{\ZZ}_p)\otimes_{\widehat{\ZZ}_p} k\rightarrow C^*(E;\widehat{\ZZ}_p)\otimes_{\widehat{\ZZ}_p} k$; 
    \item $\phi$ is a quasi-isomorphism.
\end{enumerate}
\end{lemma}

\begin{proof}
We retain the notation of Lemma \ref{Lemma: Serre spectral sequence} and Lemma \ref{Lemma: Comparison of spectral sequences for rectification}.

Choose a basis $x_1,...,x_r$ of $V_n\cong H^n(F;\widehat{\ZZ}_p)\otimes_{\widehat{\ZZ}_p} k$. Since each class $x_i$ is transgressive, there are $\widetilde{x}_i\in C^{n}(E;\widehat{\ZZ}_p)\otimes_{\widehat{\ZZ}_p} k$ and a cocycle $y_i\in C^{n+1}(B;\widehat{\ZZ}_p)\otimes_{\widehat{\ZZ}_p} k$ such that the restriction of $[\widetilde{x}_i]$ is $x_i$ and  $f^*y_i=d\widetilde{x}_i$.

Let $(D,d_D)$ be the elementary extension of $C^*(B;\widehat{\ZZ}_p)\otimes_{\widehat{\ZZ}_p} k$ by $V_n$ with the differential $d_D$ given by $d_D(x_i)=y_i$. Define a CDGA morphism $\phi:D\rightarrow C^*(E;\widehat{\ZZ}_p)\otimes_{\widehat{\ZZ}_p} k$ by extending $f^*$ and setting $\phi(x_i)=\widetilde{x}_i$. It remains to show that $\phi$ is a quasi-isomorphism.

Filter $D$ by $F^s D:=(C^{*\geq s}(B;\widehat{\ZZ}_p)\otimes_{\widehat{\ZZ}_p} k)\otimes_{k} \bigwedge^*V_n$. Let $\widetilde{\widetilde{E}}$ be the spectral sequence associated with this filtration. By the computations in \cite{MGriffiths-Morgan-rational-homotopy}*{p.~155}, $\widetilde{\widetilde{E}}^{i,j}_2=(H^i(B;\widehat{\ZZ}_p)\otimes_{\widehat{\ZZ}_p} k) \otimes_{k} (\bigwedge^*V_n)^j$. Moreover, this is an isomorphism of bigraded algebras: the multiplication on $\widetilde{\widetilde{E}}^{*,*}_2$ is induced by the cup product of $H^*(B;\widehat{\ZZ}_p)\otimes_{\widehat{\ZZ}_p} k$ and the multiplication of $(\bigwedge^*V_n)^*$. In particular, the multiplication induces an isomorphism $\widetilde{\widetilde{E}}^{i,0}_2\otimes_k \widetilde{\widetilde{E}}^{0,j}_2\rightarrow \widetilde{\widetilde{E}}^{i,j}_2$.

We will show that $\phi:D\rightarrow C^*(E;\widehat{\ZZ}_p)\otimes_{\widehat{\ZZ}_p} k$ respects the filtrations and induces an isomorphism of spectral sequences on the second page.

Recall from Lemma \ref{Lemma: identification of spectral sequences tensoring with a field} that $B_s$ denotes the $s$-th skeleton of $B$ and that $E_s=f^{-1}(B_s)$. Consider the following short exact sequence of CDGA's.
\[
0\rightarrow C^{*\geq s}(B;\widehat{\ZZ}_p)\otimes_{\widehat{\ZZ}_p} k\rightarrow C^{*}(B;\widehat{\ZZ}_p)\otimes_{\widehat{\ZZ}_p} k \rightarrow C^{*<s}(B;\widehat{\ZZ}_p)\otimes_{\widehat{\ZZ}_p} k \rightarrow 0
\]
Tensoring this sequence with $\bigwedge^* V_n$ preserves exactness.

Restriction from $B$ to $B_{s-1}$ induces a natural CDGA morphism $g_{s-1}:D=(C^*(B;\widehat{\ZZ}_p)\otimes_{\widehat{\ZZ}_p} k)\otimes_k \bigwedge^*V_n\rightarrow (C^*(B_{s-1};\widehat{\ZZ}_p)\otimes_{\widehat{\ZZ}_p} k)\otimes_k \bigwedge^*V_n$. Let $h_{s-1}:C^*(E;\widehat{\ZZ}_p)\otimes_{\widehat{\ZZ}_p} k\rightarrow C^*(E_{s-1};\widehat{\ZZ}_p)\otimes_{\widehat{\ZZ}_p} k$ be the natural CDGA morphism induced by $E_{s-1}\hookrightarrow E$. There is a natural morphism $a_{s-1}:(C^*(B_{s-1};\widehat{\ZZ}_p)\otimes_{\widehat{\ZZ}_p} k)\otimes_k \bigwedge^*V_n\rightarrow C^*(E_{s-1};\widehat{\ZZ}_p)\otimes_{\widehat{\ZZ}_p} k$ which makes the following diagram commute.
\[
\begin{tikzcd}
   D=(C^*(B;\widehat{\ZZ}_p)\otimes_{\widehat{\ZZ}_p} k)\otimes_k \bigwedge^*V_n \arrow[r,"g_{s-1}"] \arrow[d,"\phi"] & (C^*(B_{s-1};\widehat{\ZZ}_p)\otimes_{\widehat{\ZZ}_p} k)\otimes_k \bigwedge^*V_n \arrow[d,"a_{s-1}"] \\
   C^*(E;\widehat{\ZZ}_p)\otimes_{\widehat{\ZZ}_p} k \arrow[r,"h_{s-1}"] & C^*(E_{s-1};\widehat{\ZZ}_p)\otimes_{\widehat{\ZZ}_p} k
\end{tikzcd}
\]

Let $k_{s-1}:(C^*(B_{s-1};\widehat{\ZZ}_p)\otimes_{\widehat{\ZZ}_p} k)\otimes_k \bigwedge^*V_n\rightarrow (C^{*<s}(B_{s-1};\widehat{\ZZ}_p)\otimes_{\widehat{\ZZ}_p} k)\otimes_k \bigwedge^*V_n$ denote the natural projection. We now construct a chain map $l_{s-1}:(C^{*<s}(B_{s-1};\widehat{\ZZ}_p)\otimes_{\widehat{\ZZ}_p} k)\otimes_k \bigwedge^*V_n\rightarrow (C^*(B_{s-1};\widehat{\ZZ}_p)\otimes_{\widehat{\ZZ}_p} k)\otimes_k \bigwedge^*V_n$ such that $l_{s-1}\circ k_{s-1}$ induces the identity on cohomology.

Let $Z^{s-1}$ denote the subspace of cocycles in $C^{s-1}(B_{s-1};\widehat{\ZZ}_p)\otimes_{\widehat{\ZZ}_p} k$, and let $B^{s}$ be the subspace of exact cochains in $C^{s}(B_{s-1};\widehat{\ZZ}_p)\otimes_{\widehat{\ZZ}_p} k$. Choose a vector-space complement $L^s$ to $Z^{s-1}$ in  $C^{s-1}(B_{s-1};\widehat{\ZZ}_p)\otimes_{\widehat{\ZZ}_p} k$. Then the differential restricts to an isomorphism $L^s\rightarrow B^s$. Then we may write $C^{s-1}(B_{s-1};\widehat{\ZZ}_p)\otimes_{\widehat{\ZZ}_p} k=Z^{s-1}\oplus B^{s}$.
Define a chain map $C^{*<s}(B_{s-1};\widehat{\ZZ}_p)\otimes_{\widehat{\ZZ}_p} k\rightarrow C^{*}(B_{s-1};\widehat{\ZZ}_p)\otimes_{\widehat{\ZZ}_p} k$ by taking the identity for $*<s-1$ and projecting onto $Z^{s-1}$ for $*=s-1$. Tensoring this map with the identity on $\bigwedge^*V_n$ gives a chain map $l_{s-1}:(C^{*<s}(B_{s-1};\widehat{\ZZ}_p)\otimes_{\widehat{\ZZ}_p} k)\otimes_k \bigwedge^*V_n\rightarrow (C^{*}(B_{s-1};\widehat{\ZZ}_p)\otimes_{\widehat{\ZZ}_p} k)\otimes_k \bigwedge^*V_n$.
\medskip

\textbf{Claim.} $l_{s-1}\circ k_{s-1}$ induces the identity on cohomology.

\textit{Proof of the claim.} Define an auxiliary complex $\tau^s C^*(B_{s-1};\widehat{\ZZ}_p)\otimes_{\widehat{\ZZ}_p} k$ as follows: it is zero below degree $s-1$, its degree-$(s-1)$ term is $B^s$ and in higher degrees it agrees with $C^*(B_{s-1};\widehat{\ZZ}_p)\otimes_{\widehat{\ZZ}_p} k$. The quotient of $(C^{*}(B_{s-1};\widehat{\ZZ}_p)\otimes_{\widehat{\ZZ}_p} k)\otimes_k \bigwedge^*V_n$ by the image of $l_{s-1}\circ k_{s-1}$ is isomorphic to $(\tau^s C^*(B_{s-1};\widehat{\ZZ}_p)\otimes_{\widehat{\ZZ}_p} k)\otimes_k \bigwedge^*V_n$. Therefore, it suffices to show that $(\tau^s C^*(B_{s-1};\widehat{\ZZ}_p)\otimes_{\widehat{\ZZ}_p} k)\otimes_k \bigwedge^*V_n$ is acyclic. Filter $(\tau^s C^*(B_{s-1};\widehat{\ZZ}_p)\otimes_{\widehat{\ZZ}_p} k)\otimes_k \bigwedge^*V_n$ analogously to $D$. This induces a spectral sequence whose $E_2$-page is $(H^{*>s}(B_{s-1};\widehat{\ZZ}_p)\otimes_{\widehat{\ZZ}_p} k)\otimes_{k} \bigwedge^*V_n=0$.  This proves the claim.
\medskip

We now return to the proof of the lemma. Consider the following diagram, in which all solid arrows commute.
\[
\begin{tikzcd}
    D=(C^*(B;\widehat{\ZZ}_p)\otimes_{\widehat{\ZZ}_p} k)\otimes_k \bigwedge^*V_n \arrow[r,"pr_{s-1}"] \arrow[dd,bend right=80,"\phi"'] \arrow[d,"g_{s-1}"] & (C^{*<s}(B;\widehat{\ZZ}_p)\otimes_{\widehat{\ZZ}_p} k)\otimes_k \bigwedge^*V_n \arrow[d,"b_{s-1}"']  \\
    (C^*(B_{s-1};\widehat{\ZZ}_p)\otimes_{\widehat{\ZZ}_p} k)\otimes_k \bigwedge^*V_n \arrow[r,"k_{s-1}"] \arrow[dr,"a_{s-1}"']  & (C^{*<s}(B_{s-1};\widehat{\ZZ}_p)\otimes_{\widehat{\ZZ}_p} k)\otimes_k \bigwedge^*V_n \arrow[l,bend left=5, dashed, "l_{s-1}"]
    \\
   C^*(E;\widehat{\ZZ}_p)\otimes_{\widehat{\ZZ}_p} k \arrow[r,"h_{s-1}"] & C^*(E_{s-1};\widehat{\ZZ}_p)\otimes_{\widehat{\ZZ}_p} k 
\end{tikzcd}
\]
By the claim, $l_{s-1}\circ k_{s-1}$ is the identity on cohomology. Then this diagram shows that $h_{s-1}\circ \phi$ and $a_{s-1} \circ l_{s-1} \circ b_{s-1}\circ pr_{s-1}$ induce the same map on cohomology.

Every cochain complex over a field is a direct sum of its cohomology, equipped with the zero differential, and an acyclic cochain complex (\cite{Weibel-homological-algebra}*{Exercise 1.1.3}). Consequently, two chain maps between cochain complexes of vector spaces are chain homotopic whenever they induce the same map on cohomology. It follows that the two chain maps $h_{s-1}\circ \phi$ and $a_{s-1} \circ l_{s-1} \circ b_{s-1}\circ pr_{s-1}$ are chain homotopic. The chain homotopy therefore induces a chain map from the homotopy fiber of $D=(C^*(B;\widehat{\ZZ}_p)\otimes_{\widehat{\ZZ}_p} k)\otimes_k \bigwedge^*V_n\rightarrow (C^{*<s}(B;\widehat{\ZZ}_p)\otimes_{\widehat{\ZZ}_p} k)\otimes_k \bigwedge^*V_n$  to $F^s C^*(E;\widehat{\ZZ}_p)\otimes_{\widehat{\ZZ}_p} k$. Since $D=(C^*(B;\widehat{\ZZ}_p)\otimes_{\widehat{\ZZ}_p} k)\otimes_k \bigwedge^*V_n\rightarrow (C^{*<s}(B;\widehat{\ZZ}_p)\otimes_{\widehat{\ZZ}_p} k)\otimes_k \bigwedge^*V_n$ is degree-wise surjective, its homotopy fiber is represented by $F^sD=(C^{*>s}(B;\widehat{\ZZ}_p)\otimes_{\widehat{\ZZ}_p} k)\otimes_k \bigwedge^*V_n$. Then the chain map $\phi:D=(C^*(B;\widehat{\ZZ}_p)\otimes_{\widehat{\ZZ}_p} k)\otimes_k \bigwedge^*V_n\rightarrow C^*(E;\widehat{\ZZ}_p)\otimes_{\widehat{\ZZ}_p} k$ preserves filtrations and hence induces a spectral-sequence morphism $\widetilde{\widetilde{E}}^{*,*}_*\rightarrow \widetilde{E}^{*,*}_*$. 

By Lemma \ref{Lemma: Serre spectral sequence} and Lemma \ref{Lemma: Comparison of spectral sequences for rectification}, $\widetilde{E}^{i,0}_2= H^i(B;\widehat{\ZZ}_p)\otimes_{\widehat{\ZZ}_p}k$, $\widetilde{E}^{0,j}_2= H^j(F;\widehat{\ZZ}_p)\otimes_{\widehat{\ZZ}_p}k\cong \bigwedge^jV_n$ and $\widetilde{E}^{i,j}_2=\widetilde{E}^{i,0}_2\otimes_k \widetilde{E}^{0,j}_2$. Naturality with respect to $F\rightarrow pt$ shows that $\widetilde{\widetilde{E}}^{i,0}_2=H^i(B;\widehat{\ZZ}_p)\otimes_{\widehat{\ZZ}_p}k\rightarrow \widetilde{E}^{i,0}_2=H^i(B;\widehat{\ZZ}_p)\otimes_{\widehat{\ZZ}_p}k$ is the identity. Moreover, $\widetilde{\widetilde{E}}^{0,j}_2=\bigwedge^jV_n\rightarrow \widetilde{E}^{0,j}_2=H^j(F;\widehat{\ZZ}_p)\otimes_{\widehat{\ZZ}_p}k\cong \bigwedge^jV_n$ is clearly the identity. Then $\widetilde{\widetilde{E}}^{*,*}_2\rightarrow \widetilde{E}^{*,*}_2$ and therefore $\widetilde{\widetilde{E}}^{*,*}_r\rightarrow \widetilde{E}^{*,*}_r$ for every $r>2$ are isomorphisms. In particular, $\widetilde{\widetilde{E}}^{*,*}_{\infty}\rightarrow \widetilde{E}^{*,*}_{\infty}$ is an isomorphism. Applying the five lemma successively to the short exact sequences of filtration quotients shows, by induction on the filtration degree, that $\phi:D=(C^*(B;\widehat{\ZZ}_p)\otimes_{\widehat{\ZZ}_p} k)\otimes_k \bigwedge^*V_n\rightarrow C^*(E;\widehat{\ZZ}_p)\otimes_{\widehat{\ZZ}_p} k$ is a quasi-isomorphism.
\end{proof}

\begin{corollary}\label{Cor: Extension of Minimal Model for Fibrations}
Retain the assumptions of Lemma \ref{Lemma: fibration and elementary extension}. Let $A$ be a CDGA over $k$, equipped with a quasi-isomorphism $\psi:A\rightarrow C^*(B;\widehat{\ZZ}_p)\otimes_{\widehat{\ZZ}_p} k$. Then there is an elementary extension $D'$ of $A$ by $V_n$, together with a CDGA morphism $\phi':D'\rightarrow C^*(E;\widehat{\ZZ}_p)\otimes_{\widehat{\ZZ}_p} k$ satisfying the following properties:
\begin{enumerate}[leftmargin=0.25in]
    \item $\phi'$ extends $f^*\circ \psi$;
    \item $\phi'$ is a quasi-isomorphism.
\end{enumerate}
\end{corollary}

\begin{proof}
We use the notation introduced in the proof of Lemma \ref{Lemma: fibration and elementary extension}. For each $i$, choose a cocycle $z_i\in A^{n+1}$ whose cohomology class maps to $[y_i]$ under $\psi^*$. Form the elementary extension $(D',d_{D'})$ of $A$ by $V_n$ such that $d_{D'}(x_i)=z_i$.  Since $\psi(z_i)$ and $y_i$ represent the same cohomology class, there exists a cochain $y'_i\in C^{n}(B;\widehat{\ZZ}_p)\otimes_{\widehat{\ZZ}_p} k$ such that $\psi(z_i)=y_i+dy'_i$. Replace $y_i$ and $\widetilde{x}_i$ with $y_i+dy'_i$ and $\widetilde{x}_i+f^*y'_i$, respectively. These replacements preserve the relation $d\widetilde{x}_i=f^*y_i$, and we may therefore assume that $\psi(z_i)=y_i$.

Define $\psi':D'=A\otimes_{k} \bigwedge^* V_n\rightarrow D=(C^*(B;\widehat{\ZZ}_p)\otimes_{\widehat{\ZZ}_p} k)\otimes_{k} \bigwedge^* V_n$ to be the CDGA morphism which restricts to $\psi$ on $A$ and restricts to the identity on $V_n$. Define $\phi'=\phi\circ \psi'$.

Recall that $D$ is associated with the filtration $F^s D=(C^{*\geq s}(B;\widehat{\ZZ}_p)\otimes_{\widehat{\ZZ}_p} k)\otimes_{k}\bigwedge^* V_n$. The associated spectral sequence $\widetilde{\widetilde{E}}^{*,*}_*$ has $\widetilde{\widetilde{E}}^{i,j}_2=(H^i(B;\widehat{\ZZ}_p)\otimes_{\widehat{\ZZ}_p}k)\otimes_{k} (\bigwedge^* V_n)^j$. Similarly, filter $D'$ by $F^s D'=A^{*\geq s}\otimes_{k}\bigwedge^*V_n$. Its associated spectral sequence $\widetilde{\widetilde{E}}'^{*,*}_*$ has $\widetilde{\widetilde{E}}'^{i,j}_2=H^i(A)\otimes_k (\bigwedge^* V_n)^j\cong (H^i(B;\widehat{\ZZ}_p)\otimes_{\widehat{\ZZ}_p}k)\otimes_{k} (\bigwedge^* V_n)^j$. The morphism $\psi'$ respects these filtrations, and therefore it induces a spectral-sequence morphism $\widetilde{\widetilde{E}}'^{*,*}_*\rightarrow \widetilde{\widetilde{E}}^{*,*}_*$. Since $\widetilde{\widetilde{E}}'^{i,0}_2=H^i(A)\rightarrow \widetilde{\widetilde{E}}^{i,0}_2=H^i(B;\widehat{\ZZ}_p)\otimes_{\widehat{\ZZ}_p}k$ and $\widetilde{\widetilde{E}}'^{0,j}_2=\wedge^j V_n\rightarrow \widetilde{\widetilde{E}}^{0,j}_2=\wedge^j V_n$ are both isomorphisms, $\widetilde{\widetilde{E}}'^{*,*}_2\rightarrow \widetilde{\widetilde{E}}^{*,*}_2$ is an isomorphism. The spectral-sequence comparison used in Lemma \ref{Lemma: fibration and elementary extension} now shows that $\psi'$ is a quasi-isomorphism. Since $\phi$ is also a quasi-isomorphism, the composition $\phi'=\phi\circ \psi'$ is a quasi-isomorphism.
\end{proof}

The following result provides an explicit  correspondence between elementary extensions of minimal models and the $k$-invariants in the Postnikov tower.

\begin{corollary}\label{Cor: k-invariant computation}
Retain the assumptions of Lemma \ref{Lemma: Serre spectral sequence}. Assume further that $F=K(G,n)$, where $G$ is a finitely generated $\widehat{\ZZ}_p$-module. Choose a basis $\{x^{\vee}_1,...,x^{\vee}_r\}$ of $G\otimes_{\widehat{\ZZ}_p} k$. Using the duality of Lemma \ref{Lemma: Computation of Eilenberg-Maclane spaces}, let $\{x_1,...,x_r\}$ be the corresponding dual basis of $H^n(K(G,n);\widehat{\ZZ}_p)\otimes_{\widehat{\ZZ}_p} k$. Let $y_i$ denote the transgression of $x_i$ with the choices made in the proof of Lemma \ref{Lemma: fibration and elementary extension}. Let $K \in H^{n+1}(B;G)\otimes_{\widehat{\ZZ}_p} k$ be the $k$-invariant of the fibration $K(G,n)\rightarrow E \rightarrow B$.  Then, under the natural identification $H^{n+1}(B;G)\otimes_{\widehat{\ZZ}_p} k\cong (H^{n+1}(B;\widehat{\ZZ}_p)\otimes_{\widehat{\ZZ}_p} k)\otimes_{k} (G\otimes_{\widehat{\ZZ}_p} k)$, one has $K=y_1\otimes x^{\vee}_1+...+y_r\otimes x^{\vee}_r$.
\end{corollary}

\begin{proof}
Consider the following pullback diagram of fibration.
\[
\begin{tikzcd}
    K(G,n) \arrow[d,equal] 
    \arrow[r] &  E \arrow[d] \arrow[r,"f"] & B \arrow[d,"K"] \\
    K(G,n) \arrow[r] &
    pt \arrow[r] 
     & K(G,n+1)
\end{tikzcd}
\]
By the naturality of the $k$-invariant and the transgression, it suffices to verify the formula for the universal fibration $K(G,n)\rightarrow pt\rightarrow K(G,n+1)$. Under the standard identification of $H^{n+1}(K(G,n+1);G)$ with $\homo(G,G)$, its $k$-invariant corresponds to the identity of $G$. The asserted formula then follows for the universal fibration.
\end{proof}

\subsubsection{Proof of Theorem \ref{Thm: minimal model computes rational homotopy groups}}\,

Choose a principal refined Postnikov tower  $X\rightarrow ... \rightarrow X^3_{0}=X^2_{t_2} \rightarrow ... \rightarrow X^2_{0}=X^1_{t_1}\rightarrow ...\rightarrow X^1_{1}\rightarrow X^1_0=pt$ (\cite{more-concise}*{Theorem 3.2.2}) for $X$. For $1\leq s\leq t_1$, the fiber of $X^1_{s}\rightarrow X^1_{s-1}$ is $K(\Gamma^{s}\pi/\Gamma^{s+1}\pi,1)$; for $n\geq 2$ and $1\leq s\leq t_n$, the fiber of $X^n_s\rightarrow X^n_{s-1}$ is $K(A_{n,s},n)$ for $n\geq 2$, where $A_{n,s}$  is a finitely generated $\widehat{\ZZ}_p$-module. Consequently, for $n\geq 2$, $\pi_n(X^n_s)\cong \bigoplus_{j=1}^{s} A_{n,j}$. For uniformity, set $A_{1,s}=\Gamma^{s}\pi/\Gamma^{s+1}\pi$. Lemma \ref{Lemma: Computation of Eilenberg-Maclane spaces} and Lemma \ref{Lemma: Serre spectral sequence} give natural identifications $H^1(X^1_s;\widehat{\ZZ}_p)\otimes_{\widehat{\ZZ}_p}k \cong \homo_{\widehat{\ZZ}_p}(\bigoplus_{j=1}^sA_{1,j},k)$. Using Corollary \ref{Cor: Extension of Minimal Model for Fibrations} inductively and applying
the minimal model construction successively to the refined Postnikov tower gives a compatible tower of minimal CDGA's
$k=M(1,0)\subset M(1,1)\subset ...\subset M(1,t_1)=M(2,0)\subset M(2,1)\subset ...$.

\textbf{Item (1).} We proceed inductively along the refined Postnikov tower. For every pair $(n,s)$, we construct a natural nondegenerate pairing between $I^n(M(n,s))$ and $(\bigoplus_{j=1}^{s} A_{n,j})\otimes_{\widehat{\ZZ}_p} k$.

Let $\phi_{1,1}:M(1,1)\rightarrow C^*(X^1_1;\widehat{\ZZ}_p)\otimes_{\widehat{\ZZ}_p} k$ be a minimal model. By Lemma \ref{Lemma: Computation of Eilenberg-Maclane spaces}, $M(1,1)=\bigwedge^* V_{1,1}$ with $I^1(M(1,1))=V_{1,1}=H^1(X^1_1;\widehat{\ZZ}_p)\otimes_{\widehat{\ZZ}_p} k\cong \homo_{\widehat{\ZZ}_p}(A_{1,1},k)$. Define the pairing $
\langle -,- \rangle: I^1(M(1,1))\otimes_{k} (A_{1,1}\otimes_{\widehat{\ZZ}_p} k)\rightarrow k
$
by $\langle x,\alpha \rangle=\phi_{1,1}(x)(h(\alpha))$, where $h$ is the Hurewicz homomorphism $A_{1,1}\cong \pi_1(X^1_1)\rightarrow H_1(X^1_1;\ZZ)$. Clearly this pairing is nondegenerate.

Assume inductively that, for some $s<t_n$, a minimal model $\phi_{n,s}:M(n,s)\rightarrow C^*(X^n_s;\widehat{\ZZ}_p)\otimes_{\widehat{\ZZ}_p} k$  has been constructed and that
a pairing $
\langle -,- \rangle: I^n(M(n,s))\otimes_{k} ((\bigoplus_{j=1}^{s} A_{n,j})\otimes_{\widehat{\ZZ}_p} k)
$
by $\langle x,\alpha \rangle=\langle f_{\alpha}^*\phi_{n,s}(x),[S^{n}]\rangle$ has been proved to be nondegenerate, where the map $f_{\alpha}:S^n\rightarrow X$ represents the class $\alpha$ and $[S^n]$ is the fundamental class of $S^n$. 

Now consider the next fibration $K(A_{n,s+1},n)\rightarrow X^n_{s+1}\rightarrow X^{n}_{s}$ in the refined Postnikov tower. Lemma \ref{Lemma: Computation of Eilenberg-Maclane spaces} identifies $H^*(K(A_{n,s+1},n);\widehat{\ZZ}_p)\otimes_{\widehat{\ZZ}_p} k$ with $ \bigwedge^*V_{n,s+1}$, where $V_{n,s+1}=\homo_{\widehat{\ZZ}_p}(A_{n,s+1},k)$. Corollary \ref{Cor: Extension of Minimal Model for Fibrations} gives an elementary extension $M(n,s+1)$ of $M(n,s)$ by $V_{n,s+1}$, together with a quasi-isomorphism $\phi_{n,s+1}:M(n,s+1)\rightarrow C^*(X^m_t;\widehat{\ZZ}_p)\otimes_{\widehat{\ZZ}_p} k$ extending $\phi_{n,s}$. Then $I^{n}(M(n,s+1))\cong\bigoplus_{j=1}^{s+1} V_{n,j}\cong\bigoplus_{j=1}^{s+1} H^{n}(K(A_{n,j},n);\widehat{\ZZ}_p)\otimes_{\widehat{\ZZ}_p} k\cong \bigoplus_{j=1}^{s+1}\homo_{\widehat{\ZZ}_p}(A_{n,j}, k)$. Define the pairing 
$
\langle -,- \rangle: I^{n}(M(n,s+1))\otimes_{k} ((\bigoplus_{j=1}^{s+1}A_{n,j})\otimes_{\widehat{\ZZ}_p} k)\rightarrow k
$
by $\langle x,\alpha \rangle=\langle f_{\alpha}^*\phi_{n,s+1}(x),[S^{n}]\rangle$, where the map $f_{\alpha}:S^n\rightarrow X$ represents the class $\alpha$. On the previously constructed summands, nondegeneracy follows from the inductive hypothesis. On the new summand, it follows from Lemma \ref{Lemma: Computation of Eilenberg-Maclane spaces}. Hence this enlarged pairing is nondegenerate.

\textbf{Item (2).} Let $m,n\geq 2$ and consider the wedge $Y=S^{m}\bigvee S^n$. Let $\alpha_m\in \pi_m(S^{m}\bigvee S^n)$ and $\alpha_n\in \pi_n(S^{m}\bigvee S^n)$ be the homotopy classes of the canonical inclusions $S^m\hookrightarrow S^{m}\bigvee S^n$ and $S^n\hookrightarrow S^{m}\bigvee S^n$, respectively. Write $\beta:=[\alpha_m,\alpha_n] \in \pi_{m+n-1}(S^{m}\bigvee S^n)$ for their Whitehead product. Let $\phi:M\rightarrow C^*(X;\widehat{\ZZ}_p)\otimes_{\widehat{\ZZ}_p} k$ be a minimal model, where $M\cong \bigwedge^* (\bigoplus_s V_s)$ with each $V_s$ a finite-dimensional $k$-vector space concentrated in degree $s$. By item (1), for every $s\geq 2$, $V_s$ is naturally dual to $\pi_s(X)\otimes_{\widehat{\ZZ}_p} k$. Choose a basis $v_{s,1},...,v_{s,t_s}$ of each $V_s$. It suffices to show that, for every map $f:Y\rightarrow X$ such that $f_*\alpha_m$ and $f_*\alpha_n$ are dual to $v_{m,i}$ and  $v_{n,j}$, respectively, and for every $w\in V_{m+n-1}$, $\text{``the coefficient of $v_{m,i}\cdot v_{n,j}$ in $d(w)$''}=\langle w, f_*(\beta)\rangle$.

Let $\phi_Y:M_Y\rightarrow C^*(Y;\widehat{\ZZ}_p)\otimes_{\widehat{\ZZ}_p} k$ be a minimal model. Then $f:Y\rightarrow X$ induces a CDGA morphism $f':M\rightarrow M_Y$, unique up to homotopy (Theorem \ref{Fact: Lifting of minimal models}). Applying $f'$ reduces the desired identity to the corresponding identity for $Y$. The minimal model $M_Y$ is obtained from the rational minimal model of $Y$ by extending scalars from $\QQ$ to $k$. For the rational minimal model of a space, the quadratic differential is dual to the Whitehead product, which is proved in \cite{Andrews-Arkowitz-minimal-model-Whitehead-product}*{Theorem 5.4}. Extending the scalars to $k$ preserves the identity. This proves item (2).

\textbf{Item (3).} It suffices to prove the statement when $X=K(G,1)$, where $G$ is a topologically finitely generated nilpotent pro-$p$ group. We will identify
the Lie algebra dual to the $1$-minimal model of $C^*(K(G,1);\widehat{\ZZ}_p)\otimes_{\widehat{\ZZ}_p}k$ with the Lie algebra $\Lie(G\widehat{\otimes}k)$ of the continuous Mal'cev $k$-completion of $G$. By Proposition \ref{Prop: base change of Malcev completion} the proof reduces to the case of $k=\QQ_p$. Replacing $G$ by $G/\Tor(G)$ does not change its continuous Mal'cev $\QQ_p$-completion by \cite{Hu-Wang-continuou-Malcev-completion}*{Proposition 4.12}. It also does not change the $\QQ_p$-valued minimal-model data relevant here. We may therefore assume that $G$ is torsion free. We argue by induction on the nilpotency class $c$ of $G$.

If $c=1$, then $G$ is abelian. Both constructions give the vector space $G\otimes_{\widehat{\ZZ}_p}\QQ_p$. Suppose inductively that the result is known for groups of nilpotency class less than $c$. Let $Z=Z(G)$. Then $G$ fits into the continuous central extension $1\rightarrow Z\rightarrow G\rightarrow G/Z\rightarrow 1$, where $G/Z$ has smaller nilpotency class. Let $M_Z$, $M_G$ and $M_{G/Z}$ denote $1$-minimal models for $K(Z,1)$, $K(G,1)$ and $K(G/Z,1)$ over $\QQ_p$, respectively. 
Corollary \ref{Cor: Extension of Minimal Model for Fibrations} identifies $M_G$ with an elementary extension of $M_{G/Z}$ by $M^1_Z$.
Dualizing this elementary extension gives a central extension of Lie algebras $0\rightarrow Z\otimes_{\widehat{\ZZ}_p}\QQ_p\rightarrow L_G\rightarrow L_{G/Z}\rightarrow 0$. By Lemma \ref{Lem: Correspondence between Lie algebra cohomology and minimal model}, $M_G$ and $M_{G/Z}$ are the Chevalley-Eilenberg complexes of $L_G$ and $L_{G/Z}$, respectively. 

By the inductive hypothesis, there is a natural isomorphism $L_{G/Z}\cong\Lie((G/Z)\widehat{\otimes}\QQ_p)$. 
It remains to prove that the extension $0\rightarrow Z\otimes_{\widehat{\ZZ}_p}\QQ_p\rightarrow L_G\rightarrow L_{G/Z}\rightarrow 0$ is equivalent to the central extension $0\rightarrow Z\otimes_{\widehat{\ZZ}_p}\QQ_p\rightarrow \Lie(G\widehat{\otimes}\QQ_p)\rightarrow\Lie((G/Z)\widehat{\otimes}\QQ_p)=L_{G/Z}\rightarrow 0$ given by Lemma \ref{Lem: Malcev completion preserves exactness of nilpotent groups}.

Choose a $\widehat{\ZZ}_p$-basis $\{z_1,...,z_r\}$ of $Z$ and extend it to a Mal'cev basis $\{z_1,...,z_r,y_1,...,y_s\}$ of $G$. The image $\{\overline{y}_1,...,\overline{y}_s\}$ then is a Mal'cev basis of $G/Z$. Let $\{u_1,...,u_r\}$ be the basis  of $ \homo_{\widehat{\ZZ}_p}(Z,\widehat{\ZZ}_p)$ dual to $\{z_1,...,z_r\}$. After extension of scalars, $\{u_1,...,u_r\}$ also forms a $\QQ_p$-basis of $M^1_Z= \homo_{\widehat{\ZZ}_p}(Z,\widehat{\ZZ}_p)\otimes_{\widehat{\ZZ}_p}\QQ_p$. By item (1), we may choose a $\QQ_p$-basis $\{\overline{v}_1,...,\overline{v}_s\}$ of $M^1_{G/Z}$ dual to $\overline{y_1},...,\overline{y_s}$, respectively. Choose liftings $v_1,...,v_s\in M_G^1$ of $\overline{v}_1,...,\overline{v}_s$. The differential of each $v_j$ is a lifting of the corresponding differential in $M_{G/Z}$, while each $d(u_i)$ is a quadratic expression in $v_1,...,v_s$.

Let $(f_1,...,f_s)$ be the formal group law for $G/Z$ in the chosen Mal'cev basis \cite{Hu-Wang-continuou-Malcev-completion}*{Proposition 4.9}. \cite{Hu-Wang-continuou-Malcev-completion}*{Proposition 4.9, Lemma 4.10} provide polynomials $h_1,...,h_r\in \QQ_p[a_1,...,a_s,b_1,...,b_s]$ such that $(y_1^{\alpha_1}...y_s^{\alpha_s})\cdot (y_1^{\beta_1}...y_s^{\beta_s})=(z_1^{h_1(\underline{\alpha},\underline{\beta})}...z_r^{h_r(\underline{\alpha},\underline{\beta})})\cdot (y_1^{f_1(\underline{\alpha},\underline{\beta})}...y_s^{f_s(\underline{\alpha},\underline{\beta})})$ for all $\underline{\alpha}=(\alpha_1,...,\alpha_s),\underline{\beta}=(\beta_1,...,\beta_s)\in \widehat{\ZZ}_p^s$. Then, in these coordinates, multiplication of $G$ takes the form $(z_1^{\delta_1}...z_r^{\delta_r}y_1^{\alpha_1}...y_s^{\alpha_s})\cdot (z_1^{\gamma_1}...z_r^{\gamma_r}y_1^{\beta_1}...y_s^{\beta_s})=z_1^{\delta_1+\gamma_1+h_1(\underline{\alpha},\underline{\beta})}...z_r^{\delta_r+\gamma_r+h_r(\underline{\alpha},\underline{\beta})}y_1^{f_1(\underline{\alpha},\underline{\beta})}...y_s^{f_s(\underline{\alpha},\underline{\beta})}$. Let $h_{i,2}$ and $f_{j,2}$ denote the homogeneous quadratic parts of $h_i$ and $f_j$, respectively. The Lie bracket of $\Lie(G\widehat{\otimes}\QQ_p)$ is then $[(\underline{\delta},\underline{\alpha}),(\underline{\gamma},\underline{\beta})]_G=(h_{i,2}(\underline{\alpha},\underline{\beta})-h_{i,2}(\underline{\beta},\underline{\alpha}),f_{i,2}(\underline{\alpha},\underline{\beta})-f_{i,2}(\underline{\beta},\underline{\alpha}))$ for all $\underline{\delta},\underline{\gamma}\in \QQ_p^r\cong Z\otimes_{\widehat{\ZZ}_p}\QQ_p$ and $\underline{\alpha},\underline{\beta}\in \QQ_p^s\cong \Lie(G/Z)$ (\cite{Serre-Lie-algebra-Lie-group}*{p.~112}). For every $i$, define $k_i(\underline{\alpha}\wedge \underline{\beta})=h_{i,2}(\underline{\alpha},\underline{\beta})-h_{i,2}(\underline{\beta},\underline{\alpha})$. Then $(k_1,...,k_r):\bigwedge^2_{\QQ_p}(\Lie(G/Z))\rightarrow Z\otimes_{\widehat{\ZZ}_p}\QQ_p$ is the Chevalley-Eilenberg $2$-cocycle classifying the central extension $0\rightarrow Z\otimes_{\widehat{\ZZ}_p}\QQ_p \rightarrow \Lie(G\widehat{\otimes}\QQ_p)\rightarrow\Lie((G/Z)\widehat{\otimes}\QQ_p)=L_{G/Z}\rightarrow 0$.

As in the proof of Proposition \ref{Prop: injectivity of colimit of lower central series into group}, the Lyndon-Hochschild-Serre spectral sequence for $1\rightarrow Z\rightarrow G\rightarrow G/Z\rightarrow 1$ induces an exact sequence:
\[
0\rightarrow H^1_{\cont}(Z;\widehat{\ZZ}_p)=\homo_{\widehat{\ZZ}_p}(Z,\widehat{\ZZ}_p)\xrightarrow{\partial} H^2_{\cont}(G/Z;\widehat{\ZZ}_p) \rightarrow H^2_{\cont}(G;\widehat{\ZZ}_p)
\]
In the chosen basis,
the cohomology class $\partial(u_i)$ is represented by the $2$-cocycle $h_i\in C^2_{\cont}(G/Z;\widehat{\ZZ}_p)$. Since the continuous group cohomology and the ordinary group cohomology of topologically finitely generated nilpotent pro-$p$ groups agree (Proposition \ref{Prop: comparison between continuous and discrete cohomology for Serre good groups} and Lemma \ref{Lem: relation between continuous and discrete cohomology for nilpotent pro-p groups}), the connecting homomorphism $\partial$ agrees with the transgression map for the fibration $K(Z,1)\rightarrow K(G,1)\rightarrow K(G/Z,1)$. Differentiating the skew-symmetrized group cocycle $G/Z\times G/Z\rightarrow \QQ_p$ given by $(\underline{\alpha},\underline{\beta})\rightarrow h_i(\underline{\alpha},\underline{\beta})- h_i(\underline{\beta},\underline{\alpha})$ yields the $\QQ_p$-bilinear map $k_i:\bigwedge^2_{\QQ_p}(\Lie(G/Z))\rightarrow \QQ_p$. $k_i$ is a $2$-cocycle in the Chevalley-Eilenberg complex of $\Lie(G/Z)$ since $h_i$ is a $2$-cocycle. Under the identification of $M_{G/Z}$ with the Chevalley-Eilenberg complex of $\Lie(G/Z)$, the transgression map $M^1_Z\rightarrow M^2_{G/Z}$ sends $u_i$ to $k_i$, which gives the elementary extension $M_G$ of $M_{G/Z}$. Then the $2$-cocycle $\bigwedge^2_{\QQ_p}(\Lie(G/Z))\rightarrow Z\otimes_{\widehat{\ZZ}_p}\QQ_p$ corresponding to the central extension of Lie algebras $0\rightarrow Z\otimes_{\widehat{\ZZ}_p}\QQ_p\rightarrow L_G\rightarrow \Lie(G/Z)\rightarrow 0$ is $(k_1,...,k_r)$. Since $L_{G/Z}\cong\Lie(G/Z)=\Lie((G/Z)\widehat{\otimes}\QQ_p)$, this central extension is equivalent to $0\rightarrow Z\otimes_{\widehat{\ZZ}_p}\QQ_p\rightarrow \Lie(G\widehat{\otimes}\QQ_p)\rightarrow\Lie((G/Z)\widehat{\otimes}\QQ_p)=L_{G/Z}\rightarrow 0$. This completes the proof of item (3).
\qed

\subsection{Pro $p$-Finite Spaces and Non-nilpotent $p$-Complete Spaces}\,

For applications in Section \ref{Section: Application}, we prove a result (Corollary \ref{Cor: duality between minimal model and homotopy groups for pro spaces}) which is analogous to Theorem \ref{Thm: minimal model computes rational homotopy groups} for pro $p$-finite spaces. We also discuss the relation between minimal models and the continuous Mal'cev $\QQ_p$-completion of the fundamental group for non-nilpotent $p$-complete spaces (Theorem \ref{Thm: 1-minimal model is Malcev completion}).

Throughout this section, let $k$ be a field containing $\QQ_p$. Recall that a small category $\mathcal{I}$ is \textbf{sifted} if the diagonal functor $\Delta:\mathcal{I}\rightarrow \mathcal{I}\times \mathcal{I}$ is final.

\begin{definition}
Let $\{X_i\}$ be a pro system of simplicial sets. Define its \textbf{simplicial cochain complex} $S^*(\{X_i\};\ZZ/p^a):=\varinjlim_i S^*(X_i;\ZZ/p^a)$, and $S^*(\{X_i\};\widehat{\ZZ}_p):=\varprojlim_{a} S^*(\{X_i\};\ZZ/p^a)$.
\end{definition}

\begin{lemma}\label{Lem: O-algebras have sifted colimits}
Let $R$ be a commutative ring, and let $\mathcal{O}$ be an $R$-operad.  Then the category $\mathbf{Alg}(\mathcal{O})$ of $\mathcal{O}$-algebras admits all sifted colimits.
\end{lemma}

\begin{proof}
Let $\{A_i\}_{i\in I}$ be a sifted diagram of $\mathcal{O}$-algebras. Let $A$ be its colimit in the category $\mathbf{Ch}(R)$ of cochain complexes of $R$-modules. It suffices to prove that, for every $n$, $\mathcal{O}(n)\otimes_{\Sigma_n} A^{\otimes n}=\varinjlim_{i\in I}(\mathcal{O}(n)\otimes_{\Sigma_n} A^{\otimes n}_i)$, since this isomorphism allows the $\OO$-algebra structures on $A_i$'s to induce an $\OO$-algebra structure on $A$. Since $I$ is sifted, the diagonal functor $\Delta:I\rightarrow I\times ...\times I$ is final. Then $A^{\otimes n}=(\varinjlim_{i\in I}A_i)^{\otimes n}=\varinjlim_{i_1,...,i_n} A_{i_1}\otimes ... \otimes A_{i_n}= \varinjlim_{i\in I} A_i^{\otimes n}$. The functor $\mathcal{O}(n)\otimes_{R} (-)$ is a left adjoint and therefore preserves all small colimits. Hence, $\mathcal{O}(n)\otimes A^{\otimes n}=\varinjlim_{i\in I}(\mathcal{O}(n)\otimes A_i^{\otimes n})$. Taking $\Sigma_n$-coinvariants is also a colimit and therefore commutes with colimits. Thus,  $\mathcal{O}(n)\otimes_{\Sigma_n} A^{\otimes n}=\varinjlim_{i\in I}(\mathcal{O}(n)\otimes_{\Sigma_n} A_i^{\otimes n})$.
\end{proof}

\begin{lemma}
Let $\{X_i\}$ be a pro system of simplicial sets. Then the simplicial cochain complexes $S^*(\{X_i\};\ZZ/p^a)$ and $S^*(\{X_i\};\widehat{\ZZ}_p)$ are $E_{\infty}$-algebras over $\ZZ/p^a$ and over $\widehat{\ZZ}_p$, respectively.
\end{lemma}

\begin{proof}
Choose a cofibrant operad $\mathcal{E}$ over $\widehat{\ZZ}_p$. Then each $S^*(X_i;\ZZ/p^a)$ is naturally an $\mathcal{E}$-algebra. By Lemma \ref{Lem: O-algebras have sifted colimits}, $S^*(\{X_i\};\ZZ/p^a)$ is an $\mathcal{E}$-algebra. After base change, $S^*(\{X_i\};\ZZ/p^a)$ is an $E_{\infty}$-algebra over $\ZZ/p^a$. Moreover, the forgetful functor $F:\mathbf{Alg}(\mathcal{E})\rightarrow \mathbf{Ch}(\widehat{\ZZ}_p)$ preserves all filtered colimits and is right adjoint to the free algebra functor. Then $F$ preserves all small limits. Therefore, $S^*(\{X_i\};\widehat{\ZZ}_p)$ is also an $\mathcal{E}$-algebra.
\end{proof}

\begin{definition}
Let $\{X_i\}$ be a pro system of simplicial sets. Define $C^*(\{X_i\};\widehat{\ZZ}_p)\otimes_{\widehat{\ZZ}_p} k$ to be the CDGA rectification of the $E_{\infty}$-algebra $S^*(\{X_i\};\widehat{\ZZ}_p)\otimes_{\widehat{\ZZ}_p} k$.
\end{definition}

A pointed connected CW complex, or a pointed connected simplicial set, $X$, is \textbf{$p$-finite} if each $\pi_i(X)$ is a finite $p$-group. Let $\{X_i\}_{i\in I}$ be a pro system of pointed connected $p$-finite simplicial sets. The homotopy limit $X$ of $\{X_i\}$ is constructed as follows (\cite{Riehl-homotopy-categories}*{Theorem 5.2.6}). Equip the diagram category $(\mathbf{sSet}_*)^{I}$ with the injective model structure, and choose an injectively fibrant replacement $\{X_i\}_{i\in I}\rightarrow \{(RX)_i\}_{i\in I}$. Define the \textbf{homotopy limit} of $\{X_i\}_{i\in I}$ by $X:=\varprojlim_i (RX)_i$. 
The maps $\{X_i\}\rightarrow \{(RX)_i\}\leftarrow X$ induce $E_{\infty}$-algebra morphisms $S^*(\{X_i\};\widehat{\ZZ}_p)\leftarrow S^*(\{(RX)_i\};\widehat{\ZZ}_p)\rightarrow S^*(X;\widehat{\ZZ}_p)$.

\begin{lemma}\label{Lem: quasi-isomorphisms between homotopy limit and the pro space}
Let $\{X_i\}_{i\in I}$ be a pro system of pointed connected $p$-finite simplicial sets.
Retain the notation introduced above.
\begin{enumerate}[leftmargin=0.25in]
    \item The maps $\{X_i\}\leftarrow \{(RX)_i\}\rightarrow X$ induce $\varprojlim_i \pi_*(X_i)\cong  \varprojlim_i \pi_*((RX)_i)\cong \pi_*(X)$.
    \item The induced $E_{\infty}$-algebra morphisms  $S^*(\{X_i\};\ZZ/p^a)\leftarrow S^*(\{(RX)_i\};\ZZ/p^a)$ and  $S^*(\{X_i\};\widehat{\ZZ}_p)\leftarrow S^*(\{(RX)_i\};\widehat{\ZZ}_p)$ are quasi-isomorphisms.
    \item If Sullivan's $p$-completion $X\rightarrow X^{\wedge,S}_p$ of $X$ is a homotopy equivalence, then the induced $E_{\infty}$-algebra morphisms $S^*(\{(RX)_i\};\ZZ/p^a)\rightarrow S^*(X;\ZZ/p^a)$ and $S^*(\{(RX)_i\};\widehat{\ZZ}_p)\rightarrow S^*(X;\widehat{\ZZ}_p)$ are quasi-isomorphisms.
    \item If $\varinjlim_i H^q(X_i;\ZZ/p)$ ($H^q(X;\ZZ/p)$, resp.) is finite for every $q$, then $H^*(S^*(\{X_i\};\widehat{\ZZ}_p))\cong \varprojlim_{a}\varinjlim_i H^*(X_i;\ZZ/p^a)$ ($H^*(X;\widehat{\ZZ}_p)\cong\varprojlim_{a} H^*(X;\ZZ/p^a)$, resp.).
\end{enumerate}
\end{lemma}

\begin{proof}
\textbf{Item (1). } By the definition of the injective model category structure, each map $X_i\rightarrow (RX)_i$ is a weak equivalence. The asserted isomorphisms of homotopy groups now follow from \cite{Friedlander-etale-homotopy}*{Theorem 6.8} or \cite{Bousfield-Kan-homotopy-limits-completions-localizations}*{XI.7.1}.

\textbf{Item (2). } Since each map $X_i\rightarrow (RX)_i$ is a weak equivalence, the first arrow is a quasi-isomorphism. Since homology commutes with filtered colimits, applying the Milnor exact sequence (\cite{Weibel-homological-algebra}*{Theorem 3.5.8}) to the inverse limit over $a$ gives the following exact sequence.
\[
0\rightarrow \varprojlim_{a}{}^1 \varinjlim_i H^{q-1}(X_i;\ZZ/p^a) \rightarrow H^q(S^*(\{X_i\};\widehat{\ZZ}_p)) \rightarrow \varprojlim_{a} \varinjlim_i H^{q}(X_i;\ZZ/p^a) \rightarrow 0
\]
The same construction gives an analogous short exact sequence for $\{(RX)_i\}$. Then the second quasi-isomorphism follows from the first quasi-isomorphism and the five lemma.

\textbf{Item (3). } Let $X\rightarrow \{Y_j\}$ be Artin-Mazur's pro-$p$ completion (\cite{Artin-Mazur-etale-homotopy}*{Theorem 3.4}). Since each $RX_i$ is $p$-finite, the canonical map $X\rightarrow \{(RX)_i\}$ uniquely factors through $\{Y_j\}$. $X^{\wedge,S}_p$ is the homotopy limit of $\{Y_j\}$ (\cite{Sullivan-MIT-notes}*{p.~55, Corollary}). Then $\pi_*(X^{\wedge,S}_p)\cong \varprojlim_j \pi_*(Y_j)$. Since $X\rightarrow X^{\wedge,S}_p$ is a homotopy equivalence, $\pi_*(X)\cong\pi_*(X^{\wedge,S}_p)$. Since $\pi_*(X)\cong\varprojlim_i \pi_*((RX)_i)$, as shown in item (1), the map $\{\pi_*(Y_j)\}\rightarrow \{\pi_*((RX)_i)\}$ induced by $\{Y_j\}\rightarrow \{(RX)_i\}$ is an isomorphism. Artin-Mazur's comparison theorem (\cite{Artin-Mazur-etale-homotopy}*{Theorem 4.3}) therefore shows the first quasi-isomorphism. Applying the Milnor exact sequence argument in item (2) then shows the second quasi-isomorphism.

\textbf{Item (4). } 
The finiteness assumption for coefficients in $\ZZ/p$, together with the Bockstein exact sequence, implies that $\varinjlim_i H^q(X_i;\ZZ/p^a)$ is finite for every $q$ and $a$. Then the $\varprojlim^1\varinjlim_i H^q(X_i;\ZZ/p^a)=0$ for every $q$. Using the Milnor exact sequence in item (2) proves the isomorphism for $\{X_i\}$. The same argument proves the corresponding isomorphism for $X$.
\end{proof}

\begin{remark}
In summary, if Sullivan's $p$-completion $X\rightarrow X^{\wedge,S}_p$ of $X$ is a homotopy equivalence, then the $E_{\infty}$-algebra morphisms $S^*(\{X_i\};\widehat{\ZZ}_p)\leftarrow S^*(\{(RX)_i\};\widehat{\ZZ}_p)\rightarrow S^*(X;\widehat{\ZZ}_p)$ are quasi-isomorphisms. If, in addition, $\varinjlim_i H^q(X_i;\ZZ/p)$ is finite for every $q$, then the maps in the zig-zag $\{X_i\}\rightarrow \{(RX)_i\}\leftarrow X$ are $p$-adic weak equivalences by \cite{Morel-p-adic-spaces}*{Theorem 2.4.1}.
\end{remark}

Rectifying the preceding $E_{\infty}$-algebra morphisms $S^*(\{X_i\};\widehat{\ZZ}_p)\leftarrow S^*(\{(RX)_i\};\widehat{\ZZ}_p)\rightarrow S^*(X;\widehat{\ZZ}_p)$ gives CDGA morphisms $C^*(\{X_i\};\widehat{\ZZ}_p)\otimes_{\widehat{\ZZ}_p} k\leftarrow C^*(\{(RX)_i\};\widehat{\ZZ}_p)\otimes_{\widehat{\ZZ}_p} k\rightarrow C^*(X;\widehat{\ZZ}_p)\otimes_{\widehat{\ZZ}_p} k$. By Lemma \ref{Lem: quasi-isomorphisms between homotopy limit and the pro space}, these CDGA morphisms are quasi-isomorphisms if Sullivan's $p$-completion $X\rightarrow X^{\wedge,S}_p$ is a homotopy equivalence. Let $M_{\{X_i\}}$, $M_{\{(RX)_i\}}$, and $M_X$ denote minimal models of the CDGA's $C^*(\{X_i\};\widehat{\ZZ}_p)\otimes_{\widehat{\ZZ}_p} k$, $ C^*(\{(RX)_i\};\widehat{\ZZ}_p)\otimes_{\widehat{\ZZ}_p} k$, and $ C^*(X;\widehat{\ZZ}_p)\otimes_{\widehat{\ZZ}_p} k$, respectively. The previous CDGA quasi-isomorphisms induce isomorphisms $M_{\{X_i\}}\leftarrow M_{\{(RX)_i\}}\rightarrow M_X$.

\begin{definition}
A pro system $\{X_i\}$ of pointed connected $p$-finite simplicial sets is \textbf{nilpotent $p$-complete finite type} if its homotopy limit $X$ is a nilpotent $p$-complete finite type space.
\end{definition}

The following result is a direct corollary of \cite{Sullivan-1970-unpublished-notes}*{Proposition 3.8}.

\begin{lemma}\label{Lem: nilpotent $p$-complete finite types are good}
If $\{X_i\}$ is nilpotent $p$-complete finite type, then Sullivan's $p$-completion $X\rightarrow X^{\wedge,S}_p$ is a homotopy equivalence.
\end{lemma}

Theorem \ref{Thm: minimal model computes rational homotopy groups}, Lemma \ref{Lem: quasi-isomorphisms between homotopy limit and the pro space} and Lemma \ref{Lem: nilpotent $p$-complete finite types are good} immediately imply the following result.

\begin{corollary}\label{Cor: duality between minimal model and homotopy groups for pro spaces}
Let $\{X_i\}$ be a pro system of pointed connected $p$-finite simplicial sets, and assume that $\{X_i\}$ is nilpotent $p$-complete finite type. Let $\phi:M\rightarrow C^*(\{X_i\};\widehat{\ZZ}_p)\otimes_{\widehat{\ZZ}_p} k$ be a minimal model. For every $n$, let $L(n)$ denote the dual of the degree-$n$ indecomposable space $I^n(M)$ of $M$. Then
\begin{enumerate}[leftmargin=0.25in]
    \item there are natural nondegenerate $k$-bilinear pairings as follows;
    \[
\langle -,- \rangle: I^i(M)\otimes_{k} ((\varprojlim_j \pi_i(X_j))\otimes_{\widehat{\ZZ}_p} k) \rightarrow k \text{\,\,\, (for $i\geq 2$)}
\]
\[
\langle -,- \rangle: I^1(M)\otimes_{k} ((\bigoplus_s \Gamma^s(\varprojlim_j \pi_1(X_j))/\Gamma^{s+1}(\varprojlim_j \pi_1(X_j)))\otimes_{\widehat{\ZZ}_p} k) \rightarrow k
\]
    \item the vector-space isomorphism $\bigoplus_{*\geq 2} L(*)\rightarrow \bigoplus_{*\geq 2}(\varprojlim_j \pi_*(X_j))\otimes_{\widehat{\ZZ}_p} k$ induced by item (1) is an isomorphism of graded Lie algebras after shifting degrees by $1$;
    \item if, in addition, $\QQ_p\subset k$ is a closed weakly cartesian pair of Hausdorff topological fields, then the nilpotent Lie algebra $L(1)$ associated with the $1$-minimal model is naturally isomorphic to the Lie algebra $\Lie((\varprojlim_j \pi_1(X_j))\widehat{\otimes} k)$ of the continuous Mal'cev $k$-completion of $\varprojlim_j \pi_1(X_j)$.
\end{enumerate}    
\end{corollary}

\begin{theorem}\label{Thm: 1-minimal model is Malcev completion}
Let $\QQ_p\subset k$ be a weakly cartesian pair of Hausdorff topological fields.
\begin{enumerate}[leftmargin=0.25in]
    \item Let $X$ be a connected CW complex whose fundamental group $\pi_1(X)$ is a topologically finitely generated pro-$p$ group. Let $\phi_X:M_X(1)\rightarrow C^*(X;\widehat{\ZZ}_p)\otimes_{\widehat{\ZZ}_p}k$ be a $1$-minimal model, and let $L_X(1)$ denote the pro-finite-dimensional pro-nilpotent Lie algebra dual to $M_X(1)$ as in Theorem \ref{Fact: nilpotent Lie algebra induced from minimal model}. Then $L_X(1)$ is naturally isomorphic to the Lie algebra $\Lie(\pi_1(X)\widehat{\otimes} k)$ of the continuous Mal'cev $k$-completion of $\pi_1(X)$.
    \item Let $\{X_i\}$ be a pro system of pointed connected $p$-finite simplicial sets, and assume that the pro-$p$ group $\varprojlim_i \pi_1(X_i)$ is topologically finitely generated. Let $\phi_{\{X_i\}}:M_{\{X_i\}}(1)\rightarrow C^*(\{X_i\};\widehat{\ZZ}_p)\otimes_{\widehat{\ZZ}_p}k$ be a $1$-minimal model, and let $L_{\{X_i\}}(1)$ denote the pro-finite-dimensional pro-nilpotent Lie algebra dual to $M_{\{X_i\}}(1)$. Then $L_{\{X_i\}}(1)$ is naturally isomorphic to the Lie algebra $\Lie((\varprojlim_i \pi_1(X_i))\widehat{\otimes} k)$ of the continuous Mal'cev $k$-completion of $\varprojlim_i \pi_1(X_i)$.
\end{enumerate}
\end{theorem}

\begin{proof}

\textbf{Item (1). } By Proposition \ref{Prop: base change of Malcev completion}, it suffices to prove the theorem when $k=\QQ_p$. Set $\pi=\pi_1(X)$, and let $\Gamma^*\pi$ denote the lower central series. For every $s$, set $\pi_s=\pi/\Gamma^{s+1}\pi$. Consider the following diagram.
\[
\begin{tikzcd}
    & & \vdots \arrow[d] \\
    & & K(\pi_2,1) \arrow[d] \\
    X \arrow[r] & K(\pi,1) \arrow[ruu] \arrow[ru] \arrow[r] & K(\pi_1,1)
\end{tikzcd}
\]
The proof of Theorem \ref{Thm: minimal model computes rational homotopy groups} constructs a tower of minimal models $M(1,1)\subset M(1,2)\subset M(1,3)\subset ...$ associated to $C^*(K(\pi_s,1);\widehat{\ZZ}_p)\otimes_{\widehat{\ZZ}_p} \QQ_p$, where each $M(1,s-1)\subset M(1,s)$ is an elementary extension by $\homo_{\widehat{\ZZ}_p}(\Gamma^{s}\pi/\Gamma^{s+1}\pi,\QQ_p)$. Furthermore, the finite-dimensional nilpotent Lie algebra $L(1,s)$ dual to $M(1,s)$ is naturally isomorphic to $\Lie(\pi_s\widehat{\otimes}\QQ_p)$. By Proposition \ref{Prop: Malcev completion is the inverse limit of Malcev completion of lower central series}, $\varprojlim_sL(1,s)=\varprojlim_s\Lie(\pi_{s}\widehat{\otimes}\QQ_p)=\Lie(\pi\widehat{\otimes}\QQ_p)$.

Let $M(1)=\bigcup M(1,s)$. Considering the composition of natural maps $M(1)=\varinjlim_sM(1,s)\rightarrow \varinjlim_sC^*(K(\pi_s,1);\widehat{\ZZ}_p)\otimes_{\widehat{\ZZ}_p} \QQ_p\rightarrow C^*(K(\pi,1);\widehat{\ZZ}_p)\otimes_{\widehat{\ZZ}_p} \QQ_p\rightarrow C^*(X;\widehat{\ZZ}_p)\otimes_{\widehat{\ZZ}_p} \QQ_p$, it remains to prove that this composition of CDGA morphisms is a $1$-minimal model of $C^*(X;\widehat{\ZZ}_p)\otimes_{\widehat{\ZZ}_p} \QQ_p$. 

$\varinjlim_s H^1(K(\pi_s,1);\widehat{\ZZ}_p)\otimes_{\widehat{\ZZ}_p} \QQ_p\rightarrow H^1(K(\pi,1);\widehat{\ZZ}_p)\otimes_{\widehat{\ZZ}_p} \QQ_p\rightarrow H^1(X;\widehat{\ZZ}_p)\otimes_{\widehat{\ZZ}_p} \QQ_p$ is clearly an isomorphism. By Theorem \ref{Thm: minimal model computes rational homotopy groups}, for every $s$, $M(1,s)$ is indeed a minimal model of $C^*(K(\pi_s,1);\widehat{\ZZ}_p)\otimes_{\widehat{\ZZ}_p} \QQ_p$. Thus, $M(1)$ is a minimal model of $\varinjlim_sC^*(K(\pi_s,1);\widehat{\ZZ}_p)\otimes_{\widehat{\ZZ}_p} \QQ_p$.
Lemma \ref{Lem: injectivity of colimit of lower central series into group}(2) shows that $\varinjlim_s H^2(K(\pi_s,1);\widehat{\ZZ}_p)\otimes_{\widehat{\ZZ}_p} \QQ_p\rightarrow H^2(K(\pi,1);\widehat{\ZZ}_p)\otimes_{\widehat{\ZZ}_p} \QQ_p$ is injective. By the Leray-Serre spectral sequence for $\widetilde{X}\rightarrow X\rightarrow K(\pi,1)$, $H^2(K(\pi,1);\widehat{\ZZ}_p)\otimes_{\widehat{\ZZ}_p} \QQ_p\rightarrow H^2(X;\widehat{\ZZ}_p)\otimes_{\widehat{\ZZ}_p} \QQ_p$ is also injective. Together, these injectivity statements prove that $M(1)\rightarrow C^*(X;\widehat{\ZZ}_p)\otimes_{\widehat{\ZZ}_p}\QQ_p$ is a $1$-minimal model, completing the proof of item (1).

\textbf{Item (2). } We use the same strategy as in the proof of item (1). For every $i$ and $s$, set $\pi^i=\pi_1(X_i)$, $\pi^i_s=\pi^i/\Gamma^{s+1}\pi^i$, $\pi=\varprojlim_i \pi_1(X_i)$ and $\pi_s=\pi/\Gamma^{s+1}\pi$.

By Corollary \ref{Cor: duality between minimal model and homotopy groups for pro spaces} there is a tower of minimal models $M(1,1)\subset M(1,2)\subset M(1,3)\subset ...$, in which $M(1,s)$ is a minimal model of $C^*(\{K(\pi^i_s,1)\};\widehat{\ZZ}_p)\otimes_{\widehat{\ZZ}_p} \QQ_p$. Let $L(1,s)$ denote the finite-dimensional nilpotent Lie algebra dual to $M(1,s)$, which is naturally isomorphic to $\Lie((\varprojlim_i \pi^i_s)\widehat{\otimes}\QQ_p)$. Since $\pi=\varprojlim_i \pi^i$ is topologically finitely generated, we have $\Gamma^s\pi=\varprojlim_i \Gamma^s\pi^i$ for every $s$. Because finite colimits commute with cofiltered limits, $\pi_s=\varprojlim_i \pi^i_s$ for every $s$. Proposition \ref{Prop: Malcev completion is the inverse limit of Malcev completion of lower central series} therefore gives $\varprojlim_sL(1,s)=\varprojlim_s\Lie(\pi_{s}\widehat{\otimes}\QQ_p)=\Lie(\pi\widehat{\otimes}\QQ_p)$.

Set $M(1)=\bigcup M(1,s)$. It remains to show that the composition of natural maps $\phi:M(1)= \varinjlim_sM(1,s)\rightarrow \varinjlim_sC^*(\{K(\pi^i_s,1)\};\widehat{\ZZ}_p)\otimes_{\widehat{\ZZ}_p} \QQ_p\rightarrow C^*(\{K(\pi^i,1)\};\widehat{\ZZ}_p)\otimes_{\widehat{\ZZ}_p} \QQ_p\rightarrow C^*(\{X_i\};\widehat{\ZZ}_p)\otimes_{\widehat{\ZZ}_p} \QQ_p$ is a $1$-minimal model. 
By Corollary \ref{Cor: duality between minimal model and homotopy groups for pro spaces}, $M(1)$ is indeed a minimal model of $\varinjlim_sC^*(\{K(\pi^i_s,1)\};\widehat{\ZZ}_p)\otimes_{\widehat{\ZZ}_p} \QQ_p$.

Since $\QQ_p$ is flat over $\widehat{\ZZ}_p$, $H^q(C^*(\{K(\pi^i,1)\};\widehat{\ZZ}_p)\otimes_{\widehat{\ZZ}_p} \QQ_p)=H^q(\varprojlim_a \varinjlim_i S^*(K(\pi^i,1);\ZZ/p^a)\otimes_{\widehat{\ZZ}_p}\QQ_p)=H^q(\varprojlim_a \varinjlim_i S^*(K(\pi^i,1);\ZZ/p^a))\otimes_{\widehat{\ZZ}_p}\QQ_p$ for every $q$. Since $H^0(\varinjlim_i S^*(K(\pi^i,1);\ZZ/p^a)))=\varinjlim_i H^0(\pi^i;\ZZ/p^a)=\ZZ/p^a$, the Milnor exact sequence gives that $H^1(\varprojlim_a \varinjlim_i S^*(K(\pi^i,1);\ZZ/p^a))=\varprojlim_a \varinjlim_i H^1(\pi_i;\ZZ/p^a)$. Applying the same calculation to $\{K(\pi^i_s,1)\}$ and $X^i$ and then using Proposition \ref{Fact: relation between continuous cohomology and usual cohomology for the system}, the composition of maps $H^1(\varinjlim_sC^*(\{K(\pi^i_s,1)\};\widehat{\ZZ}_p)\otimes_{\widehat{\ZZ}_p} \QQ_p)\rightarrow H^1(C^*(\{K(\pi^i,1)\};\widehat{\ZZ}_p)\otimes_{\widehat{\ZZ}_p} \QQ_p)\rightarrow H^1(C^*(\{X_i\};\widehat{\ZZ}_p)\otimes_{\widehat{\ZZ}_p} \QQ_p)$ is identified with $\varinjlim_s (H^1_{\cont}(\pi_s;\widehat{\ZZ}_p)\otimes_{\widehat{\ZZ}_p}\QQ_p)\rightarrow H^1_{\cont}(\pi;\widehat{\ZZ}_p)\otimes_{\widehat{\ZZ}_p}\QQ_p\rightarrow (\varprojlim_a\varinjlim_i H^1(X_i;\ZZ/p^a))\otimes_{\widehat{\ZZ}_p}\QQ_p$. These maps are isomorphisms since each term is naturally isomorphic to $\homo_{\cont}(\pi,\widehat{\ZZ}_p)\otimes_{\widehat{\ZZ}_p}\QQ_p$. This proves that $\phi$ induces an isomorphism on $H^1$.

For each $a$, there is a natural identification $\varinjlim_i H^1(K(\pi^i,1);\ZZ/p^a)=\homo_{\cont}(\pi,\ZZ/p^a)$. This group is finite because $\pi$ is topologically finitely generated. Thus, $\varprojlim^1_a\varinjlim_i H^1(K(\pi^i,1);\ZZ/p^a)=0$. The Milnor exact sequence implies that $\varprojlim_a \varinjlim_i H^2(\pi_i;\ZZ/p^a)=H^2(\varprojlim_a \varinjlim_i S^*(K(\pi^i,1);\ZZ/p^a))$. Applying the same calculation to $\{K(\pi^i_s,1)\}$ and $X^i$, the composition of maps $H^2(\varinjlim_sC^*(\{K(\pi^i_s,1)\};\widehat{\ZZ}_p)\otimes_{\widehat{\ZZ}_p} \QQ_p)\rightarrow H^2(C^*(\{K(\pi^i,1)\};\widehat{\ZZ}_p)\otimes_{\widehat{\ZZ}_p} \QQ_p)\rightarrow H^2(C^*(\{X_i\};\widehat{\ZZ}_p)\otimes_{\widehat{\ZZ}_p} \QQ_p)$ is identified with $\varinjlim_s (H^2_{\cont}(\pi_s;\widehat{\ZZ}_p)\otimes_{\widehat{\ZZ}_p}\QQ_p)\rightarrow H^2_{\cont}(\pi;\widehat{\ZZ}_p)\otimes_{\widehat{\ZZ}_p}\QQ_p\rightarrow (\varprojlim_a\varinjlim_i H^2(X_i;\ZZ/p^a))\otimes_{\widehat{\ZZ}_p}\QQ_p$. Since tensoring with $\QQ_p$ preserves colimits, $\varinjlim_s (H^2_{\cont}(\pi_s;\widehat{\ZZ}_p)\otimes_{\widehat{\ZZ}_p}\QQ_p)=(\varinjlim_sH^2_{\cont}(\pi_s;\widehat{\ZZ}_p))\otimes_{\widehat{\ZZ}_p}\QQ_p$. Lemma \ref{Lem: injectivity of colimit of lower central series into group} shows that $\varinjlim_sH^2_{\cont}(\pi_s;\widehat{\ZZ}_p)\rightarrow H^2_{\cont}(\pi;\widehat{\ZZ}_p)$ is injective. By the Leray-Serre spectral sequence, $H^2(K(\pi^i,1);\ZZ/p^a)\rightarrow H^2(X_i;\ZZ/p^a)$ is injective for every $i$ and $a$. Since both the filtered colimit functor and the inverse limit functor are left exact,  $\varprojlim_a\varinjlim_i H^2(K(\pi^i,1);\ZZ/p^a)\rightarrow \varprojlim_a\varinjlim_i H^2(X_i;\ZZ/p^a)$ is also injective. Since $\QQ_p$ is flat over $\widehat{\ZZ}_p$, $\phi$ induces an injection on $H^2$, completing the proof of item (2).
\end{proof}

We briefly recall the inductive construction of a $1$-minimal model $M(1)$ for a CDGA $C^*$ with $H^1(C^*)$ finite-dimensional. Set $V_1=H^1(C^*)$, regarded as a vector space concentrated in degree one, and let $M(1,1)=\bigwedge^* V_1$, equipped with the zero differential. Define $V_s=\Ker(H^2(M(1,s-1))\rightarrow H^2(C))$. Then $M(1,s)$ is an elementary extension of $M(1,s-1)$ by $V_s$, placed in degree $1$. Choose a basis $\{x_1,...,x_t\}$ of $V_s$, and choose degree-$2$ cocycles $y_1,...,y_t\in M(1,s-1)^2$ representing the corresponding cohomology classes. The differential on $M(1,s)$ is given by $d(x_i)=y_i$ for every $i$. Finally, $M(1)=\bigcup_s M(1,s)$.

The following corollary is a more elaborate version of Theorem \ref{Thm: 1-minimal model is Malcev completion}.

\begin{corollary}
Retain the assumptions of Theorem \ref{Thm: 1-minimal model is Malcev completion}(1). Equip $M_X(1)$ with the tower $0\subset M_X(1,1)\subset M_X(1,2)\subset ...$ constructed above. Then the dual tower of nilpotent Lie algebras is naturally isomorphic to the tower of nilpotent Lie algebras $...\twoheadrightarrow\Lie(\pi_1(X)\widehat{\otimes} k)/\Gamma^3\twoheadrightarrow\Lie(\pi_1(X)\widehat{\otimes} k)/\Gamma^2\twoheadrightarrow 0$ of the Lie algebra of the continuous Mal'cev $k$-completion of $\pi_1(X)$. The analogous statement also holds for a pro system $\{X_i\}$ of pointed connected $p$-finite simplicial sets such that $\varprojlim_i \pi_1(X_i)$ is topologically finitely generated.
\end{corollary}

\begin{proof}
We prove the corollary only for $X$ since the proof for $\{X_i\}$ is identical. By Proposition \ref{Prop: base change of Malcev completion}, it is enough to treat the case $k=\QQ_p$. Let $\pi=\pi_1(X)$. We prove by induction that the vector space $M_X(1,i-1)^1$ is naturally dual to $\Lie(\pi\widehat{\otimes} \QQ_p)/\Gamma^i$. For $i=1$, $M_X(1,0)^1$ and $\Lie(\pi\widehat{\otimes}\QQ_p)$ are both zero, so the assertion is immediate. Assume inductively that the assertion has been proved for all integers less than $i$. Proposition \ref{Prop: lower central series of Lie algebra and malcev completion of lower central series} identifies $\Lie(\pi\widehat{\otimes} \QQ_p)/\Gamma^i\Lie(\pi\widehat{\otimes} \QQ_p)$ with the Lie algebra of $(\pi/\Gamma^i\pi)\widehat{\otimes} \QQ_p$. Consider the exact sequences $0\rightarrow (\Gamma^{i-1}\pi/\Gamma^{i}\pi)\otimes_{\widehat{\ZZ}_p}\QQ_p\rightarrow \Lie(\pi\widehat{\otimes} \QQ_p)/\Gamma^i\Lie(\pi\widehat{\otimes} \QQ_p)\rightarrow \Lie(\pi\widehat{\otimes} \QQ_p)/\Gamma^{i-1}\Lie(\pi\widehat{\otimes} \QQ_p)\rightarrow 0$ and $0\rightarrow V_{i-1}\rightarrow M(1,i-1)^1\rightarrow M(1,i-2)^1\rightarrow 0$. By the inductive hypothesis, $\Lie(\pi\widehat{\otimes} \QQ_p)/\Gamma^{i-1}$ is dual to $M(1,i-2)^1$. It therefore remains to prove that $ V_{i-1}$ is dual to $(\Gamma^{i-1}\pi/\Gamma^{i}\pi)\otimes_{\widehat{\ZZ}_p}\QQ_p$. By the inductive construction of the $1$-minimal model, $V_{i-1}\cong \Ker(H^2(K(\pi/\Gamma^{i-1}\pi,1);\widehat{\ZZ}_p)\otimes_{\widehat{\ZZ}_p}\QQ_p\rightarrow H^2(K(\pi,1);\widehat{\ZZ}_p)\otimes_{\widehat{\ZZ}_p}\QQ_p)$. Proposition \ref{Prop: comparison between continuous and discrete cohomology for Serre good groups}, Lemma \ref{Lem: relation between continuous and discrete cohomology for nilpotent pro-p groups} and Lemma \ref{Lem: injectivity of colimit of lower central series into group} identify $V_{i-1}\cong \Ker(H^2_{\cont}(\pi/\Gamma^{i-1}\pi;\widehat{\ZZ}_p)\otimes_{\widehat{\ZZ}_p}\QQ_p\rightarrow H^2_{\cont}(\pi;\widehat{\ZZ}_p)\otimes_{\widehat{\ZZ}_p}\QQ_p)$.  By Proposition \ref{Prop: injectivity of colimit of lower central series into group}, $\Ker(H^2_{\cont}(\pi/\Gamma^{i-1}\pi;\widehat{\ZZ}_p)\otimes_{\widehat{\ZZ}_p}\QQ_p\rightarrow H^2_{\cont}(\pi;\widehat{\ZZ}_p)\otimes_{\widehat{\ZZ}_p}\QQ_p)\cong\homo_{\widehat{\ZZ}_p}(\Gamma^{i-1}\pi/\Gamma^{i}\pi,\QQ_p)$. This completes the proof.
\end{proof}

\section{Applications}\label{Section: Application}
In this section, we apply the $\QQ_p$-homotopy theory developed in Section \ref{Section: Q_p homotopy theory} to several problems in topology and algebraic geometry. We first characterize simply-connected $p$-complete spaces that arise as $p$-completions of finite CW complexes (Corollary \ref{Cor: integral lifting of simply-connected p-complete spaces}). We then construct a simply-connected example showing that finite-dimensional cohomology alone does not guarantee finite realization (Example \ref{Example: p-complete space which is not a finite type space}). We establish a corresponding realization criterion for \'etale homotopy types (Corollary \ref{Cor: finite CW complex for positive haract3eristic variety}) and prove the $\QQ_{\ell}$-formality of smooth proper varieties in characteristic $p$ (Theorem \ref{Thm: formality of varieties}). We also give an affirmative answer (Proposition \ref{Prop: continuity of Galois action on automorphism group of minimal model}) to Deligne's question (Question \ref{Question: Deligne}) in \cite{Deligne-Weil-II}. Finally, we study Galois representations on \'etale homotopy groups of $p$-adic varieties (Theorem \ref{Thm: de Rham represenations for etale homotopy groups} and Theorem \ref{Thm: crystalline represenations for etale homotopy groups}), and derive bounded-presentation results for the continuous Mal’cev completions of \'etale fundamental groups (Theorem \ref{Thm: fundamental group of algebraic varieties}).

\subsection{Finite CW Complex Realization of Nilpotent $p$-Complete Spaces}\label{subsection: finite CW complex realization of nilpotent spaces}\,

This subsection addresses the finite realization problem posed in Question \ref{Question: finite CW complex realization for p-complete space}: whether a $p$-complete space with finitely generated $\widehat{\ZZ}_{p}$-cohomology, vanishing in sufficiently high degrees, arises as the $p$-completion of a finite CW complex. We show that finite realization is governed not only by cohomological finiteness, but also by a rational descent condition on the $\QQ_{p}$-homotopy type.

\begin{lemma}\label{Lem: local spaces have an integral lifting}
Let $T$ be a set of prime numbers. Every nilpotent $T$-local finite type CW complex is the $T$-localization of a nilpotent finite type CW complex.  
\end{lemma} 

\begin{proof}
We argue by induction on a refined principal Postnikov tower. It suffices to prove the following: for a nilpotent finite type CW complex $X$ and a principal fibration $K(G_{(T)},n)\xrightarrow{i} Y_{(T)}\xrightarrow{q} X_{(T)}$, where $G$ is a finitely generated abelian group, $G_{(T)}=G\otimes_{\ZZ}\ZZ_{(T)}$, and $X_{(T)}$  is a $T$-localization of $X$, there exists a nilpotent finite type CW complex $Y$, together with a map of principal fibrations as in the following diagram, such that the vertical arrows are $T$-localizations.
\[
\begin{tikzcd}
    K(G,n) \arrow[r] \arrow[d] & Y \arrow[r] \arrow[d] & X \arrow[d] \\
    K(G_{(T)},n) \arrow[r,"i"] & Y_{(T)} \arrow[r,"q"] & X_{(T)}
\end{tikzcd}
\]
Equivalently, it suffices to construct the following diagram, where $b$ is induced by the principal fibration $K(G_{(T)},n)\xrightarrow{i} Y_{(T)}\xrightarrow{q} X_{(T)}$, the upper vertical arrows are $T$-localizations, and the lower vertical arrows are homotopy equivalences.
\[
\begin{tikzcd}
    Y \arrow[r] \arrow[d] & X \arrow[r] \arrow[d] & K(G,n+1) \arrow[d] \\
     Y'_{(T)} \arrow[r,"q'"] \arrow[d,"\simeq"] & X'_{(T)} \arrow[r,"b'"] \arrow[d,"\simeq"] & K(G_{(T)},n+1) \arrow[d,"\simeq","\alpha"'] \\
     Y_{(T)} \arrow[r,"q"] & X_{(T)} \arrow[r,"b"]  & K(G_{(T)},n+1) 
\end{tikzcd}
\]
The natural isomorphism $H^{n+1}(X_{(T)};G_{(T)})\cong H^{n+1}(X;G)\otimes_{\ZZ} \ZZ_{(T)}$ implies that there exist a class $[b']\in H^{n+1}(X;G)$ and an integer $k$ not divisible by any prime in $T$ such that $k\cdot [b']=[b]\in H^{n+1}(X_{(T)};G_{(T)})$. Let $\alpha$ be the self-map of $K(G_{(T)},n+1)$ induced by multiplication by $\frac{1}{k}$ on $ G_{(T)}$. Define the middle row by pulling back the bottom fibration along $\alpha$. The localization map $X\rightarrow X_{(T)}$ induces a map $X\rightarrow X'_{(T)}$ which makes the upper-right square commute up to homotopy. Since $X'_{(T)}\rightarrow X_{(T)}$ is a homotopy equivalence, $X\rightarrow X'_{(T)}$ is also a $T$-localization. Define $Y$ to be the homotopy fiber of $b':X\rightarrow K(G,n+1)$. The resulting map $Y\rightarrow Y'_{(T)}$ has the required properties, completing the proof.
\end{proof}

\begin{theorem}\label{Cor: minimal model lifting is space lifting}
Let $X$ be a connected nilpotent $p$-complete finite type CW complex. Let $S^*(X;\widehat{\ZZ}_p)\otimes_{\widehat{\ZZ}_p} \QQ_p$ be the $E_{\infty}$-algebra of singular cochains, and let $C^*(X;\widehat{\ZZ}_p)\otimes_{\widehat{\ZZ}_p} \QQ_p$ denote its CDGA rectification. Let $\phi:M\rightarrow C^*(X;\widehat{\ZZ}_p)\otimes_{\widehat{\ZZ}_p} \QQ_p$ be a minimal model. Then the following statements are equivalent.
\begin{enumerate}[leftmargin=0.25in]
    \item $X$ is the $p$-completion of a nilpotent finite type CW complex.
    \item $M$ is isomorphic to the scalar extension of a minimal CDGA over $\QQ$.
    \item $S^*(X;\widehat{\ZZ}_p)\otimes_{\widehat{\ZZ}_p} \QQ_p$ is quasi-isomorphic to the scalar extension of an $E_{\infty}$-algebra over $\QQ$.
\end{enumerate}
\end{theorem}

\begin{proof}
The equivalence between items (2) and (3) is clear. 

\textbf{$(1)\Rightarrow (2)$.} Let $Y$ be a nilpotent finite type CW complex whose $p$-completion is $X$. Let $f:Y\rightarrow X$ be the $p$-completion map. Let $C^*(Y;\QQ)$ and $C^*(Y;\QQ_p)$ be the rectifications of the singular cochain complexes $S^*(Y;\QQ)$ and $S^*(Y;\widehat{\ZZ}_p)\otimes_{\widehat{\ZZ}_p}\QQ_p$, respectively. The natural $E_{\infty}$-algebra morphism $S^*(Y;\QQ)\rightarrow S^*(Y;\QQ_p)$ induces a CDGA morphism $i_1:C^*(Y;\QQ)\otimes_{\QQ} \QQ_p\rightarrow C^*(Y;\QQ_p)$. Analogously, there is a CDGA morphism $i_2:C^*(Y;\widehat{\ZZ}_p)\otimes_{\widehat{\ZZ}_p}\QQ_p\rightarrow C^*(Y;\QQ_p)$. Let $\phi_Y:M_Y\rightarrow C^*(Y;\QQ)$ be a minimal model over $\QQ$. These maps fit into the zig-zag of CDGA morphisms, in which both $i_1$ and $i_2$ are quasi-isomorphisms.
\[
\begin{tikzcd}
    M_Y\otimes_{\QQ} \QQ_p \arrow[r,"\phi_Y\otimes_{\QQ}\QQ_p"] & C^*(Y;\QQ)\otimes_{\QQ} \QQ_p \arrow[r,"i_1"] & C^*(Y;\QQ_p) \\
    M \arrow[r,"\phi"] & C^*(X;\widehat{\ZZ}_p)\otimes_{\widehat{\ZZ}_p} \QQ_p \arrow[r,"f^*\otimes \QQ"] &  C^*(Y;\widehat{\ZZ}_p)\otimes_{\widehat{\ZZ}_p} \QQ_p  \arrow[u,"i_2"]
\end{tikzcd}
\]
Theorem \ref{Fact: Lifting of minimal models} therefore gives an isomorphism of minimal CDGA's $M_Y\otimes_{\QQ}\QQ_p\rightarrow M$. 

\textbf{$(2)\Rightarrow (1)$.} 
Recall the following arithmetic square of a nilpotent $p$-local finite type space $Z$ in \cite{Sullivan-MIT-notes}*{p.~87}, where $a_p$, $b$, $c$ and $d_p$ are $p$-completion, rationalization, rationalization and formal $p$-completion, respectively.
\[
\begin{tikzcd}
    Z_{(p)} \arrow[r,"a_p"] \arrow[d,"b"] & Z_p \arrow[d,"c"] \\
    Z_{\QQ} \arrow[r,"d_{p}"] & Z_{\QQ_{p}}
\end{tikzcd}
\]
By Lemma \ref{Lem: local spaces have an integral lifting} and the above arithmetic square, it suffices to construct a nilpotent rational finite type CW complex $Y_{\QQ}$ together with a map $f:Y_{\QQ}\rightarrow X_{\QQ}$ such that $f_*:\pi_*(Y_{\QQ})\otimes_{\QQ} \QQ_p\rightarrow \pi_*(X_{\QQ})$ is an isomorphism.

Let $N$ be a minimal CDGA over $\QQ$, together with a CDGA isomorphism $\psi:M\rightarrow N\otimes_{\QQ} \QQ_p$. By the rational homotopy theory (\cite{Deligne-Griffiths-Morgan-Sullivan-formality}*{p.~35, (2)}), there exists a rational nilpotent finite type CW complex $Y_{\QQ}$ together with a quasi-isomorphism $\phi_Y:N\rightarrow C^*(Y_{\QQ};\QQ)$.

Choose an exhaustive tower $\QQ=N(1,0)\subset N(1,1)\subset ... \subset N(1,s_1)=N(2,0)\subset N(2,1)\subset ...$ of elementary extensions of minimal subalgebras of $N$. $\psi$ induces a corresponding exhaustive tower of elementary extensions $\QQ_p=M(1,0)\subset M(1,1)\subset ...$ of $M$. The proof of Theorem \ref{Thm: minimal model computes rational homotopy groups} and \cite{MGriffiths-Morgan-rational-homotopy}*{Corollary 12.3} identify these towers of elementary extensions with refined principal Postnikov towers $...\rightarrow Y^1_{\QQ,2}\rightarrow Y^1_{\QQ,1}=K(G_{1,1},1)\rightarrow Y^1_0=pt$ and $...\rightarrow X^1_{\QQ,2} \rightarrow X^1_{\QQ,1}=K(H_{1,1},1)\rightarrow X^1_0=pt$ of $Y_{\QQ}$ and $X_{\QQ}$ respectively. At each stage,  $X^n_{\QQ,s}\rightarrow X^n_{\QQ,s-1}$ (or $Y^n_{\QQ,s}\rightarrow Y^n_{\QQ,s-1}$, resp.) is a fibration with fiber $K(H_{n,s},n)$ (or $K(G_{n,s},n)$, resp.), and each $H_{n,s}$ (or $G_{n,s}$, resp.) is a finite-dimensional vector space over $\QQ_p$ (or $\QQ$, resp.) dual to $I^n(M(n,s))/I^n(M(n,s-1))$ ($I^n(N(n,s))/I^n(N(n,s-1))$, resp.).

The isomorphism $M(1,1)\rightarrow N(1,1)\otimes_{\QQ} \QQ_p$ induces $G_{1,1}\hookrightarrow G_{1,1}\otimes_{\QQ} \QQ_p\xrightarrow{\sim} H_{1,1}$ and hence a map $K(G_{1,1},1)\rightarrow K(H_{1,1},1)$. Assume inductively that we have constructed a map $f_{n,s}:Y^n_{\QQ,s}\rightarrow X^n_{\QQ,s}$ such that, for every $q$, the induced map $\pi_q(Y^n_{\QQ,s})\otimes_{\QQ}\QQ_p\rightarrow \pi_q(X^n_{\QQ,s})$ is an isomorphism. The isomorphism of elementary extensions $(M(n,s)\subset M(n,s+1))\cong (N(n,s)\subset N(n,s+1))\otimes_{\QQ}\QQ_p$ induces an isomorphism $\alpha_{n,s+1}: G_{n,s+1}\otimes_{\QQ} \QQ_p\rightarrow H_{n,s+1}$. To extend $f_{n,s}$ to the next stage, it suffices to prove that the following diagram commutes up to homotopy, where $k^{\QQ}_Y$ and $k^{\QQ}_X$ are $k$-invariants of $Y_{\QQ}$ and $X_{\QQ}$, respectively.
\[
\begin{tikzcd}
 Y^n_{\QQ,s} \arrow[r,"f_{n,s}"] \arrow[d,"k_Y^{\QQ}"] & X^n_{\QQ,s} \arrow[d,"k_X^{\QQ}"] \\
K(G_{n,s+1},n+1) \arrow[r,"\alpha_{n,s+1}"] & K(H_{n,s+1},n+1)   
\end{tikzcd}
\]
Equivalently, it is enough to show that $[k^{\QQ}_X\circ f_{n,s}]=[\alpha_{n,s+1}\circ k^{\QQ}_Y]\in H^{n+1}(Y^n_{\QQ,s};H_{n,s+1})$.   The fibration  $X^n_{\QQ,s+1}\rightarrow X^n_{\QQ,s}\xrightarrow{k^{\QQ}_X}K(H_{n,s+1},n+1)$ is obtained by rationalizing a stage $X^n_{s'+1}\rightarrow X^n_{s'}\xrightarrow{k_{X}} K(A_{s'},n+1)$ in a refined principal Postnikov tower of $X$, where $A_{n,s'+1}$ is a finitely generated $\widehat{\ZZ}_p$-module and $H_{n,s+1}=A_{n,s'+1}\otimes_{\widehat{\ZZ}_p}\QQ_p$. The indices $s$ and $s'$ need not agree, since torsion stages disappear after tensoring with $\QQ_p$. By Lemma \ref{Lem: local spaces have an integral lifting}, each $Y^n_{\QQ,s}$ is the rationalization of a nilpotent finite type CW complex $Y^n_s$. Under the preceding identifications, there is a natural isomorphism $H^{n+1}(Y^n_{\QQ,s};H_{n,s+1})\cong (H^{n+1}(Y^n_s;\widehat{\ZZ}_p)\otimes_{\widehat{\ZZ}_p} \QQ_p)\otimes_{\widehat{\QQ}_p} (A_{n,s'+1}\otimes_{\widehat{\ZZ}_p} \QQ_p)$.

Choose a basis $\{x_1,...,x_r\}$ of $H^{n}(K(A_{n,s'+1},n);\widehat{\ZZ}_p)\otimes_{\widehat{\ZZ}_p} \QQ_p\cong \homo_{\widehat{\ZZ}_p}(A_{n,s'+1},\QQ_p)=\homo_{\QQ_p}(H_{n,s+1},\QQ_p)$ (Lemma \ref{Lemma: Computation of Eilenberg-Maclane spaces}). Let $\{x^{\vee}_1,...,x^{\vee}_r\}$ denote the dual basis of $H_{n,s+1}=A_{n,s'+1}\otimes_{\widehat{\ZZ}}\QQ_p$. Let $y_1,...,y_r\in H^{n+1}(X^n_{s'};\widehat{\ZZ}_p)\otimes_{\widehat{\ZZ}_p} \QQ_p$ be the transgressions of $x_1,...,x_r$. By Corollary \ref{Cor: Extension of Minimal Model for Fibrations}, $\{y_1,...,y_r\}$ is a basis of $I^{n}(M(n,s+1))/I^{n}(M(n,s))$. Corollary \ref{Cor: k-invariant computation} therefore identifies $[k^{\QQ}_X\circ f_{n,s}]\in (H^{n+1}(Y^n_s;\widehat{\ZZ}_p)\otimes_{\widehat{\ZZ}_p} \QQ_p)\otimes_{\widehat{\QQ}_p} (A_{n,s'+1}\otimes_{\widehat{\ZZ}_p} \QQ_p)$ with $(f_{n,s})^*y_1\otimes x^{\vee}_1+... + (f_{n,s})^*y_r\otimes x^{\vee}_r$.

On the other hand, $\{\psi(x_1),...,\psi(x_r)\}$ is a basis of $(I^{n}N(n,s+1)/I^{n}N(n,s))\otimes_{\QQ} \QQ_p$. Considering the quasi-isomorphisms $N(n,s+1)\rightarrow  C^*(Y^n_{\QQ,s+1};\QQ)\rightarrow C^*(Y^n_{s+1};\QQ)$, $(f_{n,s})^*x_1,...,(f_{n,s})^*x_r$ are the transgressions of a basis of $\homo_{\QQ}(G_{n,s+1},\QQ_p)=H^{n+1}(K(G_{n,s+1},n+1);\QQ_p)$. Under the identification $\homo_{\QQ}(G_{n,s+1},\QQ_p)=H^{n+1}(K(G_{n,s+1},n+1);\QQ_p)$, this basis is dual to the basis $\{(\alpha_{n,s+1})^{-1}(x^{\vee}_1),...,(\alpha_{n,s+1})^{-1}(x^{\vee}_r)\}$ of $G_{n,s+1}\otimes_{\QQ} \QQ_p$. By \cite{MGriffiths-Morgan-rational-homotopy}*{Corollary 12.3}, $k^{\QQ}_Y\otimes_{\QQ} \QQ_p\in H^{n+1}(Y^n_{\QQ,s};G_{n,s+1}\otimes_{\QQ} \QQ_p)\cong (H^{n+1}(Y^n_{s};\widehat{\ZZ}_p)\otimes_{\widehat{\ZZ}_p} \QQ_p)\otimes_{\QQ_p} (G_{n,s+1}\otimes_{\QQ} \QQ_p)$ is $(f_{n,s})^*x_1\otimes (\alpha_{n,s+1})^{-1}(x^{\vee}_1)+...+ (f_{n,s})^*x_r\otimes (\alpha_{n,s+1})^{-1}(x^{\vee}_r)$. Then $[\alpha_{n,s+1}\circ k^{\QQ}_y]\in (H^{n+1}(Y^n_s;\widehat{\ZZ}_p)\otimes_{\widehat{\ZZ}_p} \QQ_p)\otimes_{\widehat{\QQ}_p} (A_{n,s'+1}\otimes_{\widehat{\ZZ}_p} \QQ_p)$ is $(f_{n,s})^*y_1\otimes x^{\vee}_1+... + (f_{n,s})^*y_r\otimes x^{\vee}_r$. This completes the proof.
\end{proof}

\begin{lemma}\label{Lem: sullivan's p-adic completion of finite CW complex is good}
Let $Z$ be a CW complex with finitely many cells in each dimension, and let $\alpha_p:Z\rightarrow Z^{\wedge,S}_{p}$ be Sullivan's $p$-completion. Then $\alpha_p$ induces isomorphisms $H^*(Z^{\wedge,S}_{p};\ZZ/p^a)\rightarrow H^*(Z;\ZZ/p^a)$ and $H^*(Z^{\wedge,S}_{p};\widehat{\ZZ}_p)\rightarrow H^*(Z;\widehat{\ZZ}_p)$.
\end{lemma}

\begin{proof}
The isomorphism $H^*(Z^{\wedge,S}_{p};\ZZ/p^a)\cong H^*(Z;\ZZ/p^a)$ follows from \cite{Sullivan-1970-unpublished-notes}*{Proposition 3.8}. Passing to the inverse limit over $a$ and applying the Milnor exact sequence gives the isomorphism $H^*(Z^{\wedge,S}_{p};\widehat{\ZZ}_p)\cong H^*(Z;\widehat{\ZZ}_p)$.
\end{proof}

\begin{corollary}\label{Cor: integral lifting of simply-connected p-complete spaces}
Let $X$ be a simply-connected, $p$-complete finite type CW complex. Let $S^*(X;\widehat{\ZZ}_p)\otimes_{\widehat{\ZZ}_p} \QQ_p$ be the $E_{\infty}$-algebra of singular cochains, and let $C^*(X;\widehat{\ZZ}_p)\otimes_{\widehat{\ZZ}_p} \QQ_p$ be its CDGA rectification. Let $M\rightarrow C^*(X;\widehat{\ZZ}_p)\otimes_{\widehat{\ZZ}_p}\QQ_p$ be a minimal model. 
\begin{enumerate}[leftmargin=0.25in]
    \item  The following statements are equivalent.
    \begin{enumerate}[leftmargin=0.27in]
        \item $X$ is Sullivan's $p$-completion of a CW complex with finitely many cells in each dimension.
        \item $M$ is isomorphic to the scalar extension of a minimal CDGA over $\QQ$.
        \item $S^*(X;\widehat{\ZZ}_p)\otimes_{\widehat{\ZZ}_p} \QQ_p$ is quasi-isomorphic to the scalar extension of an $E_{\infty}$-algebra over $\QQ$. 
    \end{enumerate}
    \item The following statements are equivalent.
    \begin{enumerate}[leftmargin=0.27in]
        \item $X$ is Sullivan's $p$-completion of a finite CW complex.
        \item $M$ is isomorphic to the scalar extension of a minimal CDGA over $\QQ$, and $H^{q}(X;\ZZ/p)=0$ for sufficiently large $q$'s.
        \item $S^*(X;\widehat{\ZZ}_p)\otimes_{\widehat{\ZZ}_p} \QQ_p$ is quasi-isomorphic to the scalar extension of an $E_{\infty}$-algebra over $\QQ$, and $H^{q}(X;\ZZ/p)=0$ for sufficiently large $q$'s.
    \end{enumerate}
\end{enumerate}
\end{corollary}

\begin{proof}
The equivalences $(1)(b)\Leftrightarrow (1)(c)$ and $(2)(b)\Leftrightarrow (2)(c)$ are clear.

For ``$(1)(b)\Rightarrow (1)(a)$'', the proof of Theorem \ref{Cor: minimal model lifting is space lifting} constructs a simply-connected finite-type CW complex $Z'$ whose Sullivan's $p$-completion is $X$. By \cite{Weinberger-stratified-spaces}*{p.~19, Proposition}, $Z'$ is homotopy equivalent to a CW complex with finitely many cells in each dimension. For ``$(1)(a)\Rightarrow (1)(c)$'', let $Z$ be a CW complex with finitely many cells in each dimension together with a $p$-completion map $Z\rightarrow X$. There are natural quasi-isomorphisms $S^*(Z;\widehat{\ZZ}_p)\otimes_{\widehat{\ZZ}_p}\QQ_p\leftarrow S^*(Z;\ZZ)\otimes_{\ZZ} \QQ_p \rightarrow S^*(Z;\QQ)\otimes_{\QQ}\QQ_p$. By Lemma \ref{Lem: sullivan's p-adic completion of finite CW complex is good}, $S^*(X;\widehat{\ZZ}_p)\rightarrow S^*(Z;\widehat{\ZZ}_p)$ is a quasi-isomorphism. These quasi-isomorphisms show that $S^*(X;\widehat{\ZZ}_p)\otimes_{\widehat{\ZZ}_p}\QQ_p$ and $ S^*(Z;\QQ)\otimes_{\QQ}\QQ_p$ are quasi-isomorphic.

The same argument for ``$(1)(a)\Rightarrow (1)(c)$'' also shows ``$(2)(a)\Rightarrow (2)(c)$''. For ``$(2)(b)\Rightarrow (2)(a)$'', the proof of Theorem \ref{Cor: minimal model lifting is space lifting} constructs a simply-connected finite type CW complex $Z'$ whose Sullivan's $p$-completion is $X$. By the condition of item (2)(b), we may assume that there exists an integer $N$ such that $H_i(X;\ZZ/p)=0$ for $i>N$. \cite{Hatcher-algebraic-topology}*{Theorem 4H.3} gives us a homology decomposition $f:Y\rightarrow Z'$, where $f$ is a homotopy equivalence and $Y$ has an exhaustive tower of simply-connected subcomplexes $Y_2\subset Y_3\subset ...$ such that, for every $n$, $H_i(Y_n;\ZZ)\cong H_i(Z';\ZZ)$ for $i\leq n$ and $H_i(Y_n;\ZZ)=0$ for $i>n$. Then the restriction $f_N:Y_N\rightarrow Z'$ of $f$ induces an isomorphism $H_i(Y_N;\ZZ/p)\rightarrow H_i(Z';\ZZ/p)$ for all $i$. Equivalently, $f_N$ is a $p$-adic weak equivalence. Thus, $X$ is Sullivan's $p$-completion of $Y_N$. \cite{Weinberger-stratified-spaces}*{p.~19, Proposition} proves that $Y_N$ is homotopy equivalent to a finite CW complex.
\end{proof}

\begin{theorem}\label{Thm: realization of a minimal model}
Let $M$ be a minimal CDGA over $\QQ_p$ such that every $M^n$ is finite-dimensional. Then there exists a nilpotent $p$-complete finite type CW complex $X$ and a CDGA quasi-isomorphism $M\rightarrow C^*(X;\widehat{\ZZ}_p)\otimes_{\widehat{\ZZ}_p} \QQ_p$. Moreover, if $M$ is simply-connected, then $X$ may also be chosen to be simply-connected.
\end{theorem}

\begin{proof}
By Theorem \ref{Fact: Lifting of minimal models}, $M$ admits an exhaustive tower $\QQ_p=M(1,0)\subset M(1,1)\subset ...\subset M(1,s_1)=M(2,0)\subset M(2,1) \subset ...$ of elementary extensions of subalgebras by finite-dimensional vector spaces.
Corresponding to this tower, we inductively construct a Postnikov tower of $...\rightarrow X^1_2\rightarrow X^1_1=K(G_{1,1},1)=X^1_0=pt$ together with compatible quasi-isomorphisms $M(n,s)\rightarrow C^*(X^n_s;\widehat{\ZZ}_p)\otimes_{\widehat{\ZZ}_p} \QQ_p$. At each stage, the relevant homotopy group $G_{n,s}$ is a finitely generated $\widehat{\ZZ}_p$-module. Once the construction is complete, let $X$ be the homotopy limit of the resulting Postnikov tower, and the compatible quasi-isomorphisms then induce the required quasi-isomorphism $M\rightarrow C^*(X;\widehat{\ZZ}_p)\otimes_{\widehat{\ZZ}_p} \QQ_p$.

For each stage $(n,s)$, choose a finitely generated free $\widehat{\ZZ}_p$-module $G_{n,s}$ and a vector-space isomorphism $G_{n,s}\otimes_{\widehat{\ZZ}_p}\QQ_p\rightarrow \homo_{\QQ_p}(I^n(M(n,s))/I^n(M(n,s-1)),\QQ_p)$. At the first stage, the chosen isomorphism induces $I^1(M(1,1))\cong \homo_{\widehat{\ZZ}_p}(G_{1,1},\QQ_p)$ and Lemma \ref{Lemma: Computation of Eilenberg-Maclane spaces} gives a quasi-isomorphism $M(1,1)=\bigwedge^* I^1(M(1,1))\rightarrow C^*(X^1_1;\widehat{\ZZ}_p)\otimes_{\widehat{\ZZ}_p} \QQ_p$. Assume inductively that we have constructed $X^n_s$ together with a quasi-isomorphism $M(n,s)\rightarrow C^*(X^n_s;\widehat{\ZZ}_p)\otimes_{\widehat{\ZZ}_p} \QQ_p$. By Corollary \ref{Cor: k-invariant computation} and Theorem \ref{Thm: minimal model computes rational homotopy groups}, the differential of the new degree-$n$ generators of $M(n,s+1)$ determines a class $k\in H^{n+1}(X^n_s;G_{n,s+1})\otimes_{\widehat{\ZZ}_p} \QQ_p$. After scaling the class $k$ by some element of $\widehat{\ZZ}_p$, we may assume that $k\in H^{n+1}(X^n_s;G_{n,s+1})$. Let $X^{n}_{s+1}$ be the homotopy fiber of $k:^n_sX\rightarrow K(G_{n,s+1},n+1)$. Lemma \ref{Lemma: fibration and elementary extension} gives a quasi-isomorphism $M(n,s+1)\rightarrow C^*(X^{n}_{s+1};\widehat{\ZZ}_p)\otimes_{\widehat{\ZZ}_p} \QQ_p$ extending the quasi-isomorphism $M(n,s)\rightarrow C^*(X^{n}_s;\widehat{\ZZ}_p)\otimes_{\widehat{\ZZ}_p} \QQ_p$.
\end{proof}

\begin{example}\label{Example: p-complete space which is not a finite type space}
Let $f_1,...,f_s\in \QQ_p[x_1,...,x_n]$ be homogeneous polynomials of degree $n$. Assume that there are no matrices $A\in GL(n,\QQ_p)$ and $B\in GL(s,\QQ_p)$  such that the transformed sequence $B\cdot (f_1(A\cdot (x_1,...,x_n)),...,f_s(A\cdot (x_1,...,x_n)))$ consists entirely of polynomials with coefficients in $\QQ$. For example, one may take a homogeneous cubic polynomial $f(x_1,x_2,x_3)\in \QQ_p[x_1,x_2,x_3]$ defining an elliptic curve whose $j$-invariant lies in $\QQ_p-\QQ$.

Let $M$ be a minimal model over $\QQ_p$ with generators $x_1,...,x_n$ in degree $2r$ and generators $y_1,...,y_s$ in degree $2rn-1$. Define the differential by $dy_j=f_j(x_1,...,x_n)$. 
There exists an inclusion of minimal CDGA's $\phi:M\hookrightarrow N$ such that $\phi^*:H^i(M)\rightarrow H^i(N)$ is an isomorphism for $i\leq 2rn$ and $H^i(N)=0$ for $i>2rn$. By Theorem \ref{Thm: realization of a minimal model}, there exists a simply-connected $p$-complete finite type space $X$ such that $N$ is a minimal model of $C^*(X;\widehat{\ZZ}_p)\otimes_{\widehat{\ZZ}_p} \QQ_p$. 

We inductively construct a simply-connected $p$-complete finite type space $X_{m}$ together with $f_m:X_m\rightarrow X$, for every $m> 2rn$, such that $H^i(X_m;\widehat{\ZZ}_p)=0$ for $2rn<i\leq m$, $(f_m)^*:H^i(X;\widehat{\ZZ}_p)\rightarrow H^i(X_m;\widehat{\ZZ}_p)$ is an isomorphism for $i\leq 2rn$, and $(f_m)^*:H^*(X;\widehat{\ZZ}_p)\otimes_{\widehat{\ZZ}_p}\QQ_p\rightarrow H^*(X_m;\widehat{\ZZ}_p)\otimes_{\widehat{\ZZ}_p}\QQ_p$ is an isomorphism. Begin the induction by setting $X_{2rn}=X$. Suppose inductively that we have constructed $X_{m-1}$ together with $f_{m-1}:X_{m-1}\rightarrow X$ satisfying the requested conditions. Since $H^m(X_{m-1};\widehat{\ZZ}_p)\otimes_{\widehat{\ZZ}_p}\QQ_p=0$, the finitely generated $\widehat{\ZZ}_p$-module $A_m:=H^m(X_{m-1};\widehat{\ZZ}_p)$ is a finite $p$-group. Choose a decomposition $A_m\cong \bigoplus_{j}\ZZ/p^{e_j}\cdot a_j$. Let $\beta_{e_j}:H^{m-1}(X_{m-1};\ZZ/p^{e_j})\rightarrow H^m(X_{m-1};\widehat{\ZZ}_p)$ be the Bockstein homomorphism associated to the homomorphism  $\widehat{\ZZ}_p\xrightarrow{p^{e_j}} \widehat{\ZZ}_p$. Then there exists a class $c_j\in H^{m-1}(X_{m-1};\ZZ/p^{e_j})$ such that $\beta_{e_j}(c_j)=a_j$. Then the collection of classes $\{c_j\}_j$ represents a map $k_m:X_{m-1}\rightarrow K(A_m,m-1)$. Let $X_m$ be the homotopy fiber of $k_m$, and let $p_m:X_m\rightarrow X_{m-1}$ be the canonical map. Define $f_m$ to be $f_{m-1}\circ p_m$.

Since $m-2\geq 2$, $H^{m-1}(K(A_m,m-2);\widehat{\ZZ}_p)\cong A_m$, and $H^q(K(A_m,m-2);\widehat{\ZZ}_p)=0$ for $0<q<m-1$ and for $q=m$. The Leray-Serre spectral sequence for the fibration $K(A_m,m-2)\rightarrow X_{m}\xrightarrow{p_m} X_{m-1}$ computes that $(p_m)^*:H^q(X_{m-1};\widehat{\ZZ}_p)\rightarrow H^q(X_m;\widehat{\ZZ}_p)$ is an isomorphism for $q<m$, that $H^m(X_m;\widehat{\ZZ}_p)=0$, and that $(p_m)^*:H^*(X_{m-1};\widehat{\ZZ}_p)\otimes_{\widehat{\ZZ}_p}\QQ_p\rightarrow H^*(X_m;\widehat{\ZZ}_p)\otimes_{\widehat{\ZZ}_p}\QQ_p$ is an isomorphism. Therefore, the space $X_m$ and the map $f_m$ satisfy all the inductive requirements.

Let $Y$ be the homotopy limit of $\{X_m\}$. Then $Y$ is a simply-connected, $p$-complete finite type space such that $H^i(Y;\widehat{\ZZ}_p)=0$ for $i>2rn$, and $N$ is the minimal model of $C^*(Y;\widehat{\ZZ}_p)\otimes_{\widehat{\ZZ}_p} \QQ_p$. 
By Corollary \ref{Cor: integral lifting of simply-connected p-complete spaces}, $Y$ is not Sullivan's $p$-completion of any finite CW complex.\qed
\end{example}

\subsection{Finite CW Realization of \'Etale Homotopy Types}\,

This subsection addresses Question \ref{Question: finite CW complex for varieties}: when the $\ell$-completion of the \'etale homotopy type of an algebraic variety is $\ell$-adic weak equivalent to the $\ell$-completion of a finite CW complex? \cite{Hu-Zhang-formal-manifold-structures} proves that the $\ell$-adic \'etale homotopy groups of an $\ell$-adically simply-connected variety in characteristic $p\geq 0$ are finitely generated $\widehat{\ZZ}_{\ell}$-modules. Nevertheless, this finite type condition does not guarantee finite realization (Example \ref{Example: p-complete space which is not a finite type space}). For varieties, the criterion is simpler than the rational descent condition on the cochain level in Subsection \ref{subsection: finite CW complex realization of nilpotent spaces}: it is enough to impose a rational descent condition on the cohomology algebra. 

Let $\{Y_i\}$ be a pro-system of pointed connected simplicial sets. Let $\widehat{(\cdot)}$ denote the functor which assigns to a set the pro-system of all its finite quotient sets. The functor $\widehat{(\cdot)}$ extends to simplicial sets and then to pro simplicial sets. The canonical map $\{Y_i\}\rightarrow \{Y_i\}^{\wedge}$ preserves the chosen base points. The target is therefore a pro system of pointed simplicial finite sets.

\begin{lemma}\label{Lem: naive profinite completion of simplicial sets}
The canonical map $\{Y_i\}\rightarrow \{Y_i\}^{\wedge}=\{Z_j\}$ is a $p$-adic weak equivalence.
\end{lemma}

\begin{proof}
By \cite{Morel-p-adic-spaces}*{Theorem 2.4.1}, it suffices to prove that $\varinjlim_j H^*(Z_j;\ZZ/p)\rightarrow \varinjlim_i H^*(Y_i;\ZZ/p)$ is an isomorphism. This is established in the proof of \cite{Isaksen-model-category-of-pro-spaces}*{Proposition 9.4}.
\end{proof}

Recall from \cite{Bousfield-Kan-homotopy-limits-completions-localizations}*{Chapter I} the Bousfield-Kan's $\FF_p$-completion $(\FF_p)_{\infty} X$ of any simplicial set $X$. It is the homotopy limit of the tower $\{(\FF_p)_{n} X\}_n$. Let $\{Y_i\}^{\wedge}_p$ be the pro-system of simplicial sets $\{(\FF_p)_{n} Z_j\}_{n,j}$.

\begin{lemma}\label{Lem: definition of pro-p completion of a pro simplicial set}
Each space $(\FF_p)_{n} Z_j$ in the pro system $\{Y_i\}^{\wedge}_p$ is a $p$-finite space and the canonical map $\{Y_i\}\rightarrow \{Y_i\}^{\wedge}_p$ is Artin-Mazur's pro-$p$ completion of pro-spaces. 
\end{lemma}

\begin{proof}
Since each $Z_j$ is a simplicial finite set, the first assertion follows from \cite{Bousfield-Kan-homotopy-limits-completions-localizations}*{Chapter I, 2.3} and the second assertion follows from Lemma \ref{Lem: naive profinite completion of simplicial sets} and \cite{Friedlander-etale-homotopy}*{Proposition 6.10}. 
\end{proof}

\begin{proposition}\label{Prop: liftings of simply connected p-complete finite type pro-spaces}
Let $\{Y_i\}$ be a pro-system of pointed connected simplicial sets, with $\pi_1(\{Y_i\})^{\wedge}_p=0$. Then 
\begin{enumerate}[leftmargin=0.25in]
    \item $\{Y_i\}$ is $p$-adic weak equivalent to a CW complex with finitely many cells in each dimension if and only if $\varinjlim_i H^q(Y_i;\ZZ/p)$ is finite for every $q$, and $S^*(\{Y_i\};\widehat{\ZZ}_p)\otimes_{\widehat{\ZZ}_p} \QQ_p$ is quasi-isomorphic to the scalar extension of an $E_{\infty}$-algebra over $\QQ$.
    \item $\{Y_i\}$ is $p$-adic weak equivalent to a finite CW complex if and only if $\varinjlim_i H^q(Y_i;\ZZ/p)$ is finite for every $q$ and vanishes for sufficiently large $q$'s, and $S^*(\{Y_i\};\widehat{\ZZ}_p)\otimes_{\widehat{\ZZ}_p} \QQ_p$  is quasi-isomorphic to the scalar extension of an $E_{\infty}$-algebra over $\QQ$.
\end{enumerate}
\end{proposition}

\begin{proof}
By \cite{Artin-Mazur-etale-homotopy}*{Theorem 4.3}, we may replace $\{Y_i\}$ by $\{Y_i\}^{\wedge}_p$ without loss of generality. Under this assumption, each $Y_i$ is $p$-finite. Let $Y_p$ be the homotopy limit of $\{Y_i\}$. We prove  only item (1) since item (2) follows from an analogous argument.

\textbf{``If'' part.}  
By \cite{Hu-Zhang-formal-manifold-structures}*{Theorem 1.5}, $Y_p$ is $p$-complete finite type. Using Lemma \ref{Lem: quasi-isomorphisms between homotopy limit and the pro space} and Lemma \ref{Lem: nilpotent $p$-complete finite types are good}, $S^*(\{Y_i\};\widehat{\ZZ}_p)\otimes_{\widehat{\ZZ}_p} \QQ_p$ is quasi-isomorphic to $S^*(Y_p;\widehat{\ZZ}_p)\otimes_{\widehat{\ZZ}_p} \QQ_p$, and $\varinjlim_i H^*(Y_i;\ZZ/p)\cong H^*(Y_p;\ZZ/p)$. By \cite{Morel-p-adic-spaces}*{Theorem 2.4.1}, the canonical map $Y_p\rightarrow \{Y_i\}$ is a $p$-adic weak equivalence. Corollary \ref{Cor: integral lifting of simply-connected p-complete spaces} shows that $Y_p$, and hence $\{Y_i\}$ is $p$-adic weak equivalent to a CW complex with finitely many cells in each dimension.

\textbf{``Only if'' part.}
By \cite{Artin-Mazur-etale-homotopy}*{Theorem 4.3}, $\varinjlim_i H^q(Y_i;\ZZ/p)$ is finite for every $q$. Then $Y_p$ is $p$-complete finite type (\cite{Hu-Zhang-formal-manifold-structures}*{Theorem 1.5}). Analogously to the ``if'' part, $S^*(\{Y_i\};\widehat{\ZZ}_p)\otimes_{\widehat{\ZZ}_p} \QQ_p$ is quasi-isomorphic to $S^*(Y_p;\widehat{\ZZ}_p)\otimes_{\widehat{\ZZ}_p} \QQ_p$ and the canonical map $Y_p\rightarrow \{Y_i\}$ is a $p$-adic weak equivalence. Corollary \ref{Cor: integral lifting of simply-connected p-complete spaces} applied to $Y_p$ now gives the required rational descent condition.
\end{proof}

The same argument proves the following nilpotent version of Proposition \ref{Prop: liftings of simply connected p-complete finite type pro-spaces}.

\begin{proposition}\label{Prop: liftings of nilpotent p-complete finite type pro-spaces}
Let $\{Y_i\}$ be a pro-system of pointed connected simplicial sets, such that $\{Y_i\}^{\wedge}_p$ is nilpotent $p$-complete finite type. Then
$\{Y_i\}$ is $p$-adic weak equivalent to a nilpotent finite type CW complex if and only if  $S^*(\{Y_i\};\widehat{\ZZ}_p)\otimes_{\widehat{\ZZ}_p} \QQ_p$ is quasi-isomorphic to the scalar extension of an $E_{\infty}$-algebra over $\QQ$.
\end{proposition}

Let $X$ be a pointed, connected, locally Noetherian scheme. Friedlander (\cite{Friedlander-etale-homotopy}*{Definition 4.4}) constructed a pro system $X_{\et}$ of pointed connected simplicial sets, called the \textbf{\'etale homotopy type} of $X$.

Proposition \ref{Prop: liftings of simply connected p-complete finite type pro-spaces} and Proposition \ref{Prop: liftings of nilpotent p-complete finite type pro-spaces} immediately give the following realization theorem.

\begin{theorem}\label{Thm: finite CW complex realizations of a locally noetherian scheme}
Let $X$ be a pointed, connected, locally Noetherian scheme, and let $p$ be a prime number. 
\begin{enumerate}[leftmargin=0.25in]
    \item If $(X_{\et})^{\wedge}_p$ is nilpotent $p$-complete finite type, then $X_{\et}$ is $p$-adic weak equivalent to a nilpotent finite type CW complex if and only if $S^*(X_{\et};\widehat{\ZZ}_p)\otimes_{\widehat{\ZZ}_p} \QQ_p$ is quasi-isomorphic to the scalar extension of an $E_{\infty}$-algebra over $\QQ$.
    \item If $\pi_1^{\et}(X)^{\wedge}_p=0$, then 
    \begin{enumerate}[leftmargin=0.27in]
        \item $X_{\et}$ is $p$-adic weak equivalent to a CW complex with finitely many cells in each dimension if and only if $H^q_{\et}(X;\ZZ/p)$ is finite for every $q$, and $S^*(X_{\et};\widehat{\ZZ}_p)\otimes_{\widehat{\ZZ}_p} \QQ_p$ is quasi-isomorphic to the scalar extension of an $E_{\infty}$-algebra over $\QQ$;
        \item $X_{\et}$ is $p$-adic weak equivalent to a finite CW complex if and only if $H^q_{\et}(X;\ZZ/p)$ is finite for every $q$ and vanishes for all sufficiently large $q$, and $S^*(X_{\et};\widehat{\ZZ}_p)\otimes_{\widehat{\ZZ}_p} \QQ_p$ is quasi-isomorphic to the scalar extension of an $E_{\infty}$-algebra over $\QQ$.
    \end{enumerate}
\end{enumerate}
\end{theorem}

\begin{corollary}\label{Cor: finite CW complex for positive haract3eristic variety}
Let $X$ be a connected variety over a separably closed field $k$ of characteristic $p\geq 0$, and let $\ell$ be a prime number. Assume that $(\pi_1^{\et}X)^{\wedge}_{\ell}=0$.
\begin{enumerate}[leftmargin=0.25in]
    \item If either $p\neq \ell$ or $X$ is proper over $k$, then $X_{\et}$ is $\ell$-adic weak equivalent to a finite CW complex if and only if $S^*(X_{\et};\widehat{\ZZ}_{\ell})\otimes_{\widehat{\ZZ}_{\ell}} \QQ_{\ell}$ is quasi-isomorphic to the scalar extension of an $E_{\infty}$-algebra over $\QQ$.
    \item If $X$ is smooth proper over $k$, then $X_{\et}$ is $\ell$-adic weak equivalent to a finite CW complex if and only if $H^*_{\et}(X;\QQ_{\ell})$ is isomorphic to the scalar extension of a graded $\QQ$-algebra. 
\end{enumerate}
\end{corollary}

\begin{proof}
Item (2) follows from item (1) and Theorem \ref{Thm: formality of varieties}. For item (1), note that \cite[VI, Theorem 1.1]{milneLEC} shows that $H^q_{\et}(X;\mathbb{Z}/\ell)=0$ for $q> 2 \mathrm{dim}(X)$ and \cite{Hu-Zhang-formal-manifold-structures}*{Corollary 1.4} shows that $(X_{\et})^{\wedge}_{\ell}$ is $\ell$-complete finite type. Now item (1) follows from Theorem \ref{Thm: finite CW complex realizations of a locally noetherian scheme}.
\end{proof}

\subsection{Weights and Formality}\label{Subsec: Formality}\,

In this subsection, we prove the $\QQ_{p}$-formality for a smooth proper variety in characteristic $p$, using Deligne's idea in \cite{Deligne-Weil-II}*{Section 5}. Then we study the upper and lower bounds for the Frobenius weights on \'etale homotopy groups. 

Let $p$ be a prime number and let $q=p^{r}$ for some $r\geq 1$.  An algebraic number $\alpha$ is said to have \textbf{$q$-weight} $n$ if $|\alpha|=q^{\frac{n}{2}}$ for every embedding $\QQ(\alpha)\hookrightarrow \CC$. Lemma \ref{Lemma: rigidity of automorphisms of a minimal model} and the purity theorem for Frobenius acting on \'etale cohomology (\cite{Deligne-Weil-II}\cite{Chiarellotto-le-Stum-on-purity-of-crystalline-cohomology}\cite{Milne-values-of-zeta-functions-of-varieties-over-finite-fields}) give the following result.

\begin{lemma}\label{Lem: Frobenius weight on minimal model}
Let $\FF_q$ be the finite field of $q$ elements of characteristic $p>0$, and let $\ell$ be a prime number. Let $X$ be a connected smooth proper variety over $\FF_{q}$. Let $M$ be a minimal model of $C^*((X_{\overline{\FF}_q})^{\wedge}_{\et,\ell};\widehat{\ZZ}_{\ell})\otimes_{\widehat{\ZZ}_{\ell}} \QQ_{\ell}$, and let $\phi_M:M\rightarrow M$ be an isomorphism induced by the geometric Frobenius on $X$. Then
each eigenvalue $\alpha$ of $\phi_M$ on $M$ is an algebraic number of integral $q$-weight. In particular, the $q$-weights of eigenvalues of $\phi_M$ make $M$ a weighted CDGA.
\end{lemma}

\begin{corollary}\label{Cor: Formality for any varieties over finite fields}
Let $\FF_q$ be the finite field of characteristic $p>0$, and let $\ell$ be a prime number. Let $X$ be a connected smooth proper variety over $\FF_{q}$. Then the pro-$\ell$ completion of the \'etale homotopy type of $X_{\overline{\FF}_q}$ is $\QQ_{\ell}$-formal.    
\end{corollary}

\begin{proof}
By \cite{Deligne-Weil-II}\cite{Chiarellotto-le-Stum-on-purity-of-crystalline-cohomology}\cite{Milne-values-of-zeta-functions-of-varieties-over-finite-fields}\cite{Jannsen-weights}*{Section 2-3}, the geometric Frobenius action on $H^n_{\et}(X_{\overline{\FF}_q};\QQ_{\ell})\cong H^n((X_{\overline{\FF}_q})^{\wedge}_{\et,\ell};\widehat{\ZZ}_{\ell})\otimes_{\widehat{\ZZ}_{\ell}}\QQ_{\ell}$ is of pure weight $n$. Lemma \ref{Lem: Frobenius weight on minimal model} equips the minimal model $M$ of $C^*((X_{\overline{\FF}_q})^{\wedge}_{\et,\ell};\widehat{\ZZ}_{\ell})\otimes_{\widehat{\ZZ}_{\ell}} \QQ_{\ell}$ with a compatible weight grading. Proposition \ref{Prop: purity implies formality} shows that $M$ is $\QQ_{\ell}$-formal.
\end{proof}

The case $\ell\neq p$ in the following theorem is essentially due to Deligne-Griffiths-Morgan-Sullivan (\cite{Deligne-Griffiths-Morgan-Sullivan-formality}) and Deligne (\cite{Deligne-Weil-II}*{Corollary 5.3.7}).

\begin{theorem}\label{Thm: formality of varieties}
Let $k$ be a separably closed field of characteristic $p\geq 0$, and let $\ell$ be a prime. Let $X$ be a connected smooth proper variety over $k$. Then the pro-$\ell$ completion of the \'etale homotopy type of $X$ is $\QQ_{\ell}$-formal.
\end{theorem}

\begin{proof}
We only prove the case $\ell=p$ here. There exists a commutative ring $S$ which is finite type over $\ZZ$ and a smooth proper scheme $\chi$ over $S$ such that the fraction field $F$ of $S$ is contained in $k$ and $\chi_{k}=X$. For a closed point $s$ of $\Spec(S)$, let $k_s$ be its residue field. By Zariski's lemma, $k_s$ is a finite field of characteristic $p$. Choose a closed point $s\in \Spec(S)$. Let $G_s\subset \Gal(\overline{F}/F)$ be the decomposition group associated with $s$. Choose a lifting $\phi\in G_s$ of the Frobenius automorphism of $\overline{k_s}$. 

\cite{Gros-Suwa}*{Chapter II, Theorem 2.1}\cite{Jannsen-weights}*{(3.2)} shows that there exists a closed point $s\in \Spec(S)$ such that, for every $r$, there is a Galois-equivariant isomorphism between $H^r_{\et}(\chi_{\overline{F}};\QQ_{p})$ and  $H^r_{\et}(\chi_{\overline{k_s}};\QQ_{p})$. Thus, each eigenvalue of $(\phi^{-1})^*$ on $H^r_{\et}(\chi_{\overline{F}};\QQ_{p})$ is an algebraic number of pure weight $r$. The action of $\phi^{-1}$ on $\chi_{\overline{F}}$ induces an automorphism $\sigma$ of a minimal model $M_{\chi,\overline{F}}$ of $C^*((\chi_{\overline{F}})^{\wedge}_{\et,p};\widehat{\ZZ}_p)\otimes_{\widehat{\ZZ}_p}\QQ_p$. By Lemma \ref{Lemma: rigidity of automorphisms of a minimal model}, the automorphism $\sigma$ makes $M_{\chi,\overline{F}}$ a weighted minimal CDGA compatible with the Frobenius weights. Proposition \ref{Prop: purity implies formality}, together with the weight purity of $H^r_{\et}(\chi_{\overline{F}};\QQ_{p})$ for all $r$, shows that $(\chi_{\overline{F}})^{\wedge}_{\et,p}$ is $\QQ_p$-formal. \cite{Artin-Mazur-etale-homotopy}*{Corollary 12.12} gives a canonical homotopy equivalence between  $(\chi_{\overline{F}})^{\wedge}_{\et,p}$ and $(\chi_{\overline{k}})^{\wedge}_{\et,p}=(X_{\overline{k}})^{\wedge}_{\et,p}$. This completes the proof.
\end{proof}

We now extend the discussion of Frobenius weights on \'etale cohomology to \'etale homotopy groups using $\QQ_{\ell}$-minimal models.

\begin{definition}[\cite{Jannsen-weights}*{Section 2}]\label{Def: weights of representations of Galois groups}
Let $k$ be a field finitely generated over its prime field, and let $\ell$ be a prime. Let $V$ be a finite-dimensional continuous $\QQ_{\ell}$-representation of $\Gal(\overline{k}/k)$.
\begin{enumerate}[leftmargin=0.25in]
    \item First suppose that $k$ is a finite field with $q$ elements. The representation $V$ is called \textbf{pure of weight $n$} if each eigenvalue $\alpha$ of the geometric Frobenius action $\phi$, which is the inverse of the arithmetic Frobenius action, is an algebraic number of $q$-weight $n$. $V$ is called \textbf{mixed of weights $n_1,...,n_r$} if its eigenvalues of $\phi$ are algebraic numbers of $q$-weights among $n_1,...,n_r$.
    \item For a general finitely generated field $k$, $V$ is called \textbf{pure of weight $n$} if there exists an integral scheme $S$ of finite type over $\ZZ$ with function field $k$ such that, for any closed point $s\in \Spec(S)$, the inertia group $I_s\subset \Gal(\overline{k}/k)$ acts trivially on $V$ and the restricted representation $V$ over $\Gal(\overline{k_s}/k_s)=G_s/I_s$ is pure of weight $n$, where $G_s$ is the decomposition group of $s$ and $k_s$ is the residue field of $s$. The notion of being \textbf{mixed of weights $n_1,...,n_r$} is defined analogously.
\end{enumerate}
\end{definition}

\begin{theorem}\label{Thm: weights on etale homotopy groups}
Let $k$ be a field of characteristic $p\geq 0$ that is finitely generated over its prime field, and let $\ell$ be a prime number. Let $X$ be a connected smooth proper variety over $k$, equipped with a $k$-rational point. Assume that $\pi^{\et}_1(X_{\overline{k}})^{\wedge}_{\ell}=0$. Then, for every $n\geq 2$, the continuous representation $\pi_n((X_{\overline{k}})^{\wedge}_{\et,\ell})\otimes_{\widehat{\ZZ}_{\ell}}\QQ_{\ell}$  of $\Gal(\overline{k}/k)$ is mixed of weights lying in the interval $[n,2^{n-1}]$. Moreover, the Whitehead product is $\Gal(\overline{k}/k)$-equivariant and additive with respect to the weight grading.
\end{theorem}

\begin{proof}
Let $M$ be a minimal model of $C^*((X_{\overline{k}})^{\wedge}_{\et,\ell};\widehat{\ZZ}_{\ell})\otimes_{\widehat{\ZZ}_{\ell}}\QQ_{\ell}$. Since $X$ is smooth and proper, there exists a commutative ring $S$ which is finite type over $\ZZ$ and a smooth proper scheme $\chi$ over $S$ such that the fraction field of $S$ is $k$ and $\chi_{k}=X$. Moreover, \cite{Gros-Suwa}*{Chapter II, Theorem 2.1}\cite{Jannsen-weights}*{(3.2)} shows that there exists a nonempty open subset $U$ of $\Spec(S)$ such that, for every closed point $s\in U$ and every $r$, there is a Galois-equivariant isomorphism between $H^r_{\et}(\chi_{\overline{k}};\QQ_{\ell})$ and  $H^r_{\et}(\chi_{\overline{k_s}};\QQ_{\ell})$, where $k_s$ is the residue field of $s$. As argued in the proof of Theorem \ref{Thm: formality of varieties}, by Definition \ref{Def: weights of representations of Galois groups} and \cite{Artin-Mazur-etale-homotopy}*{Corollary 12.12, Corollary 12.13}, it suffices to prove the assertion after specializing to a finite field. Now we assume that $k$ is finite.

By Corollary \ref{Cor: duality between minimal model and homotopy groups for pro spaces} and Proposition \ref{Prop: continuity of Galois action on automorphism group of minimal model}, it suffices to prove that the geometric Frobenius weights on $I^{n}(M)$ lie between $n$ and $2^{n-1}$ for any $n$. By Lemma \ref{Lem: Frobenius weight on minimal model} and Remark \ref{Rmk: formality uses a weighted CDGA morphism}, it suffices to construct a weighted minimal model $N$ for the cohomology algebra $H^*_{\et}(X;\QQ_{\ell})$, equipped with the zero differential, such that the Frobenius weights of $I^n(N)$ lie between $n$ and $2^{n-1}$. Write $H^* $ for the CDGA $(H^*_{\et}(X;\QQ_{\ell}),d=0)$ equipped with the zero differential.

Recall the inductive construction of a minimal model for $H^*$. Set $V_2=H^2$ and $N(2)=\bigwedge^* V_2$, and let $N(2)\rightarrow H^*$ be the CDGA morphism extending the identity on $H^2$. Since $H^2$ is pure of weight $2$, the same is true for $V_2$. Suppose inductively that, for some $n\geq 3$, we have constructed a $(n-1)$-minimal model $N(n-1)\rightarrow H^*$ and a graded vector subspace $\bigoplus_{i=2}^{n-1} V_{i}$ of $N(n-1)$ such that $N(n-1)=\bigwedge^* (\bigoplus_{i=2}^{n-1} V_{i})$ and the weights of each $V_i$ lie between $i-1$ and $2^{i-2}$. To construct an $n$-minimal model of $H^*$, we set $V_n=\Ker(H^{n+1}(N(n-1))\rightarrow H^{n+1})\oplus  \Coker(H^{n}(N(n-1))\rightarrow H^{n})$. Then we build up an elementary extension $N(n)$ of $N(n-1)$ by $V_n$, together with a CDGA morphism $N(n)\rightarrow H^*$. Every weight of $V_n$ is at least $n$ and the largest possible weight $w$ in $V_n$ is bounded above by the largest weight occurring in $H^{n+1}(N(n-1))$. 

We now show that $w\leq 2^{n-1}$. Consider a monomial of total cohomological degree $n+1$ of $N(n-1)$, whose factors have degrees $2\leq q_1\leq ...\leq q_s$. Since $2^{a-1}+2^{b-1}\leq 2^{a+b-2}$ for all $a,b\geq 2$, repeated application of this inequality gives $2^{q_1-1}+...+2^{q_s-1}\leq 2^{n-1}$. Therefore, the inductive hypothesis for the weights of $N(n-1)$ implies that $w\leq 2^{n-1}$. 

For the ``moreover'' part, the composition $\Gal(\overline{k}/k)\rightarrow \Aut_h(M)\rightarrow \bigoplus_n \End(I^n(M))$ gives a Galois action on $I^n(M)$. Furthermore, the map $I^n(M)\rightarrow \sum_{s+t=n,s<t}I^s(M)\otimes I^t(M)$ induced by the quadratic part of the differential of $M$ is Galois-equivariant. Then the ``moreover'' part
follows from Corollary \ref{Cor: duality between minimal model and homotopy groups for pro spaces}.
\end{proof}

\subsection{Continuity of Galois Representations in Homotopy Classes of Automorphisms of Minimal Models}\,

In this subsection, we provide an affirmative answer to Deligne's question in \cite{Deligne-Weil-II}*{(5.3.9.1)}.

\begin{question}\label{Question: Deligne 1}
Let $k$ be a field of characteristic $p\geq 0$, and let $\ell\neq p$ be a prime number. Let $X$ be a connected smooth proper variety over $k$. Let $M$ be a minimal model of $C^*((X_{\overline{k}})^{\wedge}_{\et,\ell};\widehat{\ZZ}_{\ell})\otimes_{\widehat{\ZZ}_{\ell}}\QQ_{\ell}$.
Is the natural homomorphism $\Gal(\overline{k}/k)\rightarrow \Aut_h(M)$ continuous?
\end{question}

We use the results in \cite{Pridham-Galois-actions-on-homotopy-groups-algebraic-varieties} to answer this question. We first recall some terminology from \cite{Pridham-Galois-actions-on-homotopy-groups-algebraic-varieties}.

Let $\Gamma$ be a profinite group. The \textbf{pro-algebraic completion} of $\Gamma$ consists of a pro-algebraic group $\Gamma^{\mathrm{alg}}$ over $\QQ_{\ell}$, together with a continuous homomorphism $\Gamma\rightarrow \Gamma^{\mathrm{alg}}(\QQ_{\ell})$, satisfying the following universal property: for every pro-algebraic group $H$ over $\QQ_{\ell}$, each continuous homomorphism $\Gamma\rightarrow H(\QQ_{\ell})$ uniquely factors through $\Gamma^{\mathrm{alg}}(\QQ_{\ell})$ (\cite{Pridham-Galois-actions-on-homotopy-groups-algebraic-varieties}*{p.~29, Definition 3.1}). This definition is analogous to that of continuous Mal'cev completion. The image of $\Gamma\rightarrow \Gamma^{\mathrm{alg}}(\QQ_{\ell})$ is Zariski dense. Pridham also defines the \textbf{relative Mal'cev completion} of a pro-groupoid $\Gamma'$ with respect to an algebraic representation of $\Gamma'$, which unifies both pro-algebraic completion and continuous Mal'cev completion (\cite{Pridham-Galois-actions-on-homotopy-groups-algebraic-varieties}*{p.~29, Definition 3.3}).

Let $(Z,z)=\{(Z_j,z_j)\}$ be a pro system of pointed connected simplicial sets, whose $0$-th level sets form a constant pro system. Define $H^n(Z;\QQ_{\ell})=(\varprojlim_a\varinjlim_j H^n(Z_j;\ZZ/\ell^{a}))\otimes_{\widehat{\ZZ}_{\ell}} \QQ_{\ell}$. \cite{Pridham-Galois-actions-on-homotopy-groups-algebraic-varieties}*{p.~34, Definition 3.23} defines the \textbf{unipotent $\QQ_{\ell}$-Mal'cev homotopy type} $Z^{\mathrm{Mal}}$ of $Z$. We briefly review its construction as follows. For simplicity, assume that each $Z_j$ has a single vertex. Let $(G,\overline{W})$ denote the pair of Dwyer-Kan loop and  delooping functors between simplicial sets and simplicial groupoids. $\pi_0(Z^{\mathrm{Mal}})$ is the continuous Mal'cev $\QQ_{\ell}$-completion of $\pi_0(G(Z))^{\wedge}_{\ell}$. In simplicial degree $n$, the pro-unipotent group $(Z^{\mathrm{Mal}})_n$ is the relative Mal'cev completion of $(G(Z)_{n})^{\wedge}_{\ell}$ with respect to the homomorphism $(G(Z)_{n})^{\wedge}_{\ell}\rightarrow \pi_0(G(Z))^{\wedge}_{\ell}$ induced by the face maps.

When $H^n(Z;\QQ_{\ell})$ is finite-dimensional for every $n$, \cite{Pridham-Galois-actions-on-homotopy-groups-algebraic-varieties}*{p.~42, Definition 4.3 and Lemma 4.4} defines the pro-algebraic group $\mathrm{RAut}(Z^{\mathrm{Mal}})$ over $\QQ_{\ell}$. Roughly speaking, when each $Z_j$ has a single vertex, $\mathrm{RAut}(Z^{\mathrm{Mal}})$ is the pro-algebraic group of automorphisms of the simplicial pro-nilpotent Lie algebra $\Lie(Z^{\mathrm{Mal}})$.
 
Say that the profinite group \textbf{$\Gamma$ acts on $(Z,z)$ continuously} if $(Z,z)$ admits a representative pro system $\{(Z_i,z_i)\}$,  together with a pro-system $\{U_i\}$ of open normal subgroups of $\Gamma$, such that $\Gamma/U_i$ acts on each $Z_i$ while fixing $z_i$ and all transition morphisms in $\{(Z_i,z_i)\}$ are equivariant. Given such a continuous action, the naturality of $(-)^{\mathrm{Mal}}$ yields a group homomorphism $\phi: \Gamma\rightarrow \Aut(Z^{\mathrm{Mal}})\rightarrow \Aut_h(Z^{\mathrm{Mal}})\cong \mathrm{RAut}(Z^{\mathrm{Mal}})(\QQ_{\ell})$, where $\Aut_h(Z^{\mathrm{Mal}})$ is the automorphism group of $Z^{\mathrm{Mal}}$ in the homotopy category of simplicial pro-unipotent groups over $\QQ_{\ell}$ and the isomorphism $\Aut_h(Z^{\mathrm{Mal}})\cong \mathrm{RAut}(Z^{\mathrm{Mal}})(\QQ_{\ell})$ is given by \cite{Pridham-Galois-actions-on-homotopy-groups-algebraic-varieties}*{p.~42, the paragraph above Lemma 4.4}. The following result shows that $\phi$ is continuous.

\begin{lemma}\label{Lem: continuity on relative algebraic homotopy type}
With the same notation as above, suppose that $\Gamma$ acts on $(Z,z)$ continuously. Further assume that $H^n(Z;\QQ_{\ell})$ is finite-dimensional and carries a continuous $\Gamma$-representation for every $n$. Then there exists a morphism of pro-algebraic groups $\Phi:\Gamma^{\mathrm{alg}}\rightarrow \mathrm{RAut}(Z^{\mathrm{Mal}})$ such that the homomorphism $\phi: \Gamma\rightarrow \mathrm{RAut}(Z^{\mathrm{Mal}})(\QQ_{\ell})$ factors as $\Gamma\rightarrow \Gamma^{\mathrm{alg}}(\QQ_{\ell})\xrightarrow{\Phi(\QQ_{\ell})} \mathrm{RAut}(Z^{\mathrm{Mal}})(\QQ_{\ell})$. In particular, $\phi$ is continuous.
\end{lemma}

\begin{proof}
For every $i$, the natural projection $q_i:G(Z_i)(z_i,z_i)\rtimes \Gamma/U_i\rightarrow \Gamma/U_i$ is a fibration of simplicial groups, together with a section $\gamma\rightarrow (1,\gamma)$, where $\Gamma/U_i$ is regarded as a constant simplicial group. Applying the functor $\overline{W}$ to $q_i$ gives a Kan fibration of simplicial sets, whose fiber over the unique vertex of $\overline{W}(\Gamma/U_i)$ is homotopy equivalent to $Z_i$ and the fibration $\overline{W}(q_i)$ admits a section. This is the simplicial model of the Borel fibration associated to the group action. We thereby obtain a pro system of Kan fibrations $Z\rightarrow E\rightarrow B\Gamma$, together with a section $s_0:B\Gamma\rightarrow E$ induced by the fixed point $z$.

Let $R:=(\Gamma^{\mathrm{alg}})^{\mathrm{red}}$ be the reductive quotient of $\Gamma^{\mathrm{alg}}$. Define $\rho:\pi_1(E)\rightarrow R(\QQ_{\ell})$ to be the composition $\pi_1(E)=\pi_1(Z)\rtimes \Gamma\rightarrow \Gamma\rightarrow \Gamma^{\mathrm{alg}}(\QQ_{\ell})\rightarrow R(\QQ_{\ell})$, where the first arrow is the projection. Note that the image $\rho(\pi_1(Z))=0$. Moreover, the hypotheses imply that, for every finite-dimensional $\QQ_{\ell}$-vector space $V$, the monodromy action of $\pi_1(B\Gamma)=\Gamma$ on $H^n(Z;V)$ factors through $\Gamma^{\mathrm{alg}}=\Gamma^{R,\mathrm{Mal}}=\pi_1(G(B\Gamma)^{R,\mathrm{Mal}})$. \cite{Pridham-Galois-actions-on-homotopy-groups-algebraic-varieties}*{Theorem 3.32} implies that $Z^{\mathrm{Mal}}$ is the homotopy fiber of $G(E)^{(R,\rho),\mathrm{Mal}}\rightarrow G(B\Gamma)^{R,\mathrm{Mal}}$. Equivalently, there is a pro system of levelwise exact sequences of simplicial pro-algebraic groups $1\rightarrow U\rightarrow G(E)^{(R,\rho),\mathrm{Mal}}\rightarrow G(B\Gamma)^{R,\mathrm{Mal}}\rightarrow 1$, where $U$ is homotopic to $Z^{\mathrm{Mal}}$. 

The section $s_0:B\Gamma\rightarrow E$ induces a section $s:G(B\Gamma)^{R,\mathrm{Mal}}\rightarrow G(E)^{(R,\rho),\mathrm{Mal}}$. Consequently, there is an isomorphism $G(E)^{(R,\rho),\mathrm{Mal}}\cong U\rtimes G(B\Gamma)^{R,\mathrm{Mal}}$. Then this defines a morphism of pro-algebraic groups $G(B\Gamma)^{R,\mathrm{Mal}}\times U\rightarrow U$ given by $(q,u)\rightarrow s(q)us(q)^{-1}$. This conjugation action induces a natural homomorphism $(G(B\Gamma)^{R,\mathrm{Mal}})_0\rightarrow \mathrm{RAut}(U)$. This homomorphism descends to $\Phi:\pi_0(G(B\Gamma)^{R,\mathrm{Mal}})=\Gamma^{R,\mathrm{Mal}}\cong \Gamma^{\mathrm{alg}}\rightarrow \mathrm{RAut}(U)\cong \mathrm{RAut}(Z^{\mathrm{Mal}})$. 

The relative Mal'cev-completion functor $(-)^{\mathrm{Mal}}$ preserves the multiplication, projection, section, and conjugation maps of the original split extension $G(Z)\rightarrow G(Z)\rtimes \Gamma\rightarrow \Gamma$. Therefore, for every $q\in \Gamma$, conjugation by the image of $q$ in $\mathrm{RAut}(Z^{\mathrm{Mal}})(\QQ_{\ell})$ is exactly the Mal'cev completion of the original automorphism of $Z$ induced by $q$. This proves the asserted factorization of $\phi$ and completes the proof.
\end{proof}

The following proposition answers Question \ref{Question: Deligne 1}.

\begin{proposition}\label{Prop: continuity of Galois action on automorphism group of minimal model}
Let $k$ be a field of characteristic $p\geq 0$, and let $\ell$ be a prime number. Let $X$ be a connected smooth proper variety over $k$, equipped with a $k$-rational point. Let $M$ be a minimal model of $C^*((X_{\overline{k}})^{\wedge}_{\et,\ell};\widehat{\ZZ}_{\ell})\otimes_{\widehat{\ZZ}_{\ell}}\QQ_{\ell}$. Then the natural homomorphism $\Gal(\overline{k}/k)\rightarrow \Aut_h(M)$ induced by the continuous $\Gal(\overline{k}/k)$-action on $(X_{\overline{k}})^{\wedge}_{\et,\ell}$ is continuous, where $\Aut_h(M)$ is equipped with the $\ell$-adic topology arising from its pro-algebraic group structure over $\QQ_{\ell}$.  
\end{proposition}

\begin{proof}  
As explained in \cite{Pridham-Galois-actions-on-homotopy-groups-algebraic-varieties}*{p.~32, Example 3.18}, the sets of vertices of the simplicial sets representing $Y:=(X_{\overline{k}})_{\et}$ form the constant pro set of geometric points of $X_{\overline{k}}$. Lemma \ref{Lem: continuity on relative algebraic homotopy type} therefore shows that the canonical homomorphism $\Gal(\overline{k}/k)\rightarrow \mathrm{RAut}(Y^{\mathrm{Mal}})(\QQ_{\ell})$ is continuous. Let $\{Y_j\}$ be a pro simplicial set representing $Y$, and let $C^{\bullet}(Y;\QQ_{\ell}):=(\varprojlim_{a}\varinjlim_{j}C^{\bullet}(Y_j;\ZZ/\ell^{a}))\otimes_{\widehat{\ZZ}_{\ell}}\QQ_{\ell}$, which is a cosimplicial commutative $\QQ_{\ell}$-algebra. By \cite{Pridham-etale-fundamental-group-of-smooth-variety}*{Definition 2.61 and Theorem 2.74}, $\mathrm{RAut}(Y^{\mathrm{Mal}})(\QQ_{\ell})\cong \Aut_h(Y^{\mathrm{Mal}})$ is canonically isomorphic to $\Aut_h(C^{\bullet}(Y;\QQ_{\ell}))$, where $\Aut_h(C^{\bullet}(Y;\QQ_{\ell}))$ denotes the automorphism group of $C^{\bullet}(Y;\QQ_{\ell})$ in the homotopy category of cosimplicial commutative $\QQ_{\ell}$-algebras. Lemma \ref{Lem: Quillen equivalence of normalized cochain complex functor into Eilenberg-Zilber algebras} and Proposition \ref{Prop: embedding of non-negatively graded E_infinity algebras to unbounded E_infinity algebras} give a canonical isomorphism $\Aut_h(C^{\bullet}(Y;\QQ_{\ell}))\cong \Aut_h(S^{*}(Y;\widehat{\ZZ}_{\ell})\otimes_{\widehat{\ZZ}_{\ell}}\QQ_{\ell})$ induced by the normalized simplicial cochain complex functor, where $\Aut_h(S^{*}(Y;\widehat{\ZZ}_{\ell})\otimes_{\widehat{\ZZ}_{\ell}}\QQ_{\ell})$ denotes the automorphism group of $S^{*}(Y;\widehat{\ZZ}_{\ell})\otimes_{\widehat{\ZZ}_{\ell}}\QQ_{\ell}$ in the homotopy category of $E_{\infty}$-algebras. By Theorem \ref{Fact: rectification}, the rectification functor induces a canonical isomorphism $\Aut_h(S^{*}(Y;\widehat{\ZZ}_{\ell})\otimes_{\widehat{\ZZ}_{\ell}}\QQ_{\ell})\cong \Aut_h(M)$. Combining these identifications gives a canonical isomorphism $\mathrm{RAut}(Y^{\mathrm{Mal}})\cong \Aut_h(M)$ of pro-algebraic groups over $\QQ_{\ell}$. It follows that the homomorphism $\Gal(\overline{k}/k)\rightarrow \Aut_h(M)$ is continuous.
\end{proof}

\subsection{$p$-adic Galois Representations on \'Etale Homotopy Groups}\label{SubSec: Galois representation}\,

In this subsection, we study the Galois representation on the \'etale homotopy groups of $p$-adic varieties, with the known results from $p$-adic Hodge theory.

We now apply Proposition \ref{Prop: continuity of Galois action on automorphism group of minimal model} to study Galois representations on the \'etale homotopy groups of $p$-adic varieties. The notions concerning Galois representations in $p$-adic Hodge theory are recalled in Appendix \ref{Appendix: p-adic Hodge theory}.

\begin{definition}\label{Def: successively de Rham}
Let $K/\QQ_p$ be a finite field extension.
A continuous finite-dimensional $\QQ_p$-representation $V$ of $\Gal(\overline{K}/K)$ is called \textbf{successively de Rham} if it admits a finite $\Gal(\overline{K}/K)$-stable filtration $0=F^0V\subset F^1V\subset ...\subset F^mV=V$ whose graded pieces $\Gr^i_FV=F^iV/F^{i-1}V$ are de Rham. A \textbf{successively crystalline representation} is defined in the same way, with ``de Rham'' replaced by ``crystalline''.
\end{definition}

\begin{lemma}\label{Lem: Closure of algebraic operations of de Rham representations}
Successively de Rham representations are closed under taking subquotients, tensor products, duals and extensions. The same closure properties hold for successively crystalline representations.
\end{lemma}

Since the proof of Lemma \ref{Lem: Closure of algebraic operations of de Rham representations} is routine and tedious, we leave it at the end of Appendix \ref{Appendix: p-adic Hodge theory}.

\begin{lemma}\label{Lem: successive de Rham on indecomposables implies the property for cohomology}
Let $K/\QQ_p$ be a finite field extension.
Let $M$ be a finitely generated minimal CDGA over $\QQ_p$. Suppose that there is  a continuous homomorphism $\Gal(\overline{K}/K)\rightarrow \Aut_h(M)$ such that each indecomposable space $I^qM$ is a successively de Rham representation of $\Gal(\overline{K}/K)$. Then, for every $q$, $H^q(M)$ is a successively de Rham representation of $\Gal(\overline{K}/K)$. The analogous statement also holds with ``successively de Rham'' replaced by ``successively crystalline''.
\end{lemma}

\begin{remark}
By Example \ref{Example: rigid invariants}, the action of automorphisms of $M$ on $I^qM$ only depends on homotopy classes. Then it is legitimate to discuss the Galois representation on  $I^qM$.
\end{remark}

\begin{proof}
We prove only the case ``successively de Rham'' since the same proof works for ``successively crystalline''. 
Let $J=M^{>0}$ be the augmentation ideal of $M$, and let $J^r$ denote its $r$-th power.  Since $M$ is minimal, $d(J^r)\subset J^{r+1}$ for every $r$. Moreover, the natural map $\mu_r:\Sym^r(J/J^2)\rightarrow J^r/J^{r+1}$ induced by multiplication of $M$ is an isomorphism. Since $I^qM=(J/J^2)^q$ is successively de Rham for every $q$, Lemma \ref{Lem: Closure of algebraic operations of de Rham representations} and the isomorphism $\mu_r$ show that each term $E^{r,q}_1=(J^r/J^{r+1})^q$ on the $E_1$-page of the spectral sequence associated to the filtration $M=J^0\supset J^1\supset J^2\supset ...$  is successively de Rham. Each term on the $E_{\infty}$-page
is a subquotient of a term on the $E_1$-page and is therefore successively de Rham by Lemma \ref{Lem: Closure of algebraic operations of de Rham representations}.  The filtration induced on each $H^q(M)$ has the corresponding $E_{\infty}$-page terms as its graded pieces. By Lemma \ref{Lem: Closure of algebraic operations of de Rham representations}, $H^q(M)$ is successively de Rham.
\end{proof}

\begin{theorem}\label{Thm: de Rham represenations for etale homotopy groups}
Let $K/\QQ_p$ be a finite field extension, and let $\mathcal{O}_{K}$ denote its ring of integers. Let $X$ be a connected smooth proper adic space over $\mathrm{Spa}(K,\OO_K)$, equipped with a $\mathrm{Spa}(K,\OO_K)$-point. 
\begin{enumerate}[leftmargin=0.25in]
    \item The Lie algebra of the continuous Mal'cev $\QQ_p$-completion of $\pi^{\et}_1(X_{\overline{K}})^{\wedge}_{p}/\Gamma^n$ is a successively de Rham representation of $\Gal(\overline{K}/K)$ for every $n$, and the Lie bracket respects the $\Gal(\overline{K}/K)$-action, where $\Gamma^*$ denotes the lower central series.
    \item If $\pi^{\et}_1(X_{\overline{K}})^{\wedge}_{p}=0$, then $\pi_n((X_{\overline{K}})^{\wedge}_{\et,p})\otimes_{\widehat{\ZZ}_{p}}\QQ_{p}$ is a successively de Rham representation of $\Gal(\overline{K}/K)$ for every $n\geq 2$.
\end{enumerate}
\end{theorem}

\begin{proof}
\textbf{Item (1).} Let $W$ denote $\Lie((\pi^{\et}_1(X_{\overline{K}})^{\wedge}_{p}/\Gamma^2)\widehat{\otimes}\QQ_p)$. The natural $\Gal(\overline{K}/K)$-equivariant isomorphisms $W^{\vee}\cong H^1_{\cont}(\pi^{\et}_1(X_{\overline{K}})^{\wedge}_{p};\QQ_{p})\cong H^1_{\et}(X_{\overline{K};\QQ_p})$ and Proposition \ref{Fact: properties of de Rahm crystalline representations} shows that $W$ is de Rham. Iterated Lie brackets give a natural $\Gal(\overline{K}/K)$-equivariant surjection $W^{\otimes r}\twoheadrightarrow 
\frac{\Lie((\pi^{\et}_1(X_{\overline{K}})^{\wedge}_{p}/\Gamma^r)\widehat{\otimes}\QQ_p)}{\Lie((\pi^{\et}_1(X_{\overline{K}})^{\wedge}_{p}/\Gamma^{r+1})\widehat{\otimes}\QQ_p)}$. Then item (1) follows from an inductive argument together with Proposition\ref{Fact: properties of de Rahm crystalline representations}.

\textbf{Item (2).} Let $M$ be a minimal model of $C^*((X_{\overline{K}})^{\wedge}_{\et,p};\widehat{\ZZ}_{p})\otimes_{\widehat{\ZZ}_{p}}\QQ_{p}$, and let $M(n)$ be the subalgebra generated by elements of degree at most $n$. Applying the argument of Proposition \ref{Prop: continuity of Galois action on automorphism group of minimal model} to adic spaces, $\Gal(\overline{K}/K)\rightarrow \Aut_h(M)$ is continuous. By Proposition \ref{Prop: homotopy automorphism group of a minimal CDGA is the inverse limit of smaller pieces}, restriction induces a continuous homomorphism $\Gal(\overline{K}/K)\rightarrow \Aut_h(M(n))$ for every $n$. By Example \ref{Example: rigid invariants}, each indecomposable space $I^nM=I^nM(n)$ is a continuous representation of $\Gal(\overline{K}/K)$. By Corollary \ref{Cor: duality between minimal model and homotopy groups for pro spaces}, it suffices to show that each $I^nM(n)$ is successively de Rham.

The assertion is immediately true for $n=1$ since $I^1M(1)=0$. Inductively suppose that, for every $q<n$, $I^qM(q)$ is  successively de Rham, and that $H^s(M(q))$ is successively de Rham  for every $s$. 

Let $Z^{n+1}(M(n-1))$ and $Z^nM(n)$ be the subspaces of degree $n+1$ cocycles of $M(n-1)$ and degree $n$ cocycles of $M(n)$, respectively. The differential induces a canonical map $\alpha: I^nM(n)=\frac{M(n)^n}{M(n-1)^{>0}\cdot M(n-1)^{>0}}\xrightarrow{d} \frac{Z^{n+1}(M(n-1))}{d(M(n-1)^{>0}\cdot M(n-1)^{>0})}\rightarrow H^{n+1}(M(n-1))$. Since every map in this composition commutes with automorphisms of $M$, this is a $\Gal(\overline{K}/K)$-equivariant map. By the inductive construction process of minimal models, $\alpha$ factors through $I^nM(n)\twoheadrightarrow \Ker(H^{n+1}(M(n-1))\rightarrow H^{n+1}(M))$. 

On the other hand, the natural inclusion $Z^nM=Z^nM(n)\hookrightarrow M(n)$ induces $Z^nM\rightarrow I^nM(n)$. This map descends to $\beta: H^n(M)\rightarrow I^nM(n)$. Since the inclusion $Z^nM\hookrightarrow M(n)$ commutes with automorphisms of $M$, $\beta$ is a $\Gal(\overline{K}/K)$-equivariant map. 
The inductive construction process of minimal models shows that $\beta$ factors through $\Coker(H^n(M(n-1))\rightarrow H^n(M))\hookrightarrow I^nM(n)$. We therefore obtain a short exact sequence of continuous $\Gal(\overline{K}/K)$-representations 
\[
0\rightarrow \Coker(H^n(M(n-1))\rightarrow H^n(M))\rightarrow I^nM(n)\rightarrow \Ker(H^{n+1}(M(n-1))\rightarrow H^{n+1}(M))\rightarrow 0
\]

By Lemma \ref{Lem: Closure of algebraic operations of de Rham representations} and the inductive assumption, $I^nM(n)$ is a successively de Rham representation. Lemma \ref{Lem: successive de Rham on indecomposables implies the property for cohomology} then implies that $H^q(M(n))$ is successively de Rham for every $q$. This completes the induction and proves item (2). 
\end{proof}

The following result follows by the same proof as Theorem \ref{Thm: de Rham represenations for etale homotopy groups}.

\begin{theorem}\label{Thm: crystalline represenations for etale homotopy groups}
Let $K$ be a finite field extension of $\QQ_{p}$, and let $\mathcal{O}_{K}$ denote its ring of integers. Let $\chi$ be a connected smooth proper scheme over $\mathcal{O}_K$. Assume that $\chi_{K}$ has a $K$-rational point. 
\begin{enumerate}[leftmargin=0.25in]
    \item The Lie algebra of the continuous Mal'cev $\QQ_p$-completion of $\pi^{\et}_1(\chi_{\overline{K}})^{\wedge}_{p}/\Gamma^n$ is a successively crystalline representation of $\Gal(\overline{K}/K)$ for every $n$, and the Lie bracket respects the $\Gal(\overline{K}/K)$-action.
    \item If $\pi^{\et}_1(\chi_{\overline{K}})^{\wedge}_{p}=0$, then $\pi_n((\chi_{\overline{K}})^{\wedge}_{\et,p})\otimes_{\widehat{\ZZ}_{p}}\QQ_{p}$ is a successively crystalline representation of $\Gal(\overline{K}/K)$, for every $n\geq 2$.
\end{enumerate}
\end{theorem}

\subsection{Mal'cev Completions of \'Etale Fundamental Groups}\,

In this subsection, we study the $\ell$-adic \'etale fundamental group of varieties in positive characteristic. Our main result is an \'etale analogue of \cite{Morgan-hodge-theory}*{Corollary 10.4}.

Let $k$ be a field of characteristic zero. Let $M$ be a finite type $1$-minimal CDGA over $k$. Recall that $M$ admits an exhaustive tower of subalgebras $k=M(0)\subset M(1)\subset M(2)\subset ...$, where $M(1)=\bigwedge^* V(1)$ and each $M(n+1)$ is an elementary extension of $M(n)$ by $V(n+1)$ with $V(n+1)\cong \Ker(H^2(M(n))\twoheadrightarrow H^2(M))$. As in Theorem \ref{Fact: nilpotent Lie algebra induced from minimal model}, let $0\twoheadleftarrow L_1\twoheadleftarrow L_2\twoheadleftarrow ...$ denote the corresponding tower of nilpotent Lie algebras dual to $0\subset V(1)\subset V(2)\subset ...$. The inductive construction of $M(n)$ shows that each central extension $L_{n+1}\twoheadrightarrow L_n$ is non-splitting. Thus, the nilpotency class of each $L_n$ is exactly $n$. For a Lie algebra $L$, let $\Gamma^*L$ denote its lower central series.

\begin{lemma}\label{Lem: Correspondence between Lie algebra cohomology and minimal model}
Each $M(n)$ is canonically isomorphic to the Chevalley-Eilenberg cochain complex of $L_n$. In particular, $H^*(M(n))$ is isomorphic to the Lie algebra cohomology $H^*(L_n;k)$.
\end{lemma}

\begin{proof}
By construction, $M(n)\cong \bigwedge^* V(n)$. Thus, it is the underlying graded algebra of the Chevalley-Eilenberg cochain complex of $L_n$. Since the Lie bracket of $L_n$ is dual to $d: M(n)^1\rightarrow M(n)^2$ (Theorem \ref{Fact: nilpotent Lie algebra induced from minimal model}), the differentials also agree.
\end{proof}

\begin{lemma}\label{Lem: More precise statement of Lie algebra tower of a 1-minimal model}
Let $F$ be the free Lie algebra generated by the vector space $L_1$. Then, for each $n$, there is a surjective Lie-algebra map $\phi_n:F\twoheadrightarrow L_n$ that is compatible with the tower $0\twoheadleftarrow L_1\twoheadleftarrow L_2\twoheadleftarrow ...$ and satisfies $\Gamma^{n+1}F\subset \Ker(\phi_n)$.
\end{lemma}

\begin{proof}
Let $K_{n+1}=\Ker(L_{n+1}\twoheadrightarrow L_n)$. Suppose inductively that we have constructed $\phi_n:F\twoheadrightarrow L_n$. Since $F$ is free, $\phi_n$ lifts to a Lie-algebra map $\phi_{n+1}:F\rightarrow L_{n+1}$. Since $0\rightarrow K_{n+1}\rightarrow L_{n+1}\rightarrow L_{n}\rightarrow 0$ is a central extension, $K_{n+1}\subset [\phi_{n+1}(F)+K_{n+1},\phi_{n+1}(F)+K_{n+1}]=[\phi_{n+1}(F),\phi_{n+1}(F)]\subset \phi_{n+1}(F)$. Then $\phi_{n+1}$ is surjective. Since $\Gamma^{n+1}L_n=0$, $\phi_n(\Gamma^{n+1}F)\subset \Gamma^{n+1}L_n=0$. Thus, $\Gamma^{n+1}F\subset \Ker(\phi_n)$.
\end{proof}

Suppose now that the $1$-minimal CDGA $M$ is equipped with a weighted grading. This weight grading induces a compatible weight grading on each Lie algebra $L_n$ in the dual tower. The next lemma generalizes \cite{Morgan-hodge-theory}*{Theorem 9.4}.

\begin{lemma}\label{Lem: weighted minimal model and Lie algebra tower}
If the weights of $H^{1}(M)$ are at least $1$ and the weights of $H^2(M)$ are at most $q$, where $q\geq 2$, then there exists a homogeneous ideal $I$ of $F$, generated by elements of weights $-2,-3,...,-q$, such that the tower $0\twoheadleftarrow L_1\twoheadleftarrow L_2\twoheadleftarrow ...$ is isomorphic to the lower-central-series tower $0\twoheadleftarrow (F/I)/\Gamma^{2}\twoheadleftarrow (F/I)/\Gamma^{3}\twoheadleftarrow ...$ of $F/I$. 
\end{lemma}

\begin{proof}
Since $V(1)\cong H^1(M)$, the weights of $V(1)$ are at least $1$.
Every monomial in a representative cocycle for a class in $\Ker(H^2(M(n))\rightarrow H^2(M))$ is a product of an element of $V(n)$ and an element of $\bigoplus_{i=1}^{n-1} V(i)$. It follows inductively that every weight-homogeneous element of $V(n)$ has weight at least $n$. In particular, the weights of $\Ker(H^2(M(n))\twoheadrightarrow H^2(M))\cong V(n+1)$ are at least $n+1$. Moreover, once $n\geq q$, the subspace of $H^2(M(n))$ generated by classes of weights $2,...,q$ is independent of $n$. 

Set $J_n=\Ker(\phi_n:F\twoheadrightarrow L_n)$. Let $I_n\subset J_n$ denote the homogeneous sub-ideal generated by elements of weights $-2,...,-q$, and let $(J_n)^r$ denote the homogeneous subspace of weight $r$. Since $L_1$ is abelian and the projection $L_n\rightarrow L_1$ is surjective, we have $J_n\subset [F,F]$. Since every generator of $F$ has weight at most $-1$, every element of $J_n$ has weight at most $-2$ and $([J_n,F])^{-2}=0$. Hence $(\frac{J_{n}}{[J_{n},F]})^{-2}=(J_n)^{-2}$.

The Lie algebra cohomology theory identifies $H^2(F/J_n;k)=H^2(L_n;k)$ with the dual of $J_n/[J_n,F]$. By Lemma \ref{Lem: Correspondence between Lie algebra cohomology and minimal model}, there is a natural isomorphism  $H^2(L_n;k)\cong H^2(M(n))$. Thus, the natural map $(\frac{J_{n+1}}{[J_{n+1},F]})^{-r}\rightarrow (\frac{J_{n}}{[J_{n},F]})^{-r}$ is an isomorphism for $2\leq r\leq q$ and $n\geq q$. Then we have $(J_n)^{-2}=(J_{n+1})^{-2}$ since $(\frac{J_{t}}{[J_{t},F]})^{-2}=(J_t)^{-2}$. Suppose inductively that $(J_{n+1})^{-s}=(J_{n})^{-s}$ for every $2\leq s<r$, where $r\leq q$ and $n\geq q$. Since $([J_n,F])^{-r}=\sum_{a+b=r,a\geq 2, b\geq 1}[(J_n)^{-a},F^{-b}]$, the inductive hypothesis gives $([J_n,F])^{-r}=([J_{n+1},F])^{-r}$. Now let $x\in (J_n)^{-r}$. The isomorphism $(\frac{J_{n+1}}{[J_{n+1},F]})^{-r}\rightarrow (\frac{J_{n}}{[J_{n},F]})^{-r}$ shows that there exists $y\in (J_{n+1})^{-r}$ such that $x-y\in ([J_n,F])^{-r}=([J_{n+1},F])^{-r}$. Then $x\in J_{n+1}$. It follows that $(J_n)^{-r}=(J_{n+1})^{-r}$.

Then we obtain a stabilization $I_q=I_{q+1}=...$. Let $I$ denote this stable homogeneous ideal. Choose a weight-homogeneous subspace $W_n\subset J_n$ that projects onto $J_n/J_{n+1}$. Since $L_{n+1}\twoheadrightarrow L_n$ is a central extension, $J_n/J_{n+1}$ is central in $F/J_{n+1}$. Thus, $[J_n,F]\subset J_{n+1}$. This gives an exact sequence $0\rightarrow J_{n+1}/[J_n,F]\rightarrow J_n/[J_n,F]\rightarrow J_n/J_{n+1}\rightarrow 0$. Under the natural duality $(J_n/[J_n,F])^{\vee}\cong H^2(M(n))$, $J_n/J_{n+1}$ is dual to $V(n+1)\cong \Ker(H^2(M(n))\rightarrow H^2(M))$. Since $H^2(M)$ has weights at most $q$ and $V(n+1)$ has weights at least $n+1>q$, $V(n+1)$ is precisely the part of $H^2(M(n))$ of weights greater than $q$. It follows that $J_{n+1}/[J_n,F]$ consists precisely of the components of weights $-2,...,-q$, whose representatives generate $I$. Then $J_n=I+W_n+[J_n,F]$. Let $K_n$ denote the ideal generated by $I+W_n$. Suppose inductively that $(J_n)^{-d}\subset K_n$. Let $x\in (J_n)^{-d}$. Then $x=k+\sum_i[y_i,z_i]$ for some $k\in (K_n)^{-d}$, $y_i\in (J_n)^{-a_i}$ and $z_i \in F^{b_i}$, where $a_i+b_i=d$. Since every $b_i\geq 1$, we have $a_i<d$. The inductive hypothesis implies that $y_i\in K_n$. Therefore, $x\in K_n$. This proves that $J_n$ is generated, as an ideal of $F$, by $I$ and $W_n$.

Since each central extension $L_{n+1}\twoheadrightarrow L_n$ is non-splitting, $J_n/J_{n+1}\cong \Gamma^{n+1}(L_{n+1})=\Gamma^{n+1}(F/J_{n+1})$. 
Under this identification, $\Gamma^{n+1}F\subset J_{n}$ projects onto  $J_n/J_{n+1}$. The previous paragraph implies that $J_n=I+\Gamma^{n+1}(F)$. Hence, for $n\geq q$, $L_n\cong F/(I+\Gamma^{n+1}(F))\cong (F/I)/\Gamma^{n+1}(F/I)$; for $n<q$, $L_n\cong (F/I)/\Gamma^{n+1}(F/I)$ since $L_n\cong L_q/\Gamma^{n+1}(L_q)$.
\end{proof} 

The \textbf{pro-nilpotent completion} of a Lie algebra $L$ is  $\widehat{L}:=\varprojlim_r L/\Gamma^r L$. The following result answers the fundamental aspect of Question \ref{Question: main question}.

\begin{theorem}\label{Thm: fundamental group of algebraic varieties}
Let $k$ be a separably closed field of characteristic $p\geq 0$, and let $\ell$ be a prime number. Let $X$ be a connected smooth variety over $k$. 
\begin{enumerate}[leftmargin=0.25in]
    \item If $X$ is proper, then the Lie algebra of the continuous Mal'cev $\QQ_{\ell}$-completion of $\pi_1^{\et}(X)^{\wedge}_{\ell}$ is the pro-nilpotent completion of a free Lie algebra on a finite-dimensional $\QQ_{\ell}$-vector space modulo an ideal generated by elements of bracket length $2$. In particular, $\pi_1^{\et}(X)^{\wedge}_{\ell}\widehat{\otimes }\QQ_{\ell}$ is determined by $(\pi_1^{\et}(X)^{\wedge}_{\ell}/\Gamma^3)\widehat{\otimes}\QQ_{\ell}$.
    \item If $\ell\neq p$ and $k$ is the algebraic closure of a finite field, then the Lie algebra of the continuous Mal'cev $\QQ_{\ell}$-completion of  $\pi_1^{\et}(X)^{\wedge}_{\ell}$ is the pro-nilpotent completion of a free Lie algebra on a finite-dimensional $\QQ_{\ell}$-vector space modulo an ideal generated by elements of bracket lengths $2$, $3$ and $4$. In particular, $\pi_1^{\et}(X)^{\wedge}_{\ell}\widehat{\otimes }\QQ_{\ell}$ is determined by $(\pi_1^{\et}(X)^{\wedge}_{\ell}/\Gamma^5)\widehat{\otimes}\QQ_{\ell}$.
\end{enumerate}
\end{theorem}

\begin{proof}
We prove item (1) in detail, since the proof of item (2) is analogous. The characteristic $0$ case is essentially \cite{Morgan-hodge-theory}*{Corollary 10.4}. We may therefore assume that $p>0$.
There exists a commutative ring $S$ which is finite type over $\ZZ$ and a smooth proper scheme $\chi$ over $S$ such that the fraction field $F$ of $S$ is contained in $k$ and $\chi_{k}=X$. For a closed point $s\in \Spec(S)$, let $k_s$ denote its residue field.

Suppose first that $\ell\neq p$. \cite{Artin-Mazur-etale-homotopy}*{Corollary 12.13} gives a homotopy equivalence $X^{\wedge}_{\et,\ell}\simeq (\chi_{\overline{k}_s})^{\wedge}_{\et,\ell}$. By the purity of Frobenius weights of $H^*_{\et}(\chi_{\overline{k}_s};\QQ_{\ell})$ (\cite{Deligne-Weil-II}), this theorem follows from Lemma \ref{Lem: weighted minimal model and Lie algebra tower}.

Now suppose that $\ell=p$. We apply an argument analogous to the proof of Theorem \ref{Thm: formality of varieties}. We retain the notation introduced in that proof. By \cite{Gros-Suwa}*{Chapter II, Theorem 2.1}\cite{Jannsen-weights}*{(3.2)}, one can choose a closed point $s\in \Spec(S)$ such that, for every $r$, there is a Galois-equivariant isomorphism between $H^r_{\et}(\chi_{\overline{F}};\QQ_{p})$ and  $H^r_{\et}(\chi_{\overline{k_s}};\QQ_{p})$. Choose a lifting $\phi\in G_s$ of the arithmetic Frobenius of $k_s$. Then each eigenvalue of the lift  $(\phi^{-1})^*$ of the Frobenius on $H^r_{\et}(\chi_{\overline{F}};\QQ_{p})$ is an algebraic number of pure weight $r$. Applying Lemma \ref{Lem: weighted minimal model and Lie algebra tower} proves the theorem for $\chi_{\overline{F}}$. By \cite{Artin-Mazur-etale-homotopy}*{Corollary 12.12}, $(\chi_{\overline{F}})^{\wedge}_{\et,p}\simeq (\chi_{k})^{\wedge}_{\et,p}=(X_{\et})^{\wedge}_p$. Then this theorem holds for $X$. 
\end{proof}

\appendix

\section{Commutative Differential Graded Algebras}

In this appendix, we recall minimal models of commutative differential graded algebras over fields of characteristic zero, including the Lie algebra structures dual to their differentials. Our main references are \cite{Deligne-Griffiths-Morgan-Sullivan-formality}\cite{MGriffiths-Morgan-rational-homotopy}\cite{Morgan-hodge-theory}\cite{Sullivan-rational-homotopy}. We also prove some results of homotopy classes of automorphisms of a minimal model, which is used in the main context.

\begin{definition}
A \textbf{commutative differential graded algebra (CDGA)} over $k$ is a graded associative $k$-algebra $C^*=\bigoplus_{s\geq 0}C^s$ equipped with a differential $d$ of degree $1$ satisfying the following axioms for all homogeneous $x,y\in C^*$:
\begin{enumerate}[leftmargin=0.25in]
    \item $x\cdot y=(-1)^{|x|\cdot |y|}y\cdot x$,
    \item $d(x\cdot y)=dx\cdot y+(-1)^{|x|}x\cdot dy$, 
    \item $d^2=0$,
\end{enumerate}
where $|x|$ denotes the degree of $x$.
\end{definition}

\begin{definition}\label{Def: finite type CDGA}
A CDGA $C^*$ is \textbf{connected} if $C^0=k$. It is \textbf{homologically connected} if $H^0(C^*)=k$. It is \textbf{simply-connected} if $H^0(C^*)=k$ and $H^1(C^*)=0$. It is \textbf{finite type} if $H^i(C^*)$ is finite-dimensional for every $i$.
\end{definition}

\begin{definition}
Let $C^*$ be a connected CDGA. Its \textbf{augmentation ideal} is $C^{>0}:=\bigoplus_{s>0} C^s$. Its \textbf{indecomposable space} is the quotient $I^*(C^*)=C^{>0}/(C^{>0}\cdot C^{>0})$.
\end{definition}

\begin{definition}
Let $V^*$ be a graded $k$-vector space. Write $V^{ev}:=\bigoplus_{i} V^{2i}$ and $V^{odd}:=\bigoplus_i V^{2i+1}$. The \textbf{free graded algebra} generated by $V$ is $\bigwedge^*V:=k[V^{ev}]\otimes_k \bigwedge^* V^{odd}$, equipped with its natural grading.
\end{definition}

\begin{definition}
Let $(C^*,d_{C})$ be a CDGA, and let $V$ be a graded space over $k$. An \textbf{elementary extension of $C^*$} by $V$ is a CDGA of the form $(C^*\otimes_{k} \bigwedge^* V,d)$ satisfying $d\vert_{C}=d_{C}$ and $d(V)\subset C^*$.
\end{definition}

\begin{definition}[\cite{MGriffiths-Morgan-rational-homotopy}*{p.~97, Definition; Proposition 10.2}]\label{Def: minimal CDGA}
Let $C^*$ be a connected CDGA, and let $C^*(n)$ be the subalgebra generated by homogeneous elements of degree at most $n$. The CDGA $C^*$ is \textbf{minimal} if it is free and, for each $n$, there exists a tower of CDGA's $C(n-1)=C(n,0)\subset C(n,1)\subset C(n,2)\subset ...$ satisfying the following conditions: 
\begin{enumerate}[leftmargin=0.27in]
    \item $\bigcup_{i}C(n,i)=C(n)$;
    \item each $C(n,i)$ is an elementary extension of $C(n,i-1)$ by a graded vector space $V(n,i)$ concentrated in degree $n$.
\end{enumerate}
If $n$ is the smallest number such that $C^*=C(n)$ and $C^*$ is minimal, then the CDGA $C^*$ is called an \textbf{$n$-minimal CDGA}. 
\end{definition}

The following lemma follows immediately from Definition \ref{Def: minimal CDGA}.

\begin{lemma}
Keep the notation of Definition \ref{Def: minimal CDGA}. If $C^*$ is minimal, then the composition $\bigoplus_i V(n,i)\hookrightarrow C^n\twoheadrightarrow I^n(C)$ is an isomorphism of vector spaces.
\end{lemma}

\begin{definition}\label{Def: n-minmal model}
An \textbf{$n$-minimal model} of a CDGA $C^*$ consists of a minimal CDGA $M(n)$ together with a morphism $\phi_n:M(n)\rightarrow C$ satisfying the following conditions:
\begin{enumerate}[leftmargin=0.25in]
    \item The CDGA $M(n)$ is generated by homogeneous elements of degree at most $n$;
    \item The induced map $(\phi_n)^*:H^i(M(n))\rightarrow H^i(C)$ is an isomorphism for $i\leq n$ and is injective for $i=n+1$.
\end{enumerate}
When $n=\infty$, the pair $(M(\infty),\phi_{\infty})$ is called a minimal model of $C$.
\end{definition}

\begin{definition}
A \textbf{homotopy} between two CDGA morphisms $f,g:C\rightarrow D$ is a CDGA morphism $H:C\rightarrow D\otimes (t,dt)$ such that $H\vert_{t=0}=f$ and $H\vert_{t=1}=g$, where $t$ has degree $0$ and $dt$ has degree $1$.
\end{definition}

\begin{theorem}[\cite{Morgan-hodge-theory}*{Theorem 5.6},\cite{Bousfield-Gugenheim-rational-homotopy}*{Proposition 7.7, Proposition 7.8}]\label{Fact: Lifting of minimal models}
Let $C^*$ be a homologically connected CDGA. Then
\begin{enumerate}[leftmargin=0.25in]
    \item $C^*$ has an $n$-minimal model $\phi_n:M(n)\rightarrow C$ for any $1\leq n\leq \infty$;
    \item for any two $n$-minimal models $\phi_n:M(n)\rightarrow C$ and $\phi'_n:M(n)'\rightarrow C$, there exists a CDGA isomorphism $I:M(n)\rightarrow M(n)'$, unique up to homotopy, such that $\phi_n$ and $\phi'_n\circ I$ are homotopic;
    \item moreover, if $C^*$ is finite type, then an $n$-minimal model $M(n)$ has a tower of elementary extensions, as in Definition \ref{Def: minimal CDGA}, by finite-dimensional vector spaces.
\end{enumerate}
\end{theorem}

\begin{theorem}[\cite{Morgan-hodge-theory}*{Proposition 5.9}]\label{Fact: Uniqueness of minimal models}
Let $M$ and $M'$ be $n$-minimal CDGA's, where $1\leq n\leq \infty$. If a CDGA morphism $I:M\rightarrow M'$ induces an isomorphism $H^i(M)\rightarrow H^i(M')$ for every $i
\leq n$ and an injection $H^{n+1}(M)\rightarrow H^{n+1}(M')$, then $I$ is an isomorphism.
\end{theorem}

Let $M$ be a finite type minimal CDGA. By Theorem \ref{Fact: Lifting of minimal models}, $M$ has a tower $k=M(1,0)\subset M(1,1)\subset M(1,2)\subset ...\subset M(2,0)\subset  M(2,1)\subset ...$ of elementary extensions by finite-dimensional vector spaces $\{V(n,i)\}$, where $M(n+1,0)=M(n)=\bigcup_i M(n,i)$. Let $L(n,i)$ denote the dual of $\bigoplus_{1\leq j\leq i}V(n,j)$, and let $L(n)$ denote the dual of $\bigoplus_{j}V(n,j)$. The restriction of the differential $d:M^n\rightarrow M^{n+1}$ to $\bigoplus_j V(n,j)$ is a polynomial on $\bigoplus_{i,m<n}V(m,i)$. Its quadratic term induces brackets $[\cdot,\cdot]:L(m)\otimes_k L(n)\rightarrow L(m+n-1)$ and $[\cdot,\cdot]:L(1,i)\otimes L(1,i)\rightarrow L(1,i)$.

\begin{theorem}[\cite{Morgan-hodge-theory}*{5.11}\cite{MGriffiths-Morgan-rational-homotopy}*{Theorem 10.3}]\label{Fact: nilpotent Lie algebra induced from minimal model}
Let $M$ be a finite type minimal CDGA with the tower in the previous paragraph.
\begin{enumerate}[leftmargin=0.25in]
    \item The bracket makes $(L(*-1),[\cdot,\cdot])$ a graded Lie algebra.
    \item For every $i$, $(L(1,i),[\cdot,\cdot])$ is a finite-dimensional nilpotent Lie algebra. Moreover, the dual tower to $k=M(1,0)\subset M(1,1)\subset M(1,2)\subset ...$ is the tower of central extensions of Lie algebras $0 \twoheadleftarrow L(1,1)\twoheadleftarrow L(1,2)\twoheadleftarrow ...$.
\end{enumerate}
\end{theorem}

\begin{definition}[\cite{Morgan-corrections}]
An invariant of a minimal CDGA is \textbf{rigid} if homotopic maps induce the identical map on the invariant. 
\end{definition}

\begin{example}[\cite{Morgan-corrections}]\label{Example: rigid invariants}
The indecomposable space of a minimal CDGA $M$ is generally not rigid, unless $M$ is simply-connected. 
\end{example}

\begin{definition}
Let $f,g:A\rightarrow B$ be CDGA morphisms.
\begin{enumerate}[leftmargin=0.25in]
    \item A degree $-1$ map $\theta:A\rightarrow B$ is an \textbf{$(f,g)$-derivation} if $\theta(xy)=\theta(x)g(y)+(-1)^{|x|}f(x)\theta(y)$ for all homogeneous elements $x,y\in A$.
    \item Assume that $A$ is minimal. A degree $-1$ $(f,g)$-derivation $\theta:A\rightarrow B$ is a \textbf{homotopy derivation} from $f$ to $g$ if $g-f =d\theta+\theta d$.
\end{enumerate}
\end{definition}

\begin{lemma}\label{Lem: equivalence of two homotopies for maps between minimal CDGA's}
Let $f,g:A\rightarrow B$ be CDGA morphisms, and assume that $A$ is minimal. Then $f$ is homotopic to $g$ if and only if there exists a homotopy derivation $\theta:A\rightarrow B$ from $f$ to $g$.
\end{lemma}

\begin{proof}
Since $A$ is minimal, there is a graded vector subspace $V$ and a well-ordered basis $\{x_i\}_{i\in I}$ of $V$, such that $A=\bigwedge^*(V)$, and, for each $i$,  $d(x_i)$ is a combination of smaller $x_j$'s.

\textbf{``Only if'' part.} Let $H:A\rightarrow B\otimes (t,dt)$ be a homotopy from $f$ to $g$. We inductively construct a degree $-1$ $(f,H)$-derivation $\eta:A\rightarrow B\otimes (t,dt)$ such that $H-f=d\eta+\eta d$ and $\eta\vert_{t=0}=0$. Once such an $\eta$ has been constructed, the evaluation $\theta=\eta\vert_{t=1}$ is a degree $-1$ $(f,g)$-derivation we need.

Suppose inductively that we have constructed such an $\eta$ on the subalgebra $A_{<i}=\bigwedge^* \langle x_{j}:j< i\rangle$. Then $d(H(x_i)-f(x_i)-\eta(dx_i))=H(dx_i)-f(dx_i)-d\eta(dx_i)-\eta(d^2x_i)=0$. Furthermore, the evaluation at $t=0$ of $H(x_i)-f(x_i)-\eta(dx_i)$ is $0$. Since $\Ker(ev_0:B\otimes (t,dt)\rightarrow B)$ is acyclic, there exists an element $y_i$ in this kernel such that $dy_i=H(x_i)-f(x_i)-\eta(dx_i)$. Define $\eta(x_i)=y_i$, and extend $\eta$ to the subalgebra $A_{\leq i}$ as a degree $-1$ $(f,H)$-derivation. Continuing inductively over the well-ordered basis constructs the desired derivation $\eta$.

\textbf{``If'' part.} Let $\theta:A\rightarrow B$ be a degree $-1$ $(f,g)$-derivation such that $g-f=d\theta+\theta d$. Suppose inductively that we have constructed a homotopy $H:A_{<i}\rightarrow B$ from $f\vert_{A_{<i}}$ to $g\vert_{A_{<i}}$, together with a degree $-1$ $(f,H)$-derivation $\eta:A_{<i}\rightarrow B\otimes (t,dt)$, such that $\eta\vert_{t=0}=0$ and $\eta\vert_{t=1}=\theta$. Define $\eta(x_i):=t\theta(x_i)$ and $H(x_i):=f(x_i)+d\eta(x_i)+\eta(dx_i)$. Then $\eta\vert_{t=0}(x_i)=0$, $\eta\vert_{t=1}(x_i)=\theta(x_i)$, $H\vert_{t=0}(x_i)=f(x_i)$, and $H\vert_{t=1}(x_i)=g(x_i)$. Extend $\eta$ to $A_{\leq i}$ as a degree $-1$ $(f,H)$-derivation, and extend $H$ to $A_{\leq i}$ as a CDGA morphism. Proceeding inductively yields a homotopy from $f$ to $g$.
\end{proof}

\begin{lemma}\label{Lem: rectification of homotopic automorphisms for elementary extensions}
Let $A$ be a minimal CDGA, and let $B$ be an elementary extension of $A$ by a graded vector space $V_n$ concentrated in degree $n$. Suppose that $\alpha$ and $\beta$ are automorphisms of $A$ and $B$, respectively, such that $\beta\vert_{A}$ is homotopic to $\alpha$. Then there exists an automorphism $\beta'$ of $B$ such that $\beta'$ is homotopic to $\beta$ and $\beta'\vert_{A}=\alpha$. 
\end{lemma}

\begin{proof}
Let $H_A:A\rightarrow A\otimes (t,dt)$ be a homotopy from $\beta\vert_{A}$ to $\alpha$. Composing with the natural inclusion $A\otimes (t,dt)\hookrightarrow B\otimes (t,dt)$, we may regard $H_A$ as a map into $B\otimes (t,dt)$. If $H_A$ extends to $H:B\rightarrow B\otimes (t,dt)$ satisfying $H{\vert}_{t=0}=\beta$, then $H\vert_{t=1}$ is an automorphism of $B$ by \cite{MGriffiths-Morgan-rational-homotopy}*{Lemma 11.7}. Choose a basis $\{v_i\}_{i\in I}$ of $V$. Since  $\Ker(ev_0:B\otimes (t,dt)\rightarrow B)$ is acyclic, there exists $x_i\in \Ker(ev_0)$ such that $dx_i=H_A(dv_i)-d\beta(v_i)$. Define $H(v_i):=\beta(v_i)+x_i$. Then $dH(v_i)=H_A(dv_i)$, so the assignment is compatible with the differential. Therefore, these assignments extend uniquely to a CDGA morphism $H:B\rightarrow B\otimes (t,dt)$. Then $H\vert_{t=0}=\beta$ and $H\vert_{A}=H_A$. By construction, $\beta':=H\vert_{t=1}$ satisfies the requirement of the lemma.
\end{proof}

For a CDGA $C$, let $\Aut(C)$ denote the group of automorphisms of $C$, and let $\Aut_h(C)$ denote the group of homotopy classes of automorphisms of $C$.

\begin{proposition}\label{Prop: homotopy automorphism group of a minimal CDGA is the inverse limit of smaller pieces}
Let $M$ be a minimal CDGA over $k$ with $M^n$ finite-dimensional for every $n$. Let $M(n)$ be the subalgebra of $M$ generated by elements of degree at most $n$. Then 
\begin{enumerate}[leftmargin=0.25in]
    \item the restriction induces homomorphisms $\Aut_h(M)\rightarrow \Aut_h(M(n))$ and $\Aut_h(M(m))\rightarrow \Aut_h(M(n))$ for every $m\geq n$;
    \item the homomorphisms in item (1) induce an isomorphism $\Aut_h(M)\rightarrow \varprojlim_n \Aut_h(M(n))$.
\end{enumerate}
\end{proposition}

\begin{proof}
\textbf{Item (1).} It suffices to prove that the restriction map $\Aut_h(M)\rightarrow \Aut_h(M(n))$ is well defined. We only need to show that, for any automorphism $f:M\rightarrow M$ together with a homotopy $H:M\rightarrow M\otimes (t,dt)$ from $f$ to $Id_M$, $f\vert_{M(n)}$ is homotopic to $Id_{M(n)}$. This follows immediately from the inclusion $H(M(n))\subset M(n)\otimes (t,dt)$.

\textbf{Item (2).} We first prove surjectivity. Let $([f_n])_n$ be an element of $\varprojlim_n \Aut_h(M(n))$. By Lemma \ref{Lem: rectification of homotopic automorphisms for elementary extensions} we may inductively construct automorphisms $g_n:M(n)\rightarrow M(n)$ such that $(g_n)\vert_{M(n-1)}=g_{n-1}$ and $g_{n}$ is homotopic to $f_n$. Then the union of $g_n$ defines an automorphism of $M$.

To prove injectivity, let $f$ be an automorphism of $M$ whose restriction to each $M(n)$ is homotopic to $Id_{M(n)}$. Let $H(n)$ denote the set of all homotopy derivations of $M(n)$ from $f\vert_{M(n)}$ to $Id_{M(n)}$. By Lemma \ref{Lem: equivalence of two homotopies for maps between minimal CDGA's}, $H(n)\neq \emptyset$. Choose a graded vector subspace $V$ of $M$ such that $M\cong \bigwedge^* V$. Then each $V^n$ is finite-dimensional. 

Choose a homogeneous basis $\{x_1,...,x_r\}$ of $V^{\leq n}$. Evaluation on $\{x_1,...,x_r\}$ identifies $H(n)$ with an affine subspace of $\prod_{i=1}^r M^{|x_i|-1}$. For every $m\geq n$, the restriction induces a map $H(m)\rightarrow H(n)$. Let $L_{m,n}\subset H(n)$ denote the image of $H(m)\rightarrow H(n)$. Then each $L_{m,n}$ is an affine subspace of $H(n)$. Since each $M^n$ is finite-dimensional, $H(n)$ is finite-dimensional. Thus, the descending sequence $L_{n,n}\supset L_{n+1,n}\supset L_{n+2,n} \supset...$ eventually stabilizes. Let $L_{\infty,n}$ denote this stable value.

Since the restriction $L_{\infty,n+1}\rightarrow L_{\infty,n}$ is surjective, we may inductively choose $\theta_n\in L_{\infty,n}$ such that $\theta_{n+1}\vert_{M(n)}=\theta_n$ for any $n$. The compatible family $\{\theta_n\}_n$ defines a homotopy derivation $\theta:M\rightarrow M$ from $f$ to $Id_{M}$. This proves injectivity and completes the proof.
\end{proof}

The following result is a direct consequence of Proposition \ref{Prop: homotopy automorphism group of a minimal CDGA is the inverse limit of smaller pieces} and \cite{Sullivan-rational-homotopy}*{Theorem 6.1}.

\begin{corollary}
Let $M$ be a minimal CDGA with $M^n$ finite-dimensional for every $n$. Then     $\Aut_h(M)$ is a pro-algebraic group over $k$.
\end{corollary}

\begin{lemma}\label{Lemma: rigidity of automorphisms of a minimal model}
Let $M$ be a finite type minimal CDGA. Then there exists an exhaustive tower $k=M_0\subset M_1\subset ...$ of finitely generated sub-CDGA's of $M$ such that every automorphism $\sigma$ of $M$ is homotopic to an automorphism which restricts to an automorphism of each $M_i$. In particular, every eigenvalue of $\sigma$ is a product of eigenvalues of the induced automorphism $\sigma^*$ on $H^*(M)$. 
\end{lemma}

\begin{proof}
We construct a tower of minimal CDGA's $0=N_0\subset N_1\subset N_2\subset ...$, together with compatible inclusions of $N_n\hookrightarrow M$, satisfying the following properties: the union of these inclusions is an isomorphism; each pair $N_i\subset N_{i+1}$ is an elementary extension by a finite-dimensional graded vector space; and every automorphism $\phi$ of $M$ lifts to automorphisms $\psi_n$ of $N_n$ such that the left square of the following diagram commutes and the right square commutes up to homotopy.
\[
\begin{tikzcd}
    N_{n} \arrow[r,hook] \arrow[d,"\psi_{n}"] & N_{n+1} \arrow[r] \arrow[d,"\psi_{n+1}"] & M \arrow[d,"\phi"] \\
    N_{n} \arrow[r,hook] & N_{n+1} \arrow[r] & M
\end{tikzcd}
\]

Suppose inductively that we have constructed $f_n:N_n\hookrightarrow M$ such that $N_n$ is generated by elements of degree at most $n$, and the induced map  $H^q(N^n)\rightarrow H^q(M)$ is surjective for every $q\leq n$, and the compatibility properties stated above hold.

Let $\phi$ be an automorphism of $M$. By the inductive hypothesis, there exists an automorphism $\psi_n$ of $N_n$ and a homotopy $H_n$ from $\phi\circ f_n$ to $f_n\circ \psi_n$. Choose a linear complement $C_{n+1}\subset H^{n+1}(M)$ to the image of $H^{n+1}(N_n)\rightarrow H^{n+1}(M)$. There is an elementary extension $N'_{n+1}$ of $N_n$ by $C_{n+1}$ such that $d(C_{n+1})=0$, together with a morphism $g_{n+1}:N'_{n+1}\rightarrow M$ extending $f_n:N_{n}\rightarrow M$ and inducing a surjection $H^{n+1}(N'_{n+1})\rightarrow H^{n+1}(M)$. 
By \cite{MGriffiths-Morgan-rational-homotopy}*{Proposition 11.1}, the obstruction to extending $\psi_n$ to an automorphism $\alpha_{n+1}:N'_{n+1}\rightarrow N'_{n+1}$  and extending $H_n$ to a homotopy $H'_{n+1}$ from $\phi\circ g_{n+1}$ to  $g_{n+1}\circ \alpha_{n+1}$ is represented by a cocycle $Q\in \homo_k(C_{n+1},(N'_{n+1})^{n+2}\oplus M^{n+1})$ given by $Q(v)=(0,g_{n+1}(v))$. Define $P:C_{n+1}\rightarrow (N'_{n+1})^{n+1}\oplus M^{n}$ by $P(v)=(v,0)$. Since $d(P)=Q$, the obstruction vanishes. Hence, $\alpha_{n+1}$ and $H_{n+1}'$ exist.

Now define the graded vector space $V_{n+1}:=\bigoplus_{q=1}^{n+1}\Ker(H^q(N'_{n+1})\rightarrow H^q(M))$. There is an elementary extension $N_{n+1}$ of $N'_{n+1}$ by $V_{n+1}$ , together with a morphism $f_{n+1}:N_{n+1}\rightarrow M$ extending $g_{n+1}$ and inducing a surjection $H^{q}(N_{n+1})\rightarrow H^{q}(M)$ for every $q\leq n+1$. We show that the obstruction to extending $\alpha_{n+1}$ to an automorphism $\psi_{n+1}:N_{n+1}\rightarrow N_{n+1}$ and extending $H'_{n+1}$ to a homotopy $H_{n+1}$ from $\phi\circ f_{n+1}$ to  $f_{n+1}\circ \psi_{n+1}$ vanishes. It is enough to consider the case in which $V_{n+1}$ is concentrated in a single degree, say $q$. 
The obstruction is represented by a cocycle $Q'\in \homo_k(V_{n+1},(N_{n+1})^{n+2}\oplus M^{n+1})$ given by $Q'(v)=(f_{n+1}(dv),f_{n+1}(v)+\int_{0}^1( H'_{n+1}\circ \alpha_{n+1}^{-1})(dv))$ for every $v\in V_{n+1}$. By the definition of $H'_{n+1}$, $-d(\int_{0}^1(H'_{n+1}\circ \alpha^{-1}_{n+1})(v))=\int_{0}^1(H'_{n+1}\circ \alpha^{-1}_{n+1})(dv)+f_{n+1}(dv)-(\phi\circ f_{n+1}\circ \alpha_{n+1}^{-1})(dv)$. Since $H^{q+1}(N'_{n+1})\rightarrow H^{q+1}(M)$ is surjective, there exists $\beta_v\in N'_{n+1}$ and $\gamma_v\in M^{n}$ such that $f_{n+1}(\beta_v)=d(\int_{0}^1(H'_{n+1}\circ \alpha^{-1}_{n+1})(v))+d\gamma_v$. It follows that $d(\alpha^{-1}_{n+1}(v)-\alpha^{-1}_{n+1}(\beta_v),-\gamma_v)=Q'(v)$ for every $v\in V_{n+1}$. Choose a basis $x_1,...,x_r$ of $V_{n+1}$, and define $P':V_{n+1}\rightarrow (N_{n+1})^{n+1}\oplus M^{n}$ by $P'(x_i)=(\alpha^{-1}_{n+1}(x_i)-\alpha_{n+1}^{-1}(\beta_{x_i}),-\gamma_{x_i})$. Then $dP'=Q'$, so the obstruction vanishes. This completes the proof.
\end{proof}

\begin{corollary}\label{Cor: automorphism of 1-minimal CDGA to automorphisms of its filtration}
Retain the same notation of Lemma \ref{Lemma: rigidity of automorphisms of a minimal model}. Suppose, in addition, that $M$ is a $1$-minimal CDGA. Then the exhaustive tower of sub-CDGA's in Lemma \ref{Lemma: rigidity of automorphisms of a minimal model} has the following properties: if an automorphism $f$ of $M$ is homotopic to $Id_{M}$ and preserves each $M_n$, then its restriction $f_n$ to $M_n$ is homotopic to $Id_{M_n}$ for every $n$. In particular, the restriction induces a canonical isomorphism $\Aut_h(M)\rightarrow \varprojlim_n\Aut_h(M_n)$. Moreover, $\Aut_h(M)$ is a pro-algebraic group over $k$.
\end{corollary}

\begin{proof}
Let $\theta$ be a degree $-1$ $(f,Id_M)$-derivation of $M$. Then, for any $v_1,...,v_r\in M^1$, $\theta(v_1\cdot...\cdot v_r)$ is a sum of products involving $f(v_i)$, $v_i$ and $\theta(v_i)$. Thus, $\theta(M_n)\subset M_n$. This proves that each $f_n$ is homotopic to $Id_{M_n}$. Then the restriction induces a well-defined homomorphism $\Aut_h(M)\rightarrow \Aut_h(M_n)$. Proposition \ref{Prop: homotopy automorphism group of a minimal CDGA is the inverse limit of smaller pieces} now implies that $\Aut_h(M)\rightarrow \varprojlim_n\Aut_h(M_n)$ is an isomorphism. \cite{Sullivan-rational-homotopy}*{Theorem 6.1} shows that $\Aut_h(M)$ is a pro-algebraic group over $k$.
\end{proof}

\section{$E_{\infty}$-algebras over a Field of Characteristic Zero}

In this appendix, we briefly review the theory of $E_{\infty}$-algebras from \cite{Hinich-Model-Structure-on-homotopy-algbras}\cite{Kriz-May-operads-algebras-modules-motives}\cite{May-operads-sheaf-cohomology} and the rectification of $E_{\infty}$-algebras over fields of characteristic zero to CDGA's. Let $R$ be a commutative ring, and let $\mathbf{Ch}(R)$ be the category of cochain complexes of $R$-modules.

\begin{definition}
An operad $\mathcal{O}$ in $\mathbf{Ch}(R)$ is an \textbf{$E_{\infty}$-operad} if each $\mathcal{O}(n)$ is acyclic and $\mathcal{O}(n)$ is a free $R[\Sigma_n]$-module for every $n$.
\end{definition}

For an operad $\mathcal{O}$, let $\mathbf{Alg}(\mathcal{O})$ be the category of $\mathcal{O}$-algebras.

\begin{theorem}[\cite{Hinich-Model-Structure-on-homotopy-algbras}*{Theorems 4.1.1, 4.6.4, 4.7.4 and the paragraph above 4.2.1}\cite{White-E-infinity-algebra-characteristic-zero}*{Corollary 4.10}]\label{Fact: rectification}
Let $\mathcal{COM}_R$, defined by $\{\mathcal{COM}_R(n)=R\}$, denote the commutative operad over $R$, and let $\mathcal{E}$ be an $E_{\infty}$-operad over $R$. Let $f:\mathcal{E}\rightarrow \mathcal{COM}_R$ be a quasi-isomorphism of operads. Assume $\QQ\subset R$.
\begin{enumerate}[leftmargin=0.25in]
    \item The categories $\mathbf{Alg}(\mathcal{COM}_R)$ and $\mathbf{Alg}(\mathcal{E})$ of CDGA's over $R$ and $\mathcal{E}$-algebras, respectively, are closed model categories in which the weak equivalences are quasi-isomorphisms and the fibrations are degree-wise surjections;
    \item The following induced pair of adjoint functors forms a Quillen equivalence; \[
    f^*:\mathbf{Alg}(\mathcal{E})\rightleftarrows \mathbf{Alg}(\mathcal{COM}_R):f_*
    \]
    \item In particular, if $A$ is an $\mathcal{E}$-algebra with the cofibrant replacement $QA\rightarrow A$, then both maps in the resulting zig-zag $A\leftarrow QA\rightarrow f_*f^*(QA)=f^*(QA)$ of $\mathcal{E}$-algebras are quasi-isomorphisms.
\end{enumerate}
\end{theorem}

Thus, if $R$ is a field of characteristic zero and $\mathcal{E}$ is an $E_{\infty}$-operad over $R$, then there is a functor $W$ from $\mathcal{E}$-algebras to CDGA's over $R$ such that, for every $\mathcal{E}$-algebra $A$, $W(A)$ is functorially quasi-isomorphic to $A$ as $\mathcal{E}$-algebras by a zig-zag of morphisms.

\begin{definition}\label{Def: rectification of an E-infinity algebra}
Retain the notation of Theorem \ref{Fact: rectification}.
Assume that $R$ is a field of characteristic zero, and let $A$ be an $\mathcal{E}$-algebra. The CDGA $f^*(QA)$ constructed in Theorem \ref{Fact: rectification} is called the \textbf{rectification} of $A$.   
\end{definition}

\cite{Hinich-Model-Structure-on-homotopy-algbras}*{Theorem 6.1.1} shows that, for any commutative ring $R$, the category of operads over $R$ carries a closed model structure, whose weak equivalences are componentwise quasi-isomorphisms and whose fibrations are componentwise surjections.

Let $s\textbf{Set}$ denote the category of simplicial sets. The following result is essentially shown in \cite{Hinich-Schechtman}*{Theorem 2.3}\cite{May-operads-sheaf-cohomology}*{Theorem 3.9}\cite{May-operads-sheaf-cohomology}*{Proposition 3.3}\cite{Mandell-cochain-multiplications}\cite{Hinich-Model-Structure-on-homotopy-algbras}\cite{Berger-Moerdijk-axiomatic-homotopy-theory-operads}\cite{McClure-Smith}.

\begin{theorem}\label{Thm: simplicial cochain complex is a functorial E-infinity algebra}
For any commutative ring $R$,
there exists a cofibrant $E_{\infty}$-operad $\mathcal{E}$ such that the simplicial cochain complexes give a functor $S^*(-;R):s\textbf{Set}\rightarrow \mathbf{Alg}(\mathcal{E})$.
\end{theorem}

\begin{proof}
Let $\mathcal{Z}$ be the Eilenberg-Zilber operad over $R$ (\cite{May-operads-sheaf-cohomology}*{Definition 3.1}). Then the augmentation map $\mathcal{Z}\rightarrow \mathcal{COM}_R$ is a quasi-isomorphism. There exists a cofibrant $E_{\infty}$-operad $\mathcal{E}$ together with a quasi-isomorphism $\mathcal{E}\rightarrow \mathcal{Z}$ (\cite{May-operads-sheaf-cohomology}*{Proposition 3.3}). Since $S^*(-;R)$ is a functor $s\textbf{Set}\rightarrow \mathbf{Alg}(\mathcal{Z})$ (\cite{May-operads-sheaf-cohomology}*{Proposition 3.2}), pulling the algebra structure along $\mathcal{E}\rightarrow \mathcal{Z}$ proves the theorem.
\end{proof}

The following folklore results are used in Section \ref{SubSec: Galois representation}.

Let $k$ be a field of characteristic $0$. Let $c\mathbf{Alg}_k$, $c\mathbf{Mod}_k$, $\mathbf{Ch}^{\geq 0}(k)$ denote the categories of cosimplicial commutative $k$-algebras, cosimplicial $k$-modules, and non-negatively graded cochain complexes of $k$-modules, respectively. These categories carry model structures with the following common description: a morphism $f$ is a weak equivalence if it induces an isomorphism on cohomology, and it is a fibration if it is levelwise surjective (\cite{Toen-affine-stacks}*{Theorem 2.1.1}).

Let $N:c\mathbf{Mod}_k\rightarrow \mathbf{Ch}^{\geq 0}(k)$ denote the normalized cochain complex functor. Its left adjoint, denoted $DK$, is the Dold-Kan functor, and $N$ and $DK$ are mutually quasi-inverse (\cite{stacks-project}*{Tag 019I}).  Let $\mathcal{Z}_N$ be the Eilenberg-Zilber operad associated with $N$ (\cite{Richter-cosimplicial-algebras-and-normalization-functor}*{Proposition 6.4.1}). By \cite{Richter-cosimplicial-algebras-and-normalization-functor}*{Theorem 6.4.3}, $\epsilon: \mathcal{Z}_N\rightarrow \mathcal{COM}_k$ is a quasi-isomorphism.

Let $\mathbf{Alg}^{\geq 0}(\mathcal{Z}_N)$ denote the model category of non-negatively graded $\mathcal{Z}_N$-algebras. A morphism $f$ is a weak equivalence if it induces an isomorphism on cohomology, and it is a fibration if it is degree-wise surjective (\cite{Richter-cosimplicial-algebras-and-normalization-functor}*{Proposition 8.3.4}).

\begin{lemma}\label{Lem: Quillen equivalence of normalized cochain complex functor into Eilenberg-Zilber algebras}
The normalized cochain complex functor $N_{\mathcal{Z}}:c\mathbf{Alg}_k\rightarrow \mathbf{Alg}^{\geq 0}(\mathcal{Z}_N)$ is the right adjoint of a Quillen equivalence.
\end{lemma}

\begin{proof}
Following the idea of \cite{Richter-cosimplicial-algebras-and-normalization-functor}*{Theorem 6.5.1}, we construct a left adjoint $L_{\mathcal{Z}}$ of $N_{\mathcal{Z}}$. Let $F_{\mathcal{Z}}:\mathbf{Ch}^{\geq 0}(k)\rightarrow \mathbf{Alg}^{\geq 0}(\mathcal{Z}_N)$ be the free $\mathcal{Z}_N$-algebra functor, and let $\mathrm{Sym}_c:c\mathbf{Mod}_k\rightarrow c\mathbf{Alg}_k$ be the free cosimplicial commutative algebra functor. Let $U$ denote the relevant forgetful functors. We have the adjoint functors $\mathrm{Sym}_c\dashv U$, $DK\dashv N$ and $F_{\mathcal{Z}}\dashv U$ between the indicated categories. Consequently, for a free $\mathcal{Z}_N$-algebra $F_{\mathcal{Z}}(V)$, with $V\in \mathbf{Ch}^{\geq 0}(k)$, we set $L_{\mathcal{Z}}(F_{\mathcal{Z}}(V)):=\mathrm{Sym}_c(DK(V))$. This gives the required value of the left adjoint on free algebras. 

To define $L_{\mathcal{Z}}$ on an arbitrary $\mathcal{Z}_N$-algebra $B$, let $P_{\mathcal{Z}}=U\circ F_{\mathcal{Z}}$ denote the free $\mathcal{Z}_N$-algebra monad on $\mathbf{Ch}^{\geq 0}(k)$. The  structure map $a:P_{\mathcal{Z}}(UB)\rightarrow UB$ exhibits $B$ as the coequalizer of $F_{\mathcal{Z}}P_{\mathcal{Z}}(UB)\rightrightarrows F_{\mathcal{Z}}(UB)$, where the two arrows are induced by $P_{\mathcal{Z}}(a)$ and the monad multiplication $\mu:P_{\mathcal{Z}}P_{\mathcal{Z}}(UB)\rightarrow P_{\mathcal{Z}}(UB)$. Define $L_{\mathcal{Z}}(B)$ to be the coequalizer of $\mathrm{Sym}_c(DK(P_{\mathcal{Z}}(UB)))\rightrightarrows \mathrm{Sym}_c(DK(UB))$. With this definition, $L_{\mathcal{Z}}$ is left adjoint to $N_{\mathcal{Z}}$.

By construction, $N_{\mathcal{Z}}$ preserves fibrations and weak equivalences. Thus, the adjunction $(L_{\mathcal{Z}},N_{\mathcal{Z}})$ is a Quillen adjunction.

Define a monad $T$ on $\mathbf{Ch}^{\geq 0}(k)$ by $T(V):=N\circ U\circ \mathrm{Sym}_c\circ DK(V)$. Since $N$ and $DK$ are mutually quasi-inverse, $c\mathbf{Alg}_k$ is categorically equivalent to the category $\mathbf{Alg}^{\geq 0}(T)$ of non-negatively graded $T$-algebras. The canonical inclusion of generators $DK(V)\rightarrow U(\mathrm{Sym}_c(DK(V)))$ induces a natural map $V=N(DK(V))\rightarrow T(V)$. By the universal property of the free $\mathcal{Z}_{N}$-algebra, this map extends uniquely to a $\mathcal{\ZZ}_N$-algebra map $F_{\mathcal{Z}}(V)\rightarrow N_{\mathcal{Z}}(\mathrm{Sym}_c(DK(V)))$. Thus, we obtain a natural monad map $\alpha:P_{\mathcal{Z}}\rightarrow T$. It remains to prove that the induced Quillen adjunction $\alpha_{!}:\mathbf{Alg}^{\geq 0}(P_{\mathcal{Z}})\rightleftarrows \mathbf{Alg}^{\geq 0}(T):\alpha^*$ is a Quillen equivalence.

The Alexander-Whitney map $AW_r:(N(DK(V)))^{\otimes r}\rightarrow N((DK(V))^{\otimes r})$ is a quasi-isomorphism. Since $H^*(\mathcal{Z}_N(r))\cong k$ for every $r$, the K\"unneth formula shows that the map $\widetilde{\alpha}_r:\mathcal{Z}_N(r)\otimes V^{\otimes r}\rightarrow N((DK(V))^{\otimes r})$ induced by $\alpha$ is a quasi-isomorphism. Since $k$ has characteristic zero, the trivial $k[\Sigma_r]$-module $k$ is projective. Consequently, passing to $\Sigma_r$-coinvariants gives a quasi-isomorphism $\alpha_r:(\mathcal{Z}_N(r)\otimes V^{\otimes r})_{\Sigma_r}\rightarrow N((DK(V))^{\otimes r})_{\Sigma_r}$. Thus, $\alpha:P_{\mathcal{Z}}(V)\rightarrow T(V)$ is a quasi-isomorphism, since $P_{\mathcal{Z}}(V)=\bigoplus_{r\geq 0}(\mathcal{Z}_N(r)\otimes V^{\otimes r})_{\Sigma_r}$  and $T(V)=\bigoplus_{r\geq 0} N((DK(V))^{\otimes r})_{\Sigma_r}$. Then the induced simplicial map $B_{\bullet}(P_{\mathcal{Z}},P_{\mathcal{Z}},V)\rightarrow B_{\bullet}(T,P_{\mathcal{Z}},V)$ between two-sided bar-constructions is a levelwise weak-equivalence. Thus $\hocolim B_{\bullet}(P_{\mathcal{Z}},P_{\mathcal{Z}},V)\rightarrow \hocolim B_{\bullet}(T,P_{\mathcal{Z}},V)$ is a weak equivalence of cochain complexes. 

Let $A$ be a cofibrant $P_{\mathcal{Z}}$-algebra. There are weak equivalences $\mathbf{L}\alpha_{!}(A)\xleftarrow{\sim} \mathbf{L}\alpha_{!}(A)(\hocolim B_{\bullet}(P_{\mathcal{Z}},P_{\mathcal{Z}},A))\xrightarrow{\sim} \hocolim \mathbf{L}\alpha_{!}B_{\bullet}(P_{\mathcal{Z}},P_{\mathcal{Z}},A)\simeq \hocolim \alpha_{!}B_{\bullet}(P_{\mathcal{Z}},P_{\mathcal{Z}},A)=\hocolim B_{\bullet}(T,P_{\mathcal{Z}},A)$.
Moreover, $A$ is weak equivalent to $\hocolim B_{\bullet}(P_{\mathcal{Z}},P_{\mathcal{Z}},A)$ and every object in $\mathbf{Alg}^{\geq 0}(T)$ is fibrant. This shows that $A\rightarrow \mathbf{R}\alpha^*\mathbf{L}\alpha_{!}(A)$ is a weak equivalence. By \cite{Hovey-model-categories}*{Corollary 1.3.16}, $(\alpha_{!},\alpha^*)$ is therefore a Quillen equivalence, which completes the proof.
\end{proof}

By \cite{Richter-cosimplicial-algebras-and-normalization-functor}*{Theorem 8.3.3}, we may choose a cofibrant replacement $\beta:\mathcal{E}^{\geq 0}\rightarrow \mathcal{Z}_N$ in the model category of reduced operads in non-negatively graded cochain complexes. Applying the argument from the final two paragraphs of the proof of Lemma \ref{Lem: Quillen equivalence of normalized cochain complex functor into Eilenberg-Zilber algebras}, $\beta_{!}:\mathbf{Alg}^{\geq 0}(\mathcal{E})\rightleftarrows \mathbf{Alg}^{\geq 0}(\mathcal{Z}_N):\beta^*$ is a Quillen equivalence. We therefore obtain the following corollary.

\begin{corollary}
The normalized cochain complex $N_{\mathcal{E}}:c\mathbf{Alg}_k\rightarrow \mathbf{Alg}^{\geq 0}(\mathcal{E})$ is the right adjoint of a Quillen equivalence.
\end{corollary}

\begin{proposition}\label{Prop: embedding of non-negatively graded E_infinity algebras to unbounded E_infinity algebras}
The inclusion functor $\mathbf{Ch}^{\geq 0}(k)\rightarrow \mathbf{Ch}(k)$ induces a fully faithful functor on homotopy categories $\mathrm{Ho}(\mathbf{Alg}^{\geq 0}(\mathcal{E}))\rightarrow \mathrm{Ho}(\mathbf{Alg}(\mathcal{E}))$.
\end{proposition}

\begin{proof}
Since every object in $\mathbf{Alg}^{\geq 0}(\mathcal{E})$ and $\mathbf{Alg}(\mathcal{E})$ is fibrant, it suffices to show that the inclusion functor $i:\mathbf{Alg}^{\geq 0}(\mathcal{E})\rightarrow \mathbf{Alg}(\mathcal{E})$ preserves cofibrant objects and it preserves cylinder objects. As described in \cite{Richter-cosimplicial-algebras-and-normalization-functor}*{p.~303}, the model category $\mathbf{Alg}^{\geq 0}(\mathcal{E})$ is cofibrantly generated by $F^{\geq 0}_{\mathcal{E}}(S^{n-1})\rightarrow F^{\geq 0}_{\mathcal{E}}(D^n)$, where $S^{n-1}$ is the cochain complex consisting of $k$ in degree $n-1$, $D^{n}$ is the cochain complex consisting of $k$ concentrated in degrees $n-1$ and $n$ with the identity differential, and $F^{\geq 0}_{\mathcal{E}}$ is the free $\mathcal{E}$-algebra functor. Since $\mathcal{E}$ is non-negatively graded, $F_{\mathcal{E}}(S^{n-1})\rightarrow F_{\mathcal{E}}(D^n)$ is a cofibration in $\mathbf{Alg}(\mathcal{E})$. Since every cofibrant object $A$ of $\mathbf{Alg}^{\geq 0}(\mathcal{E})$ is a retract of a transfinite composition of pushouts of generating cofibrations, $i(A)$ is also a retract of a transfinite pushouts of cofibrations of $\mathbf{Alg}(\mathcal{E})$. Thus $i(A)$ is cofibrant in $\mathbf{Alg}(\mathcal{E})$.

For any object $B\in \mathbf{Alg}^{\geq 0}(\mathcal{E})$, the natural maps $B\rightarrow B\otimes (t,dt)\xrightarrow{t=0,t=1}B\times B$ give a cylinder object for $B$ in $\mathbf{Alg}^{\geq 0}(\mathcal{E})$. Applying the functor $i$ to this composition of maps gives rise to a cylinder object of $i(B)$ in $\mathbf{Alg}(\mathcal{E})$. This completes the proof.
\end{proof}

\section{Continuous Cohomology of Pro-$p$ Groups}\label{Section: continuous cohomology}\,

In this appendix, we collect several results on the continuous cohomology of pro-$p$ groups from \cite{Serre-Galois-cohomology}\cite{Tate-continuous-group-cohomology}, with particular emphasis on degree-$2$ cohomology. 

Let $G$ be a topological group, and let $A$ be a topological abelian group. Define the \textbf{continuous cochain complex} $C^*_{\cont}(G;A)$ as follows. Each $C^n_{\cont}(G,A)$ is the abelian group of continuous maps $G^{\times n}\rightarrow A$, with differentials defined by the ordinary group-cohomology formula. The cohomology group $H^n_{\cont}(G;A):=H^n(C^*_{\cont}(G,A))$ is called the \textbf{continuous cohomology} of $G$ with coefficients in $A$.

\begin{proposition}[\cite{Tate-continuous-group-cohomology}*{p.~260, the paragraph above Corollary; Proposition 2.2}\cite{Serre-Galois-cohomology}*{p.~11}]\label{Fact: Milnor exact sequence for continuous cohomology}
Let $G$ be a topological group, and let $A$ be a finitely generated $\widehat{\ZZ}_p$-module. Then, for every $n$,
\begin{enumerate}[leftmargin=0.25in]
    \item $C^n_{\cont}(G;A)=\varprojlim_i C^n_{\cont}(G;A/p^iA)$;
    \item there is a natural short exact sequence
    \[
    0\rightarrow \varprojlim_i{}^1H^{n-1}_{\cont}(G;A/p^iA)\rightarrow H_{\cont}^{n}(G;A)\rightarrow \varprojlim_i H^{n}_{\cont}(G;A/p^iA) \rightarrow 0
    \]
\end{enumerate}
\end{proposition}

\begin{proposition}[\cite{Serre-Galois-cohomology}*{p.~11, Proposition 8}]\label{Fact: relation between continuous cohomology and usual cohomology for the system}
Let $G$ be a profinite group, and let $\{U\}$ denote the pro system of its open normal subgroups. Then, for every $q$ and every finite abelian group $A$, the canonical map $\varinjlim_U H^q(G/U;A)\rightarrow H^q_{\cont}(G;A)$ is an isomorphism.
\end{proposition} 

Recall from \cite{Serre-Galois-cohomology}\cite{Artin-Mazur-etale-homotopy}*{Definition 6.5} that an abstract group $G$ is \textbf{Serre $p$-good} if, for every $G$-module $M$ whose underlying abelian group is a finite $p$-group, the canonical map
$\varinjlim_U H^q(G/U;M)\rightarrow H^q(G;M)$ is an isomorphism, where $U$ ranges over the normal subgroups of $G$ of $p$-power index for which the quotient action on $M$ is defined.

\begin{proposition}\label{Prop: comparison between continuous and discrete cohomology for Serre good groups}
Let $G$ be a profinite group, and let $A$ be a finitely generated $\widehat{\ZZ}_p$-module. If the underlying abstract group of $G$ is Serre $p$-good, then the canonical map $H^q_{\cont}(G;A)\rightarrow H^q(G;A)$ is an isomorphism.
\end{proposition}

\begin{proof}
By Proposition \ref{Fact: relation between continuous cohomology and usual cohomology for the system} and Serre $p$-goodness, the canonical map $H^q_{\cont}(G;A/p^jA)\rightarrow H^q(G;A/p^jA)$ is an isomorphism for every $q$ and $j$. The proposition now follows by comparing the two natural short exact sequences in the following commutative diagram and applying the five lemma.
\[
\begin{tikzcd}
0 \arrow[r]  & \varprojlim_i{}^1H^{n-1}_{\cont}(G;A/p^iA) \arrow[r] \arrow[d,"\sim"] & H_{\cont}^{n}(G;A) \arrow[r] \arrow[d] & \varprojlim_i H^{n}_{\cont}(G;A/p^iA) \arrow[r] \arrow[d,"\sim"] & 0    \\
0 \arrow[r]  & \varprojlim_i{}^1H^{n-1}(G;A/p^iA) \arrow[r]  & H^{n}(G;A) \arrow[r] & \varprojlim_i H^{n}(G;A/p^iA) \arrow[r] & 0 
\end{tikzcd}
\]
\end{proof}

\begin{lemma}\label{Lem: relation between continuous and discrete cohomology for nilpotent pro-p groups}
Every topologically finitely generated nilpotent pro-$p$ group is Serre $p$-good.
\end{lemma}

\begin{proof}
This follows from \cite{MFernandez-Alcober-Kazachkov-Remeslennikov-Symonds-comparison-discrete-continuous-cohomology-pro-p-groups}*{Theorem 2.10}, because such a pro-$p$ group $G$ has a finite chain of closed normal subgroups $G=P_1\supset ... \supset P_s=0$ such that each quotient $P_i/P_{i+1}$ is either isomorphic to $\widehat{\ZZ}_p$ or is finite.
\end{proof}

As in ordinary group cohomology,
the second continuous cohomology group classifies continuous central extensions. Let $G$ be a profinite group, and let $A$ be an abelian profinite group. A \textbf{continuous central extension} of $G$ by $A$ is a short exact sequence of profinite groups $1\rightarrow A\rightarrow H\rightarrow G\rightarrow 1$ in which all maps are continuous homomorphisms and the image of $A$ lies in the center of $H$. Two such extensions $H$ and $H'$ are \textbf{equivalent} if there is a continuous isomorphism $H\rightarrow H'$ such that the following diagram commutes.
\[
\begin{tikzcd}
 1 \arrow[r] & A \arrow[r] \arrow[d,equal] & H \arrow[r] \arrow[d] & G \arrow[r] \arrow[d,equal] & 1 \\
1 \arrow[r] & A \arrow[r] & H' \arrow[r] & G \arrow[r] & 1   
\end{tikzcd}
\]

Every continuous central extension $H$ admits a continuous section $s:G\rightarrow H$ of the projection $H\twoheadrightarrow G$ (\cite{Ribes-Zallesskii-profinite-groups}*{Proposition 2.2.2}). Such a section determines a continuous map $f:G\times G\rightarrow A$ given by $f(x,y)=s(x)s(y)s(xy)^{-1}$. If $s':G\rightarrow H$ is another section and $f':G\times G\rightarrow A$ is its associated map, then $s$ and $s'$ determine a continuous map $\phi:G\rightarrow A$ given by $\phi(x)=s'(x)s(x)^{-1}$.

\begin{lemma}\label{Lem: Correspondence between central extension and second cohomology}
Let $G$ be a profinite group, and let $A$ be an abelian profinite group, with the notation introduced above.
\begin{enumerate}[leftmargin=0.25in]
    \item The continuous section $s:G\rightarrow H$ induces a homeomorphism of profinite sets $\iota\cdot s: A\times G\rightarrow H$ given by $(\iota\cdot s)(a,g)=a\cdot s(g)$.
    \item The map $f$ is a continuous $2$-cocycle.
    \item The associated continuous cochains satisfy
    $d(\phi)=f'-f$.
    \item Consequently, the assignment sending an extension to the cohomology class $[f]$ defines a bijection between equivalence classes of continuous central extensions of $G$ by $A$ and $H^2_{\cont}(G;A)$.
\end{enumerate}
\end{lemma}

\begin{proof}
Items (2) and (3) follow from the standard computations in ordinary group cohomology, since all the cochains involved are continuous.

\textbf{Item (1). } Since $\iota\cdot s$ factors as $A\times G\xrightarrow{\iota\times s}H\times H\xrightarrow{\nu} H$, $\iota\cdot s$ is continuous. It is also bijective. Since $A\times G$ is compact and $H$ is Hausdorff, it is a homeomorphism.

\textbf{Item (4). } To prove injectivity, let $H$ and $H'$ be two continuous central extensions, and choose continuous sections $s:G\rightarrow H$ and $s':G\rightarrow H'$. Let $f$ and $f'$ be the continuous $2$-cocycles associated with $s$ and $s'$, respectively. Assume that $[f]=[f']\in H^2_{\cont}(G;A)$. Using item (3), we may modify $s'$ by a continuous $1$-cochain so that its associated cocycle is exactly $f$. Thus, we may assume that $f=f'$. Define $\psi:H\rightarrow H'$ by $\psi(a\cdot s(g))=a\cdot s'(g)$ for all $a\in A$ and $g\in G$. Since $f=f'$, a direct computation shows that $\psi$ is a group homomorphism. By item (1), the identifications $A\times G\cong H$ and $A\times G\cong H'$ show that $\psi$ is a homeomorphism. Thus the two extensions are equivalent, proving injectivity.

To prove surjectivity, let $f\in C^2_{\cont}(G;A)$ be a continuous $2$-cocycle. Define a multiplication on the profinite set $A\times G$ by $(a,g)\cdot (a',g')=(a+a'+f(g,g'),gg')$. Define inversion by $(a,g)^{-1}=(-a-f(g,g^{-1}),g^{-1})$. The inverse formula is clear and the cocycle identity implies associativity. Since all structure maps are continuous, $A\times G$ becomes a profinite group fitting into a continuous central extension of $G$ by $A$. Its associated cohomology class is $[f]$, proving surjectivity.
\end{proof}

\begin{lemma}\label{Lem: injectivity of colimit of lower central series into group}
\begin{enumerate}[leftmargin=0.25in]
    \item Let $G$ be a topologically finitely generated profinite group, and let $\Gamma^*G$ denote its lower central series. Let $A$ be an abelian profinite group. Then the natural map $\varinjlim_r H^2_{\cont}(G/\Gamma^rG;A)\rightarrow H^2_{\cont}(G;A)$ is injective. 
    \item The analogous conclusion of ordinary group cohomology holds for an abstract group $G$ and an abstract abelian group $A$.
\end{enumerate}
\end{lemma} 

\begin{proof}
We prove only item (1), since the same argument proves item (2).

Let $G_i=G/\Gamma^{i}G$. 
By Lemma \ref{Lem: Correspondence between central extension and second cohomology}, 
it suffices to prove the following assertion. Let $1\rightarrow A\rightarrow E\xrightarrow{g} G_{i}\rightarrow 1$ be a continuous central extension. If its pullback along $\pi_i:G\twoheadrightarrow G_{i}$ splits, then its pullback along $G_{i+1}\twoheadrightarrow G_{i}$ also splits. Since the pullback along $G\twoheadrightarrow G_{i}$ splits, there is a continuous homomorphism $s:G\rightarrow E$ such that $g\circ s=\pi_i$. Since $G_i$ has nilpotency class at most $i-1$, $E$ has nilpotency class at most $i$. Then $s$ vanishes on $\Gamma^{i+1}G$ and therefore factors through a homomorphism $G_{i+1}\rightarrow E$. This gives a splitting after pullback to $G_{i+1}$ and completes the proof.
\end{proof}

\begin{lemma}\label{Lem: exactness for inverse limit}
Let $\{0\rightarrow A_i\xrightarrow{f_i} B_i\xrightarrow{g_i} C_i\xrightarrow{h_i} D_i\xrightarrow{k_i} E_i\}_i$ be a diagram of exact sequences of abelian groups. If $\varprojlim^1_i A_i=\varprojlim^2_i A_i=\varprojlim^1_i B_i=0$, then $0\rightarrow \varprojlim_iA_i\rightarrow \varprojlim_iB_i\rightarrow \varprojlim_iC_i\rightarrow \varprojlim_iD_i\rightarrow \varprojlim_iE_i$ is exact.
\end{lemma}

\begin{proof}
Set $X_i=\Im(g_i)=\Ker(h_i)$ and $Y_i=\Im(h_i)=\Ker(k_i)$. Applying the left-exact inverse-limit functor to $\{0\rightarrow A_i\rightarrow B_i\rightarrow X_i\rightarrow 0\}_i$ gives exactness at $\varprojlim_iA_i$ and $\varprojlim_iB_i$.

Since $\varprojlim^1_i A_i=0$, the induced sequence $0\rightarrow \varprojlim_iA_i\rightarrow \varprojlim_iB_i\rightarrow \varprojlim_iX_i\rightarrow 0$ is exact. The map $\varprojlim_iB_i\rightarrow \varprojlim_iC_i$ factors through the inclusion $\varprojlim_iX_i\rightarrow \varprojlim_iC_i$. Since $\varprojlim_iB_i\rightarrow \varprojlim_iX_i$ is surjective, its image in $\varprojlim_iC_i$ is $\varprojlim_iX_i$. Left exactness identifies this group with $\Ker(\varprojlim_iC_i\rightarrow \varprojlim_iD_i)$. This shows the exactness at $\varprojlim_iC_i$.

The long exact sequence of derived inverse limits associated with $\{0\rightarrow A_i\rightarrow B_i\rightarrow X_i\rightarrow 0\}_i$ shows that $\varprojlim_i^1 X_i=0$, because $\varprojlim_i^1 B_i=\varprojlim_i^2 A_i=0$. Applying inverse limits to $\{0\rightarrow X_i\rightarrow C_i\rightarrow Y_i\rightarrow 0\}_i$ therefore shows that $\varprojlim_iC_i\rightarrow \varprojlim_iY_i$ is surjective. Left exactness identifies $\varprojlim_iY_i$ with $\Ker(\varprojlim_iD_i\rightarrow \varprojlim_iE_i)$. This proves the exactness at $\varprojlim_iD_i$.
\end{proof}

\begin{proposition}\label{Prop: injectivity of colimit of lower central series into group}
Let $G$ be a topologically finitely generated pro-$p$ group, and write $\Gamma^*G$ for its lower central series. For every $i\geq 2$, there is a natural isomorphism $\Ker(H^2_{\cont}(G/\Gamma^{i}G;\widehat{\ZZ}_p) \rightarrow H^2_{\cont}(G;\widehat{\ZZ}_p))\cong\homo_{\widehat{\ZZ}_p}(\Gamma^{i} G/\Gamma^{i+1} G,\widehat{\ZZ}_p)$. 
\end{proposition} 

\begin{proof}
For each $i\geq 1$,
let $G_i=G/\Gamma^{i+1}G$. Apply the Lyndon–Hochschild–Serre spectral sequence (\cite{Serre-Galois-cohomology}*{p.~15}) to the exact sequence $1\rightarrow \Gamma^i G\rightarrow G\rightarrow G_{i-1}\rightarrow 1$, with trivial coefficients in $\ZZ/p^j$. This yields the following exact sequence:
\[
0\rightarrow H^1_{\cont}(G_{i-1};\ZZ/p^j) \xrightarrow{g_{i-1}} H^1_{\cont}(G;\ZZ/p^j) \rightarrow H^1_{\cont}(\Gamma^i G;\ZZ/p^j)^{G_{i-1}}\rightarrow H^2_{\cont}(G_{i-1};\ZZ/p^j) \xrightarrow{b_{i-1}} H^2_{\cont}(G;\ZZ/p^j)
\]
Since the fixed-point functor is a limit, it commutes with inverse limits. Thus, $\varprojlim_j H^1_{\cont}(\Gamma^i G;\ZZ/p^j)^{G_{i-1}}=(\varprojlim_j H^1_{\cont}(\Gamma^i G;\ZZ/p^j))^{G_{i-1}}$. Since $G$ and $G_{i-1}$ are topologically finitely generated, their first continuous cohomology groups with coefficients in $\ZZ/p^j$ are finite. Proposition \ref{Fact: Milnor exact sequence for continuous cohomology} and Lemma \ref{Lem: exactness for inverse limit} then give the following exact sequence.
\[
0\rightarrow H^1_{\cont}(G_{i-1};\widehat{\ZZ}_p) \xrightarrow{g_{i-1}} H^1_{\cont}(G;\widehat{\ZZ}_p) \rightarrow H^1_{\cont}(\Gamma^i G;\widehat{\ZZ}_p) ^{G_{i-1}}\rightarrow H^2_{\cont}(G_{i-1};\widehat{\ZZ}_p) \xrightarrow{b_{i-1}} H^2_{\cont}(G;\widehat{\ZZ}_p)
\]
Because $\widehat{\ZZ}_p$ is abelian, every continuous homomorphism $G\rightarrow \widehat{\ZZ}_p$ vanishes on $\Gamma^2G$. Then it also vanishes on $\Gamma^i G$ for every $i\geq 2$. Hence, every such homomorphism factors uniquely through $G_{i-1}=G/\Gamma^iG$, and the map $g_{i-1}$ is an isomorphism.
Exactness therefore identifies $\Ker(b_{i-1})$ with $H^1_{\cont}(\Gamma^i G;\widehat{\ZZ}_p)^{G_{i-1}}\cong \homo_{\cont}(\Gamma^i G,\widehat{\ZZ}_p)^{G_{i-1}}$. A continuous homomorphism $\lambda:\Gamma^i G\rightarrow \widehat{\ZZ}_p$ is $G_{i-1}$-invariant precisely when $\lambda(gxg^{-1})=\lambda(x)$ for all $g\in G$ and $x\in \Gamma^{i}G$. Equivalently, $\lambda$ vanishes on  $(G,\Gamma^iG)=\Gamma^{i+1}G$. Thus $\lambda$ factors uniquely through the quotient $\Gamma^iG/\Gamma^{i+1}G$. Consequently, $\Ker(b_{i-1})\cong \homo_{\cont}(\Gamma^i G/\Gamma^{i+1} G,\widehat{\ZZ}_p)$. Since the quotient is a finitely generated $\widehat{\ZZ}_p$-module, this is the same as $\homo_{\widehat{\ZZ}_p}(\Gamma^i G/\Gamma^{i+1} G,\widehat{\ZZ}_p)$.
\end{proof}

\section{Purity Implies Formality}

In this appendix, we prove the folklore statement that ``purity implies formality'' (\cite{Emprin-Horel-weight-and-formality}*{Section 3}) for minimal CDGA's.

\begin{definition}\label{Def: weighted CDGA and weighted homotopy}
Let $k$ be a field of characteristic $0$.
\begin{enumerate}[leftmargin=0.25in]
    \item A CDGA $C^*$ over $k$ is \textbf{weighted} if each $C^i$ admits a decomposition $\bigoplus_{w\in \ZZ} C^{i,w}$ by weights, and both the differential and multiplication preserve the \textbf{weight} grading.
    \item A \textbf{weighted homotopy} between two weighted CDGA morphisms $f,g:C\rightarrow D$ is a weighted CDGA morphism $H:C\rightarrow D\otimes (t,dt)$ such that $H\vert_{t=0}=f$ and $H\vert_{t=1}=g$, where both $t$ and $dt$ have weight zero.
    \item The cohomology of a connected weighted CDGA $C$ is said to have \textbf{pure weights} if $H^n(C)^w=0$ whenever $w\neq n$.
\end{enumerate}
\end{definition}

\begin{proposition}\label{prop: existence of weight on a minimal model}
Let $\phi:N\rightarrow C^*$ be a minimal model of a connected weighted CDGA $C^*$. Then $N$ admits a weight grading for which both its differential and multiplication preserve weights and for which $\phi$ is a weighted CDGA morphism.    
\end{proposition}

\begin{proof}
We inductively construct a tower of weighted minimal CDGA's: $k=M(1,0)\subset M(1,1)\subset ...\subset M(2,0)=\bigcup_{i}M(1,i)\subset M(2,1)\subset ...$, together with compatible weighted morphisms $M(n,s)\rightarrow C$, whose union exhibits $M=\bigcup_{n,i}M(n,i)$ to be a minimal model of $C^*$. Theorem \ref{Fact: Lifting of minimal models} then identifies the union of this tower with the given minimal model $N$, up to an isomorphism. Transporting the weight grading across this isomorphism proves the proposition. 

Choose a basis of $H^1(C)$ consisting of weight-homogeneous classes, and represent each basis element by a weight-homogeneous cocycle in $C^1$. Let $V(1,1)$ be an abstract weighted vector space with basis corresponding to these cocycles, and let $M(1,1)=\bigwedge^* V(1,1)$, equipped with the zero differential. Then the weight grading on $V(1,1)$ extends multiplicatively to $M(1,1)$. Sending each generator to its chosen cocycle defines a weighted CDGA morphism $\phi_{1,1}:M(1,1)\rightarrow C^*$. By construction, $\phi_{1,1}$ induces isomorphisms on $H^0$ and $H^1$.

Suppose inductively that a weighted minimal CDGA $M(n,s)$, together with a weighted morphism $\phi_{n,s}:M(n,s)\rightarrow C^*$, has been constructed. Assume that $\phi_{n,s}$ induces an isomorphism on $H^q$ for every $q\leq n$ when $s\geq 1$, and that $\phi_{n,0}$ induces an isomorphism on $H^q$ for every $q< n$ and an injection on $H^{n}$ when $s=0$.
The passage from $M(n,0)$ to $M(n,1)$ is obtained in the same way as the construction of $M(1,1)$. For $s\geq 1$, choose a basis of $\Ker(\phi^*_{n,s}:H^{n+1}(M(n,s))\rightarrow H^{n+1}(C^*))$ consisting of weight-homogeneous classes $\{z_1,...,z_r\}$ and represent each class $z_i$ by a weight-homogeneous cocycle $v_i\in M(n,s)^{n+1}$. Let $V(n,s+1)$ be a weighted vector space spanned by $\{v_1,...,v_r\}$. Let $M(n,s+1)=M(n,s)\otimes \bigwedge^*V(n,s+1)$ be an elementary extension of $M(n,s)$ by setting $dv_i=z_i$. Since $v_i$ and $z_i$ have the same weight, the differential preserves weights. Because the cohomology class of each $\phi_{n,s}(z_i)$ vanishes in $H^{n+1}(C)$, choose a weight-homogeneous element $c_i\in C^{n}$ such that $dc_i=\phi_{n,s}(z_i)$. Extend $\phi_{n,s}$ by sending $v_i$ to $c_i$. This defines a weighted CDGA morphism $\phi_{n,s+1}:M(n,s+1)\rightarrow C$. This completes the proof.
\end{proof}

\begin{lemma}\label{Lem: weighted liftings of weighted minimal model}
Let $f:C\rightarrow D$ be a weighted quasi-isomorphism of connected weighted CDGA's, and let $\phi:M\rightarrow D$ be a weighted minimal model. Then there exists a weighted morphism $g:M\rightarrow C$, unique up to weighted homotopy, together with a weighted homotopy between $f\circ g$ and $\phi$.
\end{lemma}

\begin{proof}
As in the proof of Proposition \ref{prop: existence of weight on a minimal model}, choose an exhaustive tower of weighted elementary extensions of minimal subalgebras $k=M(1,0)\subset M(1,1)\subset ...$ and compatible weighted morphisms $\phi_{n,s}:M(n,s)\rightarrow D$ whose union is $\phi$. Write $M(n,s+1)=M(n,s)\otimes \bigwedge^* V(n,s+1)$, and choose a basis $\{x_i\}$ of $V(n,s+1)$ consisting of weight-homogeneous elements.

\textbf{Existence.} Assume inductively that a weighted morphism $g_{n,s}:M(n,s)\rightarrow C$ has been constructed, together with a weighted homotopy $H_{n,s}:M(n,s)\rightarrow C\otimes (t,dt)$ from $f\circ g_{n,s}$ to $\phi_{n,s}$. By \cite{MGriffiths-Morgan-rational-homotopy}*{Proposition 11.1}, there are weight-homogeneous elements $c_i\in C^{n}$ and $d_i\in D^{n-1}$ such that $d(c_i)=g_{n,s}(dx_i)$ and $f(c_i)-d(d_i)=\phi_{n,s+1}(x_i)+\int^{1}_0 H_{n,s}(dx_i)$ for every $i$. Define $g_{n,s+1}(x_i)=c_i$ and $H_{n,s+1}(x_i)=f(c_i)+\int^t_0 H(dx_i)+d(d_i\otimes t)$, and extend $g_{n,s+1}$ and $H_{n,s+1}$ multiplicatively to $M(n,s+1)$. By construction and \cite{MGriffiths-Morgan-rational-homotopy}*{Proposition 11.1}, $g_{n,s+1}$ is a weighted CDGA morphism and $H_{n,s+1}$ is a weighted homotopy from $f\circ g_{n,s+1}$ to $\phi_{n,s+1}$.

\textbf{Uniqueness.} Let $g',g'':M\rightarrow C$ be weighted CDGA morphisms, together with a homotopy $H:M\rightarrow D\otimes (t,dt)$ from $f\circ g'$ to $f\circ g''$. We show that $g'$ and $g''$ are weighted homotopic. Let $P$ be the kernel of the weighted morphism $\mu: (D\otimes (t,dt))\oplus C\oplus C\rightarrow D\oplus D$ given by $\mu(\gamma,c_1,c_2)=(\gamma\vert_{t=0}-f(c_1),\gamma\vert_{t=1}-f(c_2))$. Define a weighted CDGA morphism $\rho:C\otimes (t,dt)\rightarrow P$ by $\rho(\beta)=((f\otimes 1)(\beta),\beta\vert_{t=0},\beta\vert_{t=1})$. The homotopy $H$ together with $g'$ and $g''$ defines a weighted CDGA morphism $\psi:M\rightarrow P$ by $\psi(x)=(H(x),g'(x),g''(x))$. By the computation in \cite{MGriffiths-Morgan-rational-homotopy}*{p.~108, last paragraph}, $\rho$ is a quasi-isomorphism. Applying the weighted relative version of \cite{MGriffiths-Morgan-rational-homotopy}*{Lemma 11.2}, we obtain a weighted CDGA morphism $\widetilde{H}:M\rightarrow C\otimes (t,dt)$ such that $\rho\circ \widetilde{H}=\psi$ and $\widetilde{H}$ is a weighted homotopy from $g'$ to $g''$.
\end{proof}

\begin{corollary}
Let $C$ and $D$ be connected weighted CDGA's over $k$. 
\begin{enumerate}[leftmargin=0.25in]
    \item Let $f:C\rightarrow D$ be a weighted morphism. Let $\phi_C:M_C\rightarrow C$ and $\phi_D:M_D\rightarrow D$ be weighted minimal models. Then  there exists a weighted morphism $g:M_C\rightarrow M_D$, unique up to a weighted homotopy, together with a weighted homotopy between $f\circ \phi_C$ and $\phi_D\circ g$.
    \item In particular, if $\phi_1:M_1\rightarrow C$ and $\phi_2:M_2\rightarrow C$ are two weighted minimal models, then there exists a weighted CDGA isomorphism $\psi:M_1\rightarrow M_2$ and a weighted homotopy $H$ between $\phi_1$ and $\phi_2\circ \psi$.   
\end{enumerate}
\end{corollary}

\begin{proof}
To prove item (1), apply Lemma \ref{Lem: weighted liftings of weighted minimal model} to the weighted minimal model $\phi_D:M_D\rightarrow D$ and the weighted quasi-isomorphism $f\circ \phi_C$. For item (2), apply item (1) to the identity of $C$. The resulting weighted CDGA morphism $\psi:M_1\rightarrow M_2$ is a quasi-isomorphism. Theorem \ref{Fact: Lifting of minimal models} implies that $\psi$ is an isomorphism.
\end{proof}

\begin{definition}
Let $k$ be a field of characteristic $0$.
\begin{enumerate}[leftmargin=0.25in]
    \item A connected CDGA $C$ over $k$ is \textbf{$k$-formal} if its minimal model is quasi-isomorphic to its cohomology algebra $(H^*(C))$, equipped with the zero differential.
    \item Suppose that $k$ contains $\QQ_{p}$. A connected CW complex $X$ is \textbf{$k$-formal} if the CDGA $C^*(X;\widehat{\ZZ}_{p})\otimes_{\widehat{\ZZ}_{p}}k$ is $k$-formal. A pro system $\{X_i\}$ of pointed connected $p$-finite simplicial sets is \textbf{$k$-formal} if the CDGA $C^*(\{X_i\};\widehat{\ZZ}_{p})\otimes_{\widehat{\ZZ}_{p}}k$ is $k$-formal.
\end{enumerate}
\end{definition}

\begin{lemma}\label{Lem: weight and cohomological degree in a weighted minimal CDGA}
Let $M$ be a weighted minimal CDGA over $k$ whose cohomology has pure weights. If $M^{n,w}\neq 0$, then $w\geq n$. Moreover, $d(M^{n,n})=0$ for every $n$.
\end{lemma}

\begin{proof}
Choose an exhaustive tower $k=M(1,0)\subset M(1,1)\subset ...$ of weighted elementary extensions of minimal CDGA's for $M$, as in the proof of Proposition \ref{prop: existence of weight on a minimal model}. The generators of $M(1,1)$ represent classes in $H^1(M)$. Purity therefore gives them weight $1$ and they are closed. Assume inductively that the lemma holds for $M(n,s)$. If $s=0$,  every new generator $x\in M(n,1)^n$ has weight $n$ since $H^n(M)$ is of pure weight $n$. If $s\geq 1$, for every weight-homogeneous new generator $x\in M(n,s+1)^{n,w}$ with $dx\neq 0$, each weight-homogeneous monomial $z$ in $dx$ is a product of homogeneously weighted elements of $M(n,s)$. The inductive hypothesis forces that the weight of $z$ is at least $n+1$. Then $w\geq n+1$.
\end{proof}

\begin{proposition}\label{Prop: purity implies formality}
Let $C$ be a connected weighted CDGA. If $H^*(C)$ has pure weights, then $C$ is $k$-formal. 
\end{proposition}

\begin{proof}
By Proposition \ref{prop: existence of weight on a minimal model}, $C$ admits a weighted minimal model $M\rightarrow C$. Write $M\cong \bigwedge^*(V_1\oplus V_2\oplus ... )$, where each $V_i$ is bigraded both by cohomological degree and by weight and the cohomological degree of $V_i$ is concentrated at $i$. By Lemma \ref{Lem: weight and cohomological degree in a weighted minimal CDGA}, $V_i^w=0$ when $w<i$. It suffices to construct a quasi-isomorphism $M\rightarrow H^*(M)$. Define $\rho:V_n^w\rightarrow H^n(M)$ on a weight-homogeneous generator $x\in V_n^w$ by $\rho(x)=[x]$ when $w=n$ and $\rho(x)=0$ when $w>n$. Extend $\rho$ multiplicatively to $M$.

We first verify that $\rho$ is a chain map. For every $x\in V_n^w$, if $w=n$, then $dx=0$ by Lemma \ref{Lem: weight and cohomological degree in a weighted minimal CDGA}, so $\rho(dx)=0$; if $w\geq n+2$, every monomial in $dx$ has weight strictly larger than its cohomological degree, so $\rho(dx)=0$; if $w=n+1$, $\rho(dx)=[dx]=0=d\rho (x)$.

Let $0\neq [z]\in H^n(M)$. Since $H^n(M)$ is pure of weight $n$, we may choose a representative $z\in M^{n,n}$. Since $V^{w}_i=0$ for any $w<i$, every monomial occurring in $z$ is a product of generators in $\bigoplus_i V^{i,i}$. The definition of $\rho$ therefore gives $\rho(z)=[z]$. Hence $\rho$ on cohomology is the identity. Thus $\rho$ is a quasi-isomorphism.
\end{proof}

\begin{remark}\label{Rmk: formality uses a weighted CDGA morphism}
In the proof of Proposition \ref{Prop: purity implies formality}, the quasi-isomorphism $\rho:M\rightarrow H^*(M)$ is indeed a weighted CDGA morphism.
\end{remark}

\section{$p$-adic Galois Representations}\label{Appendix: p-adic Hodge theory}

In this appendix, we briefly recall several basic notions concerning Galois representations in $p$-adic Hodge theory. 

Let $K/\QQ_p$ be a finite field extension. Fontaine's \textbf{$p$-adic period rings} $B_{\dR}$ and $B_{\crys}$ (\cite{Fontaine-p-adic-hodge-I}\cite{Fontaine-p-adic-Hodge-II}) are topological $\QQ_{p}$-algebras, equipped with continuous $\Gal(\overline{K}/K)$-actions. Their fixed subrings are  $B_{\dR}^{\Gal(\overline{K}/K)}=K$ and $B_{\crys}^{\Gal(\overline{K}/K)}=K_0$, where $K_0$ is the maximal unramified subfield of $K$. A continuous finite-dimensional $\QQ_p$-representation $V$ of $\Gal(\overline{K}/K)$ is \textbf{de Rham} if $\dim_K((B_{\dR}\otimes_{\QQ_p}V)^{\Gal(\overline{K}/K)})=\dim_{\QQ_p}(V)$, and \textbf{crystalline} if $\dim_{K_0}((B_{\crys}\otimes_{\QQ_p}V)^{\Gal(\overline{K}/K)})=\dim_{\QQ_p}(V)$. Every crystalline representation is de Rham.

\begin{proposition}[\cite{Fontaine-p-adic-Hodge-III}]\label{Fact: properties of de Rahm crystalline representations}
De Rham representations are closed under taking subquotients, tensor products, and duals. The same holds for crystalline representations.
\end{proposition}

A principal class of examples for such representations arises from the \'etale cohomology of ``nice'' algebraic varieties or analytic varieties.

\begin{theorem}[\cite{Faltings-p-adic-Hodge}\cite{Scholze-p-adic-Hodge}*{Theorem 8.4}]
Let $K/\QQ_p$ be a finite field extension, and let $\mathcal{O}_{K}$ denote its ring of integers.
\begin{enumerate}[leftmargin=0.25in]
    \item If $X$ is a connected smooth proper adic space over $\mathrm{Spa}(K,\OO_{K})$, then $H^n_{\et}(X_{\overline{K}};\QQ_{p})$ is a de Rham representation of $\Gal(\overline{K}/K)$ for every $n$.
    \item If $\chi$ is a connected smooth proper scheme over $\mathcal{O}_{K}$, then $H^n_{\et}(\chi_{\overline{K}};\QQ_{p})$ is a crystalline representation of $\Gal(\overline{K}/K)$ for every $n$.
\end{enumerate}
\end{theorem}

We now prove Lemma \ref{Lem: Closure of algebraic operations of de Rham representations}.

\begin{proof}[Proof of Lemma \ref{Lem: Closure of algebraic operations of de Rham representations}]
We only prove the lemma for successively de Rham representations, since the proof for  successively crystalline representations is identical. 

Let $0\rightarrow U\rightarrow V\xrightarrow{\pi} W\rightarrow 0$ be a short exact sequence of continuous finite-dimensional $\QQ_p$-representations of $\Gal(\overline{K}/K)$. Suppose first that $V$ is successively de Rham, and choose a $\Gal(\overline{K}/K)$-stable filtration $0=F^0V\subset F^1V\subset ...\subset F^mV=V$ with de Rham graded pieces. The intersections $U\cap F^iV$ form a finite $\Gal(\overline{K}/K)$-stable filtration of $U$. For every $i$, the natural map $(U\cap F^iV)/(U\cap F^{i-1}V)\rightarrow F^iV/F^{i-1}V$ is injective. By Proposition \ref{Fact: properties of de Rahm crystalline representations}, every graded piece  of this filtration of $U$ is de Rham. Hence, $U$ is successively de Rham. An analogous argument for $W$ shows that it is successively de Rham. Therefore, successively de Rham representations are closed under subquotients.

Now suppose that $U$ and $W$ are successively de Rham, and choose $\Gal(\overline{K}/K)$-stable filtrations $0=F^0U\subset F^1U\subset ...\subset F^mU=U$ and $0=G^0W\subset G^1W\subset ...\subset G^nW=W$ whose graded pieces are de Rham. Then these filtrations combine to give a $\Gal(\overline{K}/K)$-stable filtration of $V$: $0=F^0U\subset F^1U\subset ...\subset F^mU=\pi^{-1}(G^0W)\subset \pi^{-1}(G^1W)\subset ...\subset \pi^{-1}(G^nW)=V$. The graded pieces arising from the first part are the graded pieces of $U$, while those arising from the second part satisfy $\pi^{-1}(G^{j}W)/\pi^{-1}(G^{j-1}W)\cong G^jW/G^{j-1}W$. All of these representations are de Rham, so $V$ is successively de Rham. This proves closure under extensions. 

Next, let $U$ and $W$ be successively de Rham representations, equipped with the $\Gal(\overline{K}/K)$-stable filtrations chosen above. Define a $\Gal(\overline{K}/K)$-stable filtration on $U\otimes_{\QQ_p}W$ by $T^r(U\otimes_{\QQ_p}W)=\sum_{i+j\leq r}F^iU\otimes_{\QQ_p} G^jW$. Since its $r$-graded piece $\Gr^r_{T}(U\otimes_{\QQ_p}W)\cong \bigoplus_{i+j=r} \Gr^i_{F}(U)\otimes_{\QQ_p}\Gr^j_{G}(W)$, $\Gr^r_{T}(U\otimes_{\QQ_p}W)$ is de Rham by Proposition \ref{Fact: properties of de Rahm crystalline representations}. This shows closure under tensor products. 

Finally, let $V$ be a successively de Rham representation, and choose a $\Gal(\overline{K}/K)$-stable filtration $0=F^0V\subset F^1V\subset ...\subset F^mV=V$ with de Rham graded pieces. Define $E^{m-i}V^{\vee}:=(F^iV)^{\perp}=\{\lambda\in V^{\vee}:\lambda(F^iV)=0\}$. These annihilators form a $\Gal(\overline{K}/K)$-stable filtration $0=E^0V^{\vee}\subset ...\subset E^mV^{\vee}=V^{\vee}$ of $V^{\vee}$. Since $E^{r}V^{\vee}/E^{r-1}V^{\vee}=(F^{m-r}V)^{\perp}/(F^{m-r+1}V)^{\perp}\cong (F^{m-r+1}V/F^{m-r}V)^{\vee}$, $E^{r}V^{\vee}/E^{r-1}V^{\vee}$ is de Rham by Proposition \ref{Fact: properties of de Rahm crystalline representations}. This proves closure under duals. 
\end{proof}

\section{Continuous Mal'cev $\QQ_p$-Completion of Pro-$p$ Groups}

In this appendix, we briefly review the definition and several properties of the continuous Mal'cev $\QQ_p$-completion of pro-$p$ groups. Our main reference is \cite{Hu-Wang-continuou-Malcev-completion}.

\begin{definition}[\cite{Betts-thesis}*{Definition-Lemma 2.3.1}]\label{Fact: definition of continuous Malcev completion}
Let $G$ be a topological group, and let $k$ be a topological field. Consider the functor from pro-unipotent groups over $k$ to sets defined by $F(U)=\homo_{\cont}(G,U(k))$, which consists of continuous homomorphisms $G\rightarrow U(k)$. This functor is represented by a pro-unipotent group $G\widehat{\otimes} k$, which is called the \textbf{continuous Mal'cev $k$-completion} of $G$.
\end{definition}

\begin{proposition}[\cite{Hu-Wang-continuou-Malcev-completion}*{Proposition 1.2}]\label{Prop: Malcev completion is the inverse limit of Malcev completion of lower central series}
Retain the same assumptions of Definition \ref{Fact: definition of continuous Malcev completion}.
The natural morphism
$G\widehat{\otimes} k\rightarrow \varprojlim (G/\overline{\Gamma^n G})\widehat{\otimes} k$ is an isomorphism of pro-unipotent groups over $k$.
\end{proposition}

Continuous Mal'cev completion is compatible with extension of the base field under a condition that controls the relevant field topologies.

\begin{definition}[\cite{Warner-topological-rings}*{Theorem 15.5}]\label{Fact: properties of straight topological field}
Let $k$ be a Hausdorff topological field. The field $k$ is \textbf{straight} if every linear functional on a Hausdorff topological $k$-vector space with closed kernel is continuous.
\end{definition}

\begin{example}[\cite{Warner-topological-rings}*{Theorem 13.8, Theorem 14.12}]\label{Example: Valution fields are straight}
Every valuation field, equipped with its valuation topology, is straight. For example, $\QQ_p$ is straight. \qed
\end{example}

\begin{definition}[\cite{Hu-Wang-continuou-Malcev-completion}*{Definitions A.11, A.13}]\label{Def: weakly cartesian subfield}
Let $k$ be a closed subfield of a Hausdorff topological field $K$. We call $k\subset K$ a \textbf{weakly cartesian pair} if $K$, regarded as a topological vector space over $k$, satisfies the following conditions:
\begin{itemize}[leftmargin=0.25in]
    \item every finite-dimensional subspace is linearly homeomorphic to the product topological space $k^r$ for some $r$;
    \item every finite-dimensional subspace is closed.
\end{itemize}
In this case, $k$ is called a \textbf{weakly cartesian subfield} of $K$.
\end{definition}

\begin{example}[\cite{Hu-Wang-continuou-Malcev-completion}*{Theorem A.10, Example A.14}]\label{Example: weakly cartesian pair}
$\QQ_p$ is a weakly cartesian subfield of any closed subfield of the completion $\CC_p$ of the algebraic closure of $\QQ_p$. 
\end{example}

\begin{proposition}[\cite{Hu-Wang-continuou-Malcev-completion}*{Proposition 1.6}]\label{Prop: base change of Malcev completion}
Let $G$ be a topologically finitely generated topological group. Let $k\subset K$ be a weakly cartesian pair of Hausdorff topological fields. Assume that $k$ is straight. Then there is a canonical isomorphism $G\widehat{\otimes}K \cong (G\widehat{\otimes} k)_K$, where the right-hand side denotes base change from $k$ to $K$.
\end{proposition}

For nilpotent pro-$p$ groups, continuous Mal'cev $\QQ_p$-completion has the following useful properties, which we use throughout this manuscript.

\begin{lemma}[\cite{Hu-Wang-continuou-Malcev-completion}*{Lemma 4.13}]\label{Lem: Malcev completion preserves exactness of nilpotent groups}
Let $1\rightarrow G_1\rightarrow G_2\rightarrow G_3\rightarrow 1$ be a short exact sequence of topologically finitely generated nilpotent pro-$p$ groups and continuous homomorphisms. Assume that $\QQ_p\subset k$ is a weakly cartesian pair of Hausdorff topological fields. Then the induced sequence $1\rightarrow G_1\widehat{\otimes}k\rightarrow G_2\widehat{\otimes}k\rightarrow G_3\widehat{\otimes}k\rightarrow 1$ is exact.
\end{lemma}

\begin{proposition}[\cite{Hu-Wang-continuou-Malcev-completion}*{Proposition 4.14}]\label{Prop: lower central series of Lie algebra and malcev completion of lower central series}
Let $G$ be a topologically finitely generated pro-$p$ group, and assume that $\QQ_p\subset k$ is a weakly cartesian pair of Hausdorff topological fields. Then, for every $i$,
the Lie algebra $\Lie(G\widehat{\otimes}k)/\Gamma^i\Lie(G\widehat{\otimes}k)$ is naturally isomorphic to $\Lie((G/\Gamma^iG)\widehat{\otimes}k)$.
\end{proposition}

\bibliographystyle{amsalpha}
\bibliography{ref}

\Addresses

\end{document}